\documentclass[a4paper,USenglish, cleveref, autoref, nolineno]{socg-lipics-v2019}

\makeatletter
\let\cite\relax
\DeclareRobustCommand{\cite}{%
  \let\new@cite@pre\@gobble
  \@ifnextchar[\new@cite{\@citex[]}}
\def\new@cite[#1]{\@ifnextchar[{\new@citea{#1}}{\@citex[#1]}}
\def\new@citea#1{\def\new@cite@pre{#1}\@citex}
\def\@cite#1#2{[{\new@cite@pre\space#1\if\relax\detokenize{#2}\relax\else, #2\fi}]}
\makeatother
\usepackage[all]{xy}

\title{Homotopy Reconstruction via the Cech Complex and the Vietoris-Rips Complex} 

\titlerunning{Homotopy Reconstruction via Cech Complex and Vietoris-Rips Complex}

\author{Jisu Kim}{Inria Saclay -- \^{I}le-de-France, Palaiseau, France, \url{http://pages.saclay.inria.fr/jisu.kim/} }{jisu.kim@inria.fr}{}{Partially supported by Samsung Scholarship}

\author{Jaehyeok Shin}{Department of Statistics \& Data Science, Carnegie Mellon University, Pittsburgh, USA \url{http://www.stat.cmu.edu/~jaehyeos/}}{shinjaehyeok@cmu.edu}{}{}

\author{Fr\'{e}d\'{e}ric Chazal}{Inria Saclay -- \^{I}le-de-France, Palaiseau, France, \url{https://geometrica.saclay.inria.fr/team/Fred.Chazal/}}{frederic.chazal@inria.fr}{}{}

\author{Alessandro Rinaldo}{Department of Statistics \& Data Science, Carnegie Mellon University, Pittsburgh, USA, \url{http://www.stat.cmu.edu/~arinaldo/}}{arinaldo@cmu.edu}{}{}

\author{Larry Wasserman}{\{Department of Statistics \& Data Science, Machine Learning Department\}, Carnegie Mellon University, Pittsburgh, USA, \url{http://www.stat.cmu.edu/~larry/}}{larry@stat.cmu.edu}{}{}

\authorrunning{J. Kim, J. Shin, F. Chazal, A. Rinaldo, and L. Wasserman}

\Copyright{Jisu Kim, Jaehyeok Shin, Fr\'{e}d\'{e}ric Chazal, Alessandro Rinaldo, and Larry Wasserman}



\category{}

\relatedversion{The full version of the paper is available at \url{https://arxiv.org/abs/1903.06955} and \url{https://hal.archives-ouvertes.fr/hal-02425686}}

\supplement{The code is available at \url{https://github.com/jisuk1/nerveshape}}

\acknowledgements{We want to thank Andr{\'e} Lieutier and Henry Adams for the thoughtful discussions and comments.}

\nolinenumbers 

\hideLIPIcs  

\EventEditors{Sergio Cabello and Danny Z. Chen}
\EventNoEds{2}
\EventLongTitle{36th International Symposium on Computational Geometry (SoCG 2020)}
\EventShortTitle{SoCG 2020}
\EventAcronym{SoCG}
\EventYear{2020}
\EventDate{June 23--26, 2020}
\EventLocation{Z\"{u}rich, Switzerland}
\EventLogo{socg-logo}
\SeriesVolume{164}
\ArticleNo{54}

\begin{document}

\maketitle


\begin{abstract}
	We derive conditions under which the reconstruction of a target space is topologically correct via the \v{C}ech complex or the Vietoris-Rips complex obtained from possibly noisy point cloud data. We provide two novel theoretical results. First, we describe sufficient conditions under which any non-empty intersection of finitely many Euclidean balls intersected with a positive reach set is contractible, so that the Nerve theorem applies for the restricted \v{C}ech complex. Second, we demonstrate the homotopy equivalence of a positive $\mu$-reach set and its offsets. Applying these results to the restricted \v{C}ech complex and using the interleaving relations with the \v{C}ech complex (or the Vietoris-Rips complex), we formulate conditions guaranteeing that the target space is homotopy equivalent to the \v{C}ech complex (or the Vietoris-Rips complex), in terms of the $\mu$-reach. Our results sharpen existing results.
\end{abstract}

\section{Introduction}

A fundamental task in topological data analysis, geometric inference, and computational geometry is that of estimating the topology of a set $\mathbb{X} \subset \mathbb{R}^d$ based on a finite collection of data points $\mathcal{X}$ that lie in it or in its proximity. This problem naturally occurs in many applications area, such as cosmology \cite{StarckMDLQS2005}, time series data \cite{RobinsMB2000}, machine learning \cite{Dey2006}, and so on.

A natural way to approximate the target space  is to consider an $r$-offset of the data points, that is, to take the union of the open balls of radius $r>0$ centered at the data points. Under appropriate conditions, by the Nerve theorem~\cite{Alexandroff1928} this offset is topologically equivalent to the target space $\mathbb{X}$ via the \v{C}ech complex \cite{Bjorner1996, Hatcher2002}. For computational reasons, the Alpha shape complex may be used instead, which is homotopy equivalent to the \v{C}ech complex \cite{EdelsbrunnerM1994}. To further speed up calculations, and in particular if the data are high dimensional, the Vietoris-Rips complex may be preferable as only the pairwise distances between the data points are used.

To guarantee that the topological approximation based on the data points recovers correctly the homotopy type of $\mathbb{X}$, it is necessary that the data points are dense and close to the target space, and that the radius parameter used for constructing the \v{C}ech complex or the Vietoris-Rips complex be of appropriate size.

The conditions require the offset $r$ to be lower bounded by a constant times the Hausdorff distance between the target space and the data points, and upper bounded by another constant times a measure of the size of the topological features of the target space. Originally, the topological feature size was described as a sufficiently small number, for the Vietoris-Rips complex in \cite{Hausmann1995,Latschev2001}. Then, the topological feature size was expressed in terms of the reach of $\mathbb{X}$: see, for the \v{C}ech complex, in  \cite{ChazalL2008,NiyogiSW2008}. Subsequently, the notion of $\mu$-reach was put forward to allow for more general target spaces: the condition for the \v{C}ech complex is studied in \cite{AttaliLS2013, ChazalCL2009}, and the condition for the Vietoris-Rips complex is studied in \cite{AttaliLS2013}. Also, the radii parameters are allowed to vary across the data points in \cite{ChazalL2008}. For the case when the target space equals the data points, the conditions for the \v{C}ech complex or the Vietoris-Rips complex is studied in \cite{AdamaszekAF2018,AdamsM2019}. When the offset $r$ is beyond the topological feature size so that the homotopy equivalence does not hold, the homotopy type of the Vietoris-Rips complex was studied for the circle in \cite{AdamaszekA2017}.

In this paper, we derive  conditions under which the homotopy type of the target space is correctly recovered via the \v{C}ech complex or the Vietoris-Rips complex, in terms of the Hausdorff distance and the $\mu$-reach of the target space. To tackle this problem, we provide two novel theoretical results. First, we describe sufficient conditions under which any non-empty intersection of finitely many Euclidean balls intersected with a set of positive reach  is contractible, so that the Nerve theorem applies for the restricted \v{C}ech complex. Second, we demonstrate the homotopy equivalence of a positive $\mu$-reach set and its offsets. These results are new and of independent interest. 

Overall, our new bounds offer significant improvements over the previous results in \cite{NiyogiSW2008,AttaliLS2013} and are sharp: in particular, they achieve the optimal upper bound for the parameter of the \v{C}ech complex and the Vietoris-Rips complex under a positive reach condition. We will provide a detailed comparison of our results with existing ones in Section \ref{sec:conclusion}.

\section{Background}
\label{sec:background}

This section provides a brief introduction to simplicial complex, Nerve theorem, reach, and $\mu$-reach. We refer 
 to Appendix A
 and \cite{Hatcher2002, EdelsbrunnerH2010, Federer1959, AamariKCMRW2019, ChazalCL2009, ChazalO2008, Lee2013, Eckhoff1993} for further definitions and details.
Throughout the paper, we let $\mathbb{X}$ and $\mathcal{X}$ be subsets of $\mathbb{R}^{d}$. For $x,y\in\mathbb{R}^{d}$, we let $d(x, y) : = \| x- y \|$ be the Euclidean distance with $\| \cdot \|$ being the Euclidean norm. Let $d(x,\mathbb{X})=\inf_{y\in \mathbb{X}}d(x,y)$ denotes the distance from a point $x$ to a set $\mathbb{X}$, and let $d_{\mathbb{X}}:\mathbb{R}^{d}\to\mathbb{R}$ be the distance function $x\mapsto d(x,\mathbb{X})$. For $r>0$, we let $\mathbb{B}_{\mathbb{X}}(x,r) := \left\{y \in \mathbb{X} : d(x, y) < r \right\}$ be the open restricted ball centered at $x \in \mathbb{R}^d$ of radius $r>0$. For $r>0$, we let $\mathbb{X}^{r}$ be an $r$-offset  of a set $\mathbb{X}$ defined by  the collection of all points that are within $r$ distance to $\mathbb{X}$, that is, $\mathbb{X}^{r}:=\bigcup_{x\in\mathbb{X}}\mathbb{B}_{\mathbb{R}^{d}}(x,r)$. Finally, for two sets $X,Y\subset\mathbb{R}^{d}$, we let $d_{H}(X,Y):=\inf\{r>0:X\subset Y^{r}\text{ and }Y\subset X^{r}\}$ be the Hausdorff distance between $X$ and $Y$.

\subsection{Simplicial complex and Nerve theorem}

A natural way to approximate the target space $\mathbb{X}$ with the data points $\mathcal{X}$ is to take the union of open balls centered at the data points. In detail, let $r = \{r_x, x \in \mathcal{X} \} \in \mathbb{R}^{\mathcal{X}}_{+}$ be pre-specified radii 
and consider the union of restricted balls
\begin{equation}
	\label{eq:background_unionballs}
	\bigcup_{x\in\mathcal{X}} \mathbb{B}_{\mathbb{X}}(x,r_{x}).
\end{equation}
Though we allow for the points in $\mathcal{X}$ to lie outside $\mathbb{X}$, we will assume throughout that $\mathbb{B}_{\mathbb{X}}(x,r_{x}) \neq \emptyset$ for all $x \in \mathcal{X}$.

To infer the topological properties of the union of balls in \eqref{eq:background_unionballs}, we rely on a simplicial complex, which can be seen as a high dimensional generalization of a graph. Given a set $V$, an \textit{(abstract) simplicial complex} is a collection $K$ of finite subsets of $V$ such that $\alpha\in K$ and $\beta\subset\alpha$ implies $\beta\in K$. Each set $\alpha\in K$ is called its \textit{simplex}, and each element of $\alpha$ is called a \textit{vertex} of $\alpha$.

A simplicial complex encoding the topological properties of the union of balls in \eqref{eq:background_unionballs} is the \v{C}ech complex. 

\begin{definition}[\v{C}ech complex]
	\label{def:background_cech}
	Let
	$\mathcal{X}$, $\mathbb{X}$ be two subsets
	and $r \in \mathbb{R}^{\mathcal{X}}_{+}$. The (weighted) \emph{\v{C}ech complex} $\textrm{\v{C}ech}_{\mathbb{X}}(\mathcal{X},r)$ is the simplicial complex 
	\begin{equation}
		\textrm{\v{C}ech}_{\mathbb{X}}(\mathcal{X},r) :=\left\{\sigma=\{x_{1},\dots,x_{k}\}\subset\mathcal{X} : \bigcap\limits_{j=1}^k \mathbb{B}_{\mathbb{X}}(x_{j},r_{x_{j}}) \neq \emptyset\right\}.\label{eq:background_cech}
	\end{equation}
\end{definition}

Computing the \v{C}ech complex requires computing all possible intersections of the balls. To further speed up the calculation, we only check the pairwise distances between the data points and instead build the Vietoris-Rips complex.

\begin{definition}[Vietoris-Rips complex] \label{def:background_rips}
	Let
	$\mathcal{X}$, $\mathbb{X}$ be two subsets
	and $r \in \mathbb{R}^{\mathcal{X}}_{+}$. 
	The weighted \emph{Vietoris-Rips complex} $\textrm{Rips}(\mathcal{X},r)$ is the simplicial complex defined as
	\begin{equation}
		\textrm{Rips}(\mathcal{X},r):=\left\{ \sigma\subset\mathcal{X}:d(x_{i},x_{j})<r_{x_{i}}+r_{x_{j}},\text{ for all } x_{i},x_{j}\in\sigma\right\} .\label{eq:background_rips}
	\end{equation}
\end{definition}

The ambient \v{C}ech complex in \eqref{eq:background_cech} (that is, $\mathbb{X}=\mathbb{R}^{d}$) and the Vietoris-Rips complex in \eqref{eq:background_rips} have the following interleaving relationship when all radii are equal (e.g., see Theorem 2.5 in \cite{deSilvaG2007}). That is, when $r_{x}=r > 0$ for all $x\in\mathcal{X}$, then
\begin{equation}
	\textrm{\v{C}ech}_{\mathbb{R}^{d}}(\mathcal{X},r)\subset\textrm{Rips}(\mathcal{X},r)\subset\textrm{\v{C}ech}_{\mathbb{R}^{d}}\left(\mathcal{X},\sqrt{\frac{2d}{d+1}}r\right).\label{eq:background_interleaving_rips_ambientcech}
\end{equation}
This interleaving relation is extended to the case of different radii in Lemma~\ref{lem:homotopy_interleaving_rips_ambientcech}.

The union of balls in \eqref{eq:background_unionballs} and the \v{C}ech complex in \eqref{eq:background_cech} are homotopy equivalent  under appropriate conditions. This remarkable result is precisely the renowned nerve theorem \cite{Alexandroff1928, Bjorner1996, Hatcher2002},  which we recall next. We first introduce the \emph{nerve}, which is a more abstract notion of the \v{C}ech complex.




\begin{definition}[Nerve] Let $\mathcal{U}=\{U_{\alpha}\}$ be an
	open cover of a given topological space $\mathbb{X}$. The nerve of $\mathcal{U}$,
	denoted by $\mathcal{N}(\mathcal{U})$, is the abstract simplicial complex defined
	as 
	\[
	\mathcal{N}(\mathcal{U})=\left\{ \{U_{1},\ldots,U_{k}\}\subset\mathcal{U}:\,\bigcap_{j=1}^{k}U_{j}\neq\emptyset\right\} .
	\]
	
\end{definition}

The nerve theorem prescribes conditions under which  the nerve of an open cover of $\mathbb{X}$ is homotopy equivalent to $\mathbb{X}$ itself.

\begin{theorem}[Nerve theorem]
	
	\label{thm:background_nerve}
	
	Let $\mathbb{X}$ be a paracompact space and $\mathcal{U}$ be an open cover of $\mathbb{X}$.
	If every nonempty intersection of finitely many sets in $\mathcal{U}$
	is contractible, then $\mathbb{X}$ is homotopy equivalent to the
	nerve $\mathcal{N}(\mathcal{U})$.
	
\end{theorem}


Thus, in order to conclude that the $\textrm{\v{C}ech}_{\mathbb{X}}(\mathcal{X},r)$ complex in \eqref{eq:background_cech} has the same homotopy type as $\mathbb{X}$, it is enough to show, by the nerve theorem, that the union of restricted balls $\bigcup_{x\in\mathcal{X}} \mathbb{B}_{\mathbb{X}}(x,r_{x})$ covers the target space $\mathbb{X}$ and that any arbitrary non-empty intersection of restricted balls is contractible. The difficulty in establishing the latter, more technical, condition lies in the fact that it is not clear a priori what properties of $\mathbb{X}$ will imply it. If $\mathbb{X}$ is a convex set, then the nerve theorem applies straightforwardly. But for more general spaces, such as smooth lower-dimensional manifolds, it is not obvious how contractibility may be guaranteed. 
One of the main results of this paper, given below in Theorem~\ref{thm:nerve_contractible}, asserts that if $\mathbb{X}$ has positive reach and the radii of the restricted balls are small compared to the reach, then any non-empty intersection of restricted balls is contractible.



\subsection{The reach}
\label{subsec:background_reach}
First introduced by \cite{Federer1959}, the reach is a quantity expressing the degree of geometric regularity of a set. In detail, given a closed subset $\mathbb{X}\subset\mathbb{R}^{d}$,
the medial axis of $\mathbb{X}$, denoted by ${\rm Med}(\mathbb{X})$, is the subset
of $\mathbb{R}^{d}$ consisting of all the points that have at least two
nearest neighbors in $\mathbb{X}$. Formally, 
\begin{equation}
	{\rm Med}(\mathbb{X})=\left\{ x\in\mathbb{R}^{d} \setminus \mathbb{X} \colon \text{there exist } q_{1}\neq q_{2}\in \mathbb{X},||q_{1}-x||=||q_{2}-x||=d(x,\mathbb{X})\right\} ,\label{eq:background_medialaxis}
\end{equation}
The reach of $\mathbb{X}$ is then defined as the minimal distance from $\mathbb{X}$ to ${\rm Med}(\mathbb{X})$. 
\begin{definition}
	The reach of a closed subset $\mathbb{X}\subset\mathbb{R}^{d}$ is defined
	as 
	\begin{align}
		\tau_{\mathbb{X}}=\inf_{q\in \mathbb{X}}d\left(q,{\rm Med}(\mathbb{X})\right) = \inf_{q\in \mathbb{X},x\in \mathrm{Med}(\mathbb{X})}||q-x||.\label{eq:background_reach_medial_axis}
	\end{align}
\end{definition}

Some authors \cite[see, e.g.][]{NiyogiSW2008, SingerW2012} refer to $\tau_\mathbb{X}^{-1}$ as the \emph{condition number}.
From the definition of the medial axis in \eqref{eq:background_medialaxis}, the projection $\pi_{\mathbb{X}}(x) = \arg\min_{p \in \mathbb{X}} \left\Vert p-x \right\Vert$ onto $\mathbb{X}$ is well defined (i.e. unique) outside $Med(\mathbb{X})$. 
In fact, the reach is the largest distance $\rho \geq 0$ such that $\pi_{\mathbb{X}}$ is well defined on the $\rho$-offset $\left\{ x \in \mathbb{R}^{d} : d(x,\mathbb{X}) < \rho \right\}$. 
Hence, assuming the set $\mathbb{X}$ has positive reach can be seen as a generalization or weakening of convexity, since a set $\mathbb{X} \subset \mathbb{R}^{d}$ is convex if and only if $\tau_{\mathbb{X}} = \infty$.
In the next section, we describe how to use the reach condition to ensure that the union of restricted balls is contractible, which in turn allows us to apply the Nerve theorem to recover the homotopy type of the target space $\mathbb{X}$.


For a non-smooth target space, the reach of the space can be zero. In this case, we can deploy a more general notion of feature size, called $\mu$-reach, introduced by \cite{ChazalCL2009}. For any point $x \in \mathbb{R}^d \setminus \mathbb{X}$, let $\Gamma_{\mathbb{X}}(x)$ be the set of points in $\mathbb{X}$ closest to $x$. Let $\Theta_{\mathbb{X}}(x)$ be the center of the unique smallest closed ball enclosing $\Gamma_{\mathbb{X}}(x)$.
Then, for $x \in \mathbb{R}^d \setminus \mathbb{X}$, the generalized gradient of the distance function $d_{\mathbb{X}}$ is defined as
\begin{equation}
	\nabla_{\mathbb{X}}(x) = \frac{x - \Theta_{\mathbb{X}}(x)}{d_{\mathbb{X}}(x)},	\label{eq:background_distancefunction_gradient}
\end{equation}
and set $\nabla_{\mathbb{X}}(x)=0$ for $x\in\mathbb{X}$. See Figure \ref{fig:background_reach_gradient} for a graphical illustration.
Then, for $\mu\in(0,1]$, the $\mu$-medial axis of $\mathbb{X}$ is defined as 
\begin{equation}
	{\rm Med}_{\mu}(\mathbb{X}) = \left\{x \in \mathbb{R}^d \setminus \mathbb{X}: \|\nabla_{\mathbb{X}}(x)\| < \mu  \right\},
\end{equation}
Finally, the $\mu$-reach of $\mathbb{X}$ is defined as the minimal distance from $\mathbb{X}$ to ${\rm Med}_{\mu}(\mathbb{X})$.
\begin{definition}
	The $\mu$-reach of a closed subset $\mathbb{X}\subset\mathbb{R}^{d}$ is defined
	as 
	\begin{align}
		\tau_{\mathbb{X}}^\mu=\inf_{q\in \mathbb{X}}d\left(q,{\rm Med}_{\mu}(\mathbb{X})\right) = \inf_{q\in \mathbb{X},x\in \mathrm{Med}_{\mu}(\mathbb{X})}||q-x||.\label{eq:background_mu_reach}
	\end{align}
\end{definition}
Note that if $\mu = 1$, the corresponding $\mu$-reach equals to the reach of $\mathbb{X}$.

\begin{figure}
	\centering
	\includegraphics{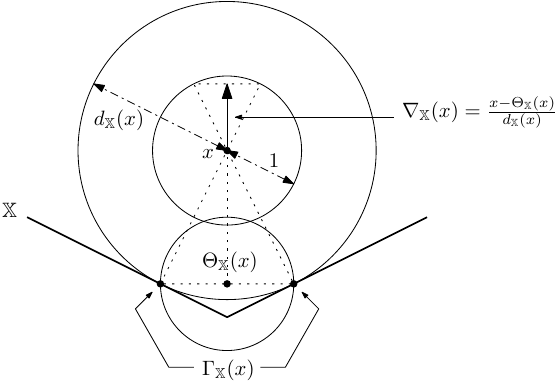}
	\caption{The graphical illustration for the generalized gradient $\nabla_{\mathbb{X}}(x)$, from \cite{ChazalCLT2007, ChazalCL2009}.}
	\label{fig:background_reach_gradient}
\end{figure}

Two offsets $\mathbb{X}^{r}$ and $\mathbb{X}^{s}$ of the target space $\mathbb{X}$ are topologically equivalent if they are free of critical points of the distance function $d_{\mathbb{X}}$ in the sense specified below (see e.g., \cite{Grove1993} or Proposition 3.4 in \cite{ChazalL2005}).
\begin{lemma}[Isotopy Lemma]
	\label{lem:background_offset_isotopy}
	Let $\mathbb{X}\subset\mathbb{R}^{d}$ be a set, and for $r,s>0$
	with $s\leq r$, let $\mathbb{X}^{r}$ and $\mathbb{X}^{s}$ be two
	offsets of $\mathbb{X}$. Suppose the distance function $d_{\mathbb{X}}$
	does not have a critical point on $\overline{\mathbb{X}^{r}}\backslash\mathbb{X}^{s}$,
	that is, $\nabla_{\mathbb{X}}(x)\neq0$ for all $x\in\overline{\mathbb{X}^{r}}\backslash\mathbb{X}^{s}$
	where $\nabla_{\mathbb{X}}$ is from \eqref{eq:background_distancefunction_gradient}. Then $\mathbb{X}^{r}$ and
	$\mathbb{X}^{s}$ are homeomorphic.
\end{lemma}
Note that requiring $\nabla_{\mathbb{X}}(x)\neq0$ for all $x\in\overline{\mathbb{X}^{r}}\backslash\mathbb{X}^{s}$ is weaker than the $\mu$-reach condition $\tau_{\mathbb{X}}^{\mu}>r$ for any $\mu\in(0,1]$. One of the main results of the paper, given in Theorem~\ref{thm:defretract_mureach}, generalizes this topological relation to the relation between the target space and its offset under a stronger positive $\mu$-reach condition.

\subsection{Restricted versus Ambient balls}

It is important to point out that the nerve theorem needs not to be applied to the \v{C}ech complex built using ambient, as opposed, to restricted balls. In particular, the homotopy type of $\mathbb{X}$, may not be correctly recovered using unions of ambient balls even if the point cloud is dense in $\mathbb{X}$ and the radii of the balls all vanish. We elucidate this point in the next example. Below, $\mathbb{B}_{\mathbb{R}^d}(x,r)$ denotes the open ambient ball in $\mathbb{R}^d$ centered at $x$ and of radius $r>0$.



\begin{example}
	\label{ex:eg_unionballs_nonhomotopic}
	Let $\mathbb{X}=(\partial\mathbb{B}_{\mathbb{R}^{2}}(0,1))\cap\{x\in\mathbb{R}^{2}:\,x_{2}\geq0\}$
	be a semicircle in $\mathbb{R}^{2}$. Let $\epsilon\in(0,1)$ be fixed,
	and $x_{1}$, $x_{2}$ be points on $\mathbb{X}$ satisfying $\left\Vert x_{1}-x_{2}\right\Vert \in\left(\epsilon\sqrt{4-\epsilon^{2}},2\epsilon\right)$.
	Then, $\mathbb{B}_{\mathbb{R}^{2}}(x_{1},\epsilon)\cap\mathbb{B}_{\mathbb{R}^{2}}(x_{2},\epsilon)$
	is nonempty but has an empty intersection with $\mathbb{X}$. Now, choose
	$\rho<d(\mathbb{X},\,\mathbb{B}_{\mathbb{R}^{2}}(x_{1},\epsilon)\cap\mathbb{B}_{\mathbb{R}^{2}}(x_{2},\epsilon))$ and choose $\mathcal{X}_{0}\subset\mathbb{X}$
	be dense enough so that $\bigcup_{x\in\mathcal{X}_{0}}\mathbb{B}_{\mathbb{R}^{2}}(x,\rho)$
	covers $\mathbb{X}$. Now, consider the union of ambient balls 
	\begin{equation}
		\left(\mathbb{B}_{\mathbb{R}^{2}}(x_{1},\epsilon)\bigcup\mathbb{B}_{\mathbb{R}^{2}}(x_{2},\epsilon)\right)\bigcup\left(\bigcup_{x\in\mathcal{X}_{0}}\mathbb{B}_{\mathbb{R}^{2}}(x,\rho)\right).\label{eq:eg_unionballs_nonhomotopic}
	\end{equation}
	Then from the fact $\rho<d\left(\mathbb{X},\,\mathbb{B}_{\mathbb{R}^{2}}(x_{1},\epsilon)\cap\mathbb{B}_{\mathbb{R}^{2}}(x_{2},\epsilon)\right)$
	and $\bigcup_{x\in\mathcal{X}_{0}}\mathbb{B}_{\mathbb{R}^{2}}(x,\rho)$ is a
	covering of $\mathbb{X}$, we have that the union of balls in \eqref{eq:eg_unionballs_nonhomotopic}
	is homotopy equivalent to a circle, hence its homotopy is different from the
	semicircle $\mathbb{X}$. Note that the above construction holds for all choices of $\epsilon \in (0,1)$. Since $\rho\to 0$ as $\epsilon\to0$, $\mathcal{X}_{0}$ can be arbitrary dense in $\mathbb{X}$. See \Cref{fig:eg_unionballs_nonhomotopic}.

	\begin{figure}
		\centering
		\includegraphics[width=0.7\linewidth]{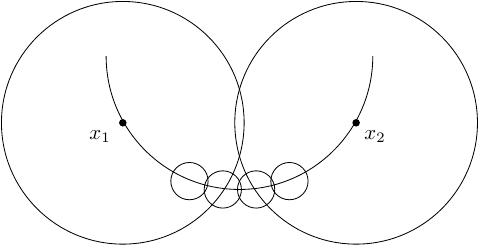}
		\caption{\small \emph{An example in which the union of balls is different from the underlying space in terms of the homotopy. In the figure, the union of balls deformation retracts to a circle, hence its homotopy is different from the underlying semicircle.}}
		\label{fig:eg_unionballs_nonhomotopic}
	\end{figure}	
\end{example}


\section{The nerve theorem for Euclidean sets of positive reach}
\label{sec:nerve}

In order to apply the nerve theorem to the \v{C}ech complex built on restricted balls, it is enough to check whether any finite intersection of the restricted balls $\bigcap_{j=1}^{k}\mathbb{B}_{\mathbb{X}}(x_{j},r_{x_{j}})$ is contractible
(since $\mathbb{X}$ is a subset of $\mathbb{R}^d$ and is endowed with the subspace topology, it is paracompact.).


Theorem~\ref{thm:nerve_contractible} is one of the main statements of this paper and shows that, if a subset $\mathbb{X}\subset\mathbb{R}^{d}$ has a positive reach $\tau >0$, 
any non-empty intersection of restricted balls is contractible if the radii are small enough compared to $\tau$.




\begin{theorem}
	\label{thm:nerve_contractible}	
	Let $\mathbb{X} \subset\mathbb{R}^{d}$ be a subset with reach $\tau>0$ and let $\mathcal{X}\subset\mathbb{R}^{d}$ be a set of points. Let $\{r_x > 0: x \in \mathcal{X}\}$ be a set of radii indexed by $x\in\mathcal{X}$. Then, if $r_x \leq \sqrt{\tau^{2}+(\tau-d_{\mathbb{X}}(x))^{2}}$ for all $x\in\mathcal{X}$, any nonempty intersection of restricted balls $\bigcap_{x\in I} \mathbb{B}_{\mathbb{X}}(x,r_{x})$ for $I\subset\mathcal{X}$ is contractible.

\end{theorem}

Therefore, by combining Theorem~\ref{thm:nerve_contractible} and the Nerve Theorem (Theorem~\ref{thm:background_nerve}), we can establish that the topology of the subspace $\mathbb{X}$ can be recovered by the corresponding restricted \v{C}ech complex $\textrm{\v{C}ech}_{\mathbb{X}}(\mathcal{X},r)$, provided the radii of the balls are not too large with respect to the reach. This result is summarized in the following corollary.


%
%

\begin{corollary}[Nerve Theorem on the restricted balls]
	\label{cor:nerve_contractible_covering}
	Under the same condition of Theorem~\ref{thm:nerve_contractible}, suppose $r_x \leq \sqrt{\tau^{2}+(\tau-d_{\mathbb{X}}(x))^{2}}$ for all $x\in\mathcal{X}$, then the union of restricted balls  $\bigcup_{x\in\mathcal{X}}\mathbb{B}_{\mathbb{X}}(x,r_{x})$ is homotopy equivalent to the restricted \v{C}ech complex $\textrm{\v{C}ech}_{\mathbb{X}}(\mathcal{X},r)$. If, in addition, the union of restricted balls covers the target space $\mathbb{X}$, that is, 
	\begin{equation}
	\mathbb{X} \subset \bigcup_{x\in\mathcal{X}} \mathbb{B}_{\mathbb{X}}(x,r_{x}),\label{eq:nerve_covering}
	\end{equation}
	then $\mathbb{X}$ is homotopy equivalent to the restricted \v{C}ech complex $\textrm{\v{C}ech}_{\mathbb{X}}(\mathcal{X},r)$.
	
	
\end{corollary}

The reach condition $r_{x}\leq\sqrt{\tau^{2}+(\tau-d_{\mathbb{X}}(x))^{2}}$ is tight as the following example shows.

\begin{example}

	Let $\mathbb{X}$ be the unit Euclidean sphere in $\mathbb{R}^{d}$,
and fix $\epsilon>0$. Let $x_{1}:=(1-\epsilon,0,\ldots,0),x_{2}:=(-1+\epsilon,0,\ldots,0)\in\mathbb{R}^{d}$
, and set $\mathcal{X}:=\{x_{1},x_{2}\}$. For a unit Euclidean sphere, the
reach is equal to its radius $1$. Therefore, if $r=(r_{1},r_{2})\in\left(0,\sqrt{1+(1-\epsilon)^{2}}\right]^{2}$ then 
$\mathbb{B}_{\mathbb{X}}(x_{1},r_{1})\bigcup\mathbb{B}_{\mathbb{X}}(x_{2},r_{2})$
is homotopy equivalent to $\textrm{\v{C}ech}_{\mathbb{X}}\left(\mathcal{X},r\right)$
by Corollary \ref{cor:nerve_contractible_covering}. However, if $r_{1},r_{2}>\sqrt{1+(1-\epsilon)^{2}}$,
$\mathbb{B}_{\mathbb{X}}(x_{1},r_{1})\bigcup\mathbb{B}_{\mathbb{X}}(x_{2},r_{2})\simeq\mathbb{X}$
but $\textrm{\v{C}ech}_{\mathbb{X}}\left(\mathcal{X},r\right)\simeq0$.
Figure \ref{fig:eg_reach_condition} illustrates the 2-dimensional
case. 

 	\begin{figure}
 		\centering
 		\includegraphics{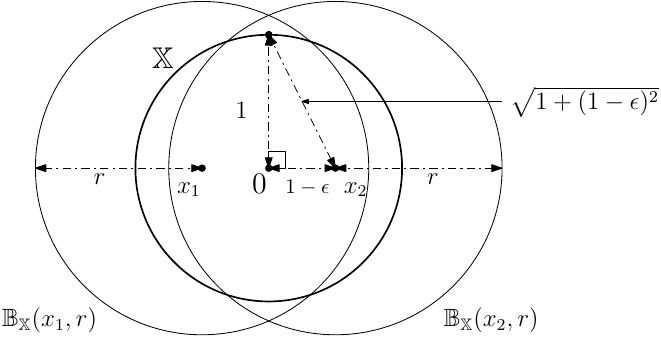}
 		\caption{\small \emph{An example in which $\mathbb{B}_{\mathbb{X}}(x_{1},r)\bigcup\mathbb{B}_{\mathbb{X}}(x_{2},r)$
        is not homotopy equivalent to $\textrm{\v{C}ech}_{\mathbb{X}}\left(\mathcal{X},r\right)$
        where $\mathbb{X}=\left\{ x\in\mathbb{R}^{2}:\|x\|_{2}=1\right\} $,
        $x_{1}=(-1+\epsilon,0)$, $x_{2}=(1-\epsilon,0)$, $\mathcal{X}=\{x_{1},x_{2}\}$, and $r>\sqrt{1+(1-\epsilon)^{2}}$,
        for any $\epsilon>0$.}}
 		\label{fig:eg_reach_condition}
 	\end{figure}
\end{example}

\section{Deformation retraction on positive $\mu$-reach}
\label{sec:defretract}

The positive reach condition is critical for the nerve theorem on the restricted \v{C}ech complex. However, it is not easily generalized to the positive $\mu$-reach condition. Instead, we find a positive reach set that approximates the positive $\mu$-reach set.
And to show their homotopy equivalence, we discover the topological relation between the positive $\mu$-reach set and its offset.

The homeomorphic relation between two offsets $\mathbb{X}^{r}$ and $\mathbb{X}^{s}$ of the target space $\mathbb{X}$ in Lemma~\ref{lem:background_offset_isotopy} does not hold in general between the target space and its offset, but a weakened topological relation holds under a stronger condition on the target space. Theorem~\ref{thm:defretract_mureach}, which is one of the main results in our paper, asserts that if the target space $\mathbb{X}$ has a positive $\mu$-reach, then the offset $\mathbb{X}^{r}$ deformation retracts to $\mathbb{X}$ when the offset size is not large, and in particular, they are homotopy equivalent.


\begin{theorem}
	\label{thm:defretract_mureach}
	
	Let $\mathbb{X}\subset\mathbb{R}^{d}$ be a subset with positive $\mu$-reach $\tau^{\mu}>0$.
	For $r\leq\tau^{\mu}$, the $r$-offset $\mathbb{X}^{r}$
	deformation retracts to $\mathbb{X}$. In particular, $\mathbb{X}$
	and $\mathbb{X}^{r}$ are homotopy equivalent.
	
\end{theorem}

The positive $\mu$-reach condition $r\leq\tau^{\mu}$ in Theorem~\ref{thm:defretract_mureach} is critical and cannot be weakened to $\nabla_{\mathbb{X}}(x)\neq0$ for all $x\in\overline{\mathbb{X}^{r}}\backslash\mathbb{X}$ as in Lemma~\ref{lem:background_offset_isotopy}. Indeed, Example \ref{ex:defretract_mureach_distancefunction} shows that the offset does not deformation retract to the target space although $\nabla_{\mathbb{X}}(x)\neq0$ for all $x\in\mathbb{R}^{d}$.

\begin{example}
	\label{ex:defretract_mureach_distancefunction}
	Let $\mathbb{X}\subset\mathbb{R}^{2}$ be a topologist's sine circle,
	that is, $\mathbb{X}=\mathbb{X}_{0}\cup\mathbb{X}_{1}\cup\mathbb{X}_{2}$,
	with $\mathbb{X}_{0}=\left\{ \left(x,\sin\frac{\pi}{x}\right)\in\mathbb{R}^{2}:x\in(0,1]\right\} $,
	$\mathbb{X}_{1}=\{0\}\times[-1,1]$, and $\mathbb{X}_{2}$ is a sufficiently
	smooth curve joining $(0,1)$ and $(1,0)$ and meets $\mathbb{X}_{0}\cup\mathbb{X}_{1}$
	only at its endpoints. See Figure \ref{fig:defretract_topologist_circle}. Then, $\tau_{\mathbb{X}}^{\mu}=0$
	for any $\mu\in(0,1]$ but $\nabla_{\mathbb{X}}$ is nonzero for all
	$x\in\mathbb{R}^{2}\backslash\mathbb{X}$. Now, $H_{1}(\mathbb{X})=0$,
	but for any sufficiently small $r>0$, $\mathbb{X}^{r}$ is homeomorphic
	to an annulus $\mathbb{B}_{\mathbb{R}^{2}}(0,2)\backslash\overline{\mathbb{B}_{\mathbb{R}^{2}}(0,1)}$
	and hence $H_{1}(\mathbb{X}^{r})=\mathbb{Z}$. Hence $\mathbb{X}^{r}$
	cannot deformation retract to $\mathbb{X}$.
	
	 \begin{figure}
		\centering
		\includegraphics[width=0.5\linewidth]{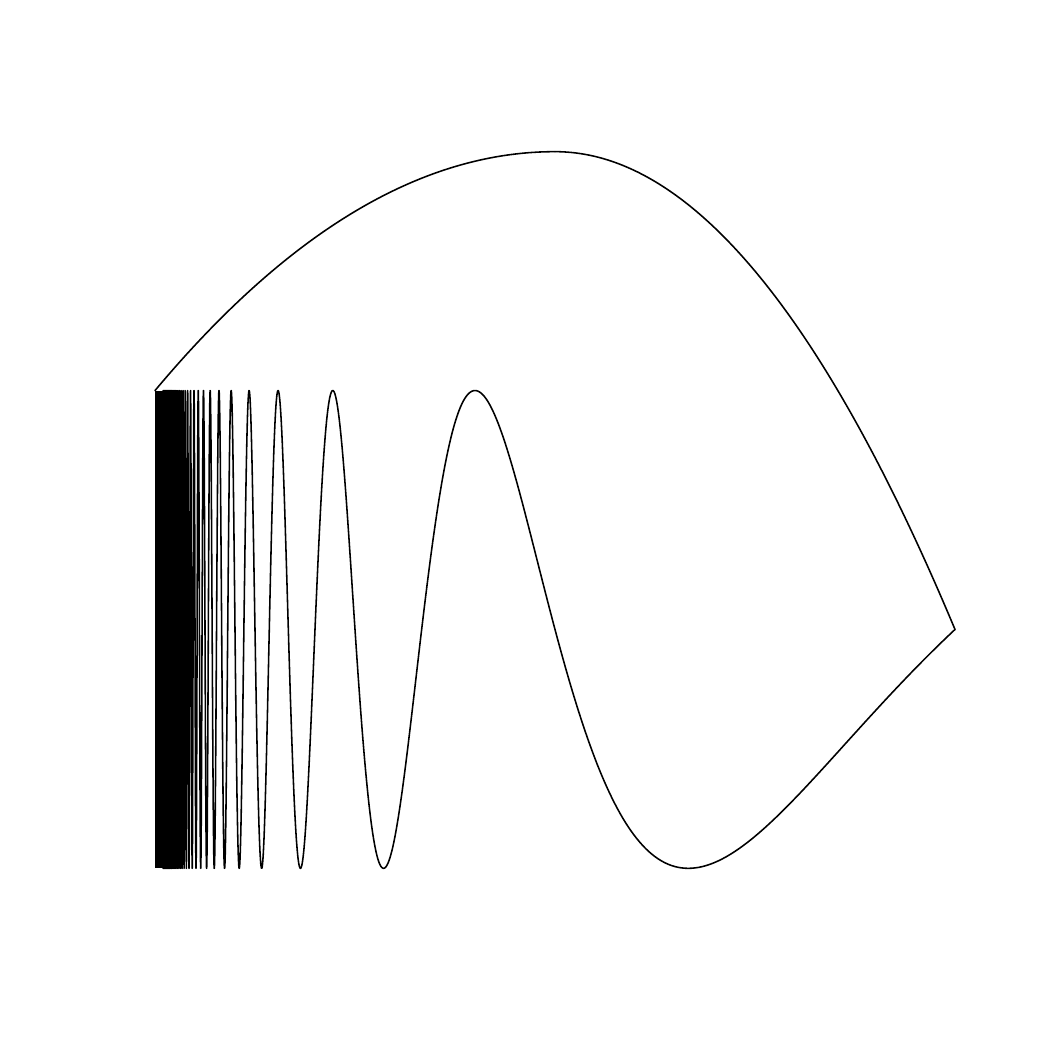}
		\caption{\small \emph{An example where $\mathbb{X}^{r}$ does not deformation retracts to $\mathbb{X}$. $\mathbb{X}$ is a topologist's sine circle, that is, $\mathbb{X}=\mathbb{X}_{0}\cup\mathbb{X}_{1}\cup\mathbb{X}_{2}$,
				with $\mathbb{X}_{0}=\left\{ \left(x,\sin\frac{\pi}{x}\right)\in\mathbb{R}^{2}:x\in(0,1]\right\} $,
				$\mathbb{X}_{1}=\{0\}\times[-1,1]$, and $\mathbb{X}_{2}$ is a sufficiently
				smooth curve joining $(0,1)$ and $(1,0)$  and meeting $\mathbb{X}_{0}\cup\mathbb{X}_{1}$
				only at its endpoints.}}
		\label{fig:defretract_topologist_circle}
	\end{figure}
\end{example}

Using Theorem~\ref{thm:defretract_mureach}, we find a positive reach set that approximates the positive $\mu$-reach set. 
The set we will use 
is the double offset \cite{ChazalCLT2007}. Recall that, for $r>0$, an $r$-offset $\mathbb{X}^{r}$ of a set $\mathbb{X}$ is the collection of all points that are within $r$ distance to $\mathbb{X}$, that is, $\mathbb{X}^{r}:=\bigcup_{x\in\mathbb{X}}\mathbb{B}_{\mathbb{R}^{d}}(x,r)$. The double offset is to take offset, take complement, take offset, and take complement, that is, for $s\geq t>0$, $\mathbb{X}^{s,t}:=(((\mathbb{X}^{s})^{\complement})^{t})^{\complement}$. Roughly speaking, it is to inflate your set first, and then deflate your set, so that sharp corners become smooth. See \cite{ChazalCLT2007} for more details.
To set up the homotopy equivalence of the positive $\mu$-reach set and its double offset, we need another tool for finding the homotopy equivalence of the complement set. This is done in the next lemma. 
\begin{lemma}
	
	\label{lem:defretract_complement}
	
	Let $\mathbb{X}\subset\mathbb{R}^{d}$ be a subset with positive reach $\tau>0$. For $r\leq\tau$,
	$\mathbb{X}^{\complement}$ deformation retracts to $(\mathbb{X}^{r})^{\complement}$.
	In particular, $\mathbb{X}^{\complement}$ and $(\mathbb{X}^{r})^{\complement}$
	are homotopy equivalent.
	
\end{lemma}

Now, combining Theorem~\ref{thm:defretract_mureach} and Lemma~\ref{lem:defretract_complement} gives the desired homotopy equivalence between the target set of positive $\mu$-reach and its double offset, where the double offset has a positive reach.

\begin{corollary}

\label{cor:defretract_doubleoffset}

Let $\mathbb{X}\subset\mathbb{R}^{d}$ be a subset with positive $\mu$-reach $\tau^{\mu}>0$.
For $s,t>0$ with $t\leq s$, let $\mathbb{X}^{s,t}:=(((\mathbb{X}^{s})^{\complement})^{t})^{\complement}$
be the double offset of $\mathbb{X}$. If $s < \tau^{\mu}$ and $t < \mu s$, then $\mathbb{X}^{s,t}$ and $\mathbb{X}$ are homotopy equivalent, and  the reach of $\mathbb{X}^{s,t}$ is greater than or equal to $t$, that is, $\tau_{\mathbb{X}^{s,t}} \geq t$. 
\end{corollary}

\section{Homotopy Reconstruction via Cech complex and Vietoris-Rips complex}
\label{sec:homotopy}

Next, we derive conditions under which the homotopy type of the target space is correctly recovered via the \v{C}ech complex and the Vietoris-Rips complex.
We first extend the interleaving relationship of the ambient \v{C}ech complex and the Vietoris-Rips complex in \eqref{eq:background_interleaving_rips_ambientcech} to the different radii case in Lemma~\ref{lem:homotopy_interleaving_rips_ambientcech}.


\begin{lemma}
	
	\label{lem:homotopy_interleaving_rips_ambientcech}
	
	Let $\mathcal{X}\subset\mathbb{R}^{d}$ be a set of points and $r=\{r_{x}>0:x\in\mathcal{X}\}$
	be a set of radii indexed by $x\in\mathcal{X}$. Then, 
	\[
	\textrm{\v{C}ech}_{\mathbb{R}^{d}}(\mathcal{X},r)\subset\textrm{Rips}(\mathcal{X},r)\subset\textrm{\v{C}ech}_{\mathbb{R}^{d}}\left(\mathcal{X},\sqrt{\frac{2d}{d+1}}r\right).
	\]
	
\end{lemma}

To recover the homotopy of the target set via the ambient \v{C}ech complex and the Vietoris-Rips complex, we utilize the restricted \v{C}ech complex. Hence, we set up the interleaving relationship between the restricted \v{C}ech complex and the ambient \v{C}ech complex in Lemma~\ref{lem:homotopy_interleaving_ambientcech_restrictedcech} and between the restricted \v{C}ech complex and the Vietoris-Rips complex in Corollary \ref{cor:homotopy_interleaving_rips_restrictedcech}.


\begin{lemma}
	\label{lem:homotopy_interleaving_ambientcech_restrictedcech}	
	Let $\mathbb{X}\subset\mathbb{R}^{d}$ be a subset with reach $\tau>0$
    and let $\mathcal{X}\subset\mathbb{R}^{d}$ be a set of points. Let
    $r=\{r_{x}>0:x\in\mathcal{X}\}$ be a set of radii indexed by $x\in\mathcal{X}$.
    Then, 
    \[
    \textrm{\v{C}ech}_{\mathbb{X}}(\mathcal{X},r)\subset\textrm{\v{C}ech}_{\mathbb{R}^{d}}(\mathcal{X},r)\subset\textrm{\v{C}ech}_{\mathbb{X}}(\mathcal{X},r'),
    \]
    where $r'=\{r'_{x}>0:x\in\mathcal{X}\}$ with
    \[
    r'_{x}=\sqrt{\frac{2\tau\left(r_{x}^{2}+d_{\mathbb{X}}(x)\left(2\tau-d_{\mathbb{X}}(x)\right)\right)}{\tau+\sqrt{\tau^{2}-\left(r_{x}^{2}+d_{\mathbb{X}}(x)\left(2\tau-d_{\mathbb{X}}(x)\right)\right)}}-d_{\mathbb{X}}(x)\left(2\tau-d_{\mathbb{X}}(x)\right)}.
    \]
    Equivalently, 
    \[
    \textrm{\v{C}ech}_{\mathbb{R}^{d}}(\mathcal{X},r'')\subset\textrm{\v{C}ech}_{\mathbb{X}}(\mathcal{X},r)\subset\textrm{\v{C}ech}_{\mathbb{R}^{d}}(\mathcal{X},r),
    \]
    where $r''=\{r''_{x}>0:x\in\mathcal{X}\}$ with 
    \[
    r''_{x}=\sqrt{\tau^{2}-d_{\mathbb{X}}(x)(2\tau-d_{\mathbb{X}}(x))-\frac{(2\tau^{2}-r_{x}^{2}-d_{\mathbb{X}}(x)(2\tau-d_{\mathbb{X}}(x)))^{2}}{4\tau^{2}}}.
    \]
\end{lemma}
\begin{corollary}
	\label{cor:homotopy_interleaving_rips_restrictedcech}
	Let $\mathbb{X}\subset\mathbb{R}^{d}$ be a subset with reach $\tau>0$
	and let $\mathcal{X}\subset\mathbb{R}^{d}$ be a set of points. Let
	$r=\{r_{x}>0:x\in\mathcal{X}\}$ be a set of radii indexed by $x\in\mathcal{X}$.
	Then, 
    \[
	\textrm{\v{C}ech}_{\mathbb{X}}(\mathcal{X},r)\subset\textrm{Rips}(\mathcal{X},r)\subset\textrm{\v{C}ech}_{\mathbb{X}}(\mathcal{X},r'''),
	\]
	where $r'''=\{r'''_{x}>0:x\in\mathcal{X}\}$ with
	\[
	r'''_{x}=\sqrt{\frac{2\tau\left(\frac{2d}{d+1}r_{x}^{2}+d_{\mathbb{X}}(x)\left(2\tau-d_{\mathbb{X}}(x)\right)\right)}{\tau+\sqrt{\tau^{2}-\left(\frac{2d}{d+1}r_{x}^{2}+d_{\mathbb{X}}(x)\left(2\tau-d_{\mathbb{X}}(x)\right)\right)}}-d_{\mathbb{X}}(x)\left(2\tau-d_{\mathbb{X}}(x)\right)}.
	\]	
\end{corollary}

Combining Nerve Theorem on the restricted balls (Corollary \ref{cor:nerve_contractible_covering}) with the covering condition \eqref{eq:nerve_covering} and Lemma~\ref{lem:homotopy_interleaving_ambientcech_restrictedcech} or Corollary \ref{cor:homotopy_interleaving_rips_restrictedcech} gives the following commutative diagram:
\begin{equation}
\xymatrix{ & \mathbb{X}\ar[dl]\\
	\textrm{\v{C}ech}_{\mathbb{X}}(\mathcal{X},r)\ar[dr]\ar[rr] &  & \textrm{\v{C}ech}_{\mathbb{X}}(\mathcal{X},r'''')\ar[ul]\\
	& \mathcal{S}\ar[ur]
},\label{eq:homotopy_diagram}
\end{equation}
where $\mathcal{S}$ is either the ambient \v{C}ech complex $\textrm{\v{C}ech}_{\mathbb{R}^{d}}(\mathcal{X},r)$ or the Vietoris-Rips complex $\textrm{Rips}(\mathcal{X},r)$. Using this diagram, we develop the homotopy equivalence between the target space and either the ambient \v{C}ech complex or the Vietoris-Rips complex. First, Theorem~\ref{thm:homotopy_cech} asserts that when the target space of positive reach is densely covered by the data points and if they are not too far apart, the ambient \v{C}ech complex can be used to recover the homotopy type.

\begin{theorem} \label{thm:homotopy_cech}

Let $\mathbb{X}\subset\mathbb{R}^{d}$ be a subset with reach $\tau>0$
and let $\mathcal{X}\subset\mathbb{R}^{d}$ be a closed discrete set of points. Let
$\{r_{x}>0:x\in\mathcal{X}\}$ be a set of radii indexed by $x\in\mathcal{X}$
with $r_{\min}:=\min_{x\in\mathcal{X}}\{r_{x}\}$ and $r_{\max}:=\max_{x\in\mathcal{X}}\{r_{x}\}$,
and let $\epsilon:=\max\{d_{\mathbb{X}}(x):x\in\mathcal{X}\}$. Suppose
$\mathbb{X}$ is covered by the union of balls centered at $x\in\mathcal{X}$
and radius $\delta$ as 
\begin{equation}
\mathbb{X}\subset\bigcup_{x\in\mathcal{X}}\mathbb{B}_{\mathbb{R}}(x,\delta).\label{eq:homotopy_cech_covering}
\end{equation}
Suppose that the maximum radius $r_{\max}$ is bounded as 
\begin{equation}
r_{\max}\leq\tau-\epsilon.\label{eq:homotopy_cech_condition_radius}
\end{equation}
Also, suppose $\delta$ satisfies the following condition: 
\begin{align}
& \delta+\sqrt{r_{\max}^{2}-\tilde{l}^{2}+\epsilon(2\tau-\epsilon)-((\tau-\epsilon)^{2}-r_{\max}^{2}+\tilde{l}^{2}+(\tau-\epsilon_{\tilde{l}})^{2})\left(\frac{\tau}{\sqrt{\tau^{2}-\tilde{r}_{\delta,c}}}-1\right)}\nonumber \\
& \leq r_{\min},\nonumber \\
& \sqrt{\frac{d}{2(d+1)}}\frac{r_{\max}}{r_{\min}}\left(\sqrt{\tilde{r}_{b}^{2}-(2\tau^{2}-\tilde{r}_{b}^{2})\left(\frac{\tau}{\sqrt{\tau^{2}-\tilde{r}_{\delta,b}^{2}}}-1\right)}+2\delta\right)\leq r_{\min}'',\label{eq:homotopy_cech_condition_covering}
\end{align}
\begin{align*}
& \tilde{l}:=\frac{1}{2}\left(r_{\min}-\tau+\sqrt{(\tau-\epsilon)^{2}-r_{\max}^{2}}-\delta\right),\qquad\epsilon_{\tilde{l}}:=\tau-\sqrt{(\tau-\epsilon)^{2}-r_{\max}^{2}}+\tilde{l},\\
& \tilde{r}_{\delta,c}^{2}:=\min\left\{ \delta^{2}+\epsilon(2\tau-\epsilon),\frac{1}{2}(r_{\max}^{2}-\tilde{l}^{2}+\epsilon(2\tau-\epsilon)+\epsilon_{\tilde{l}}(2\tau-\epsilon_{\tilde{l}}))\right\} ,\\
& r_{\min}'':=\sqrt{\tau^{2}-\epsilon(2\tau-\epsilon)-\frac{(2\tau^{2}-r_{\min}^{2}-\epsilon(2\tau-\epsilon))^{2}}{4\tau^{2}}},\\
& \tilde{r}_{b}^{2}:=\frac{2\tau\left((r_{\min}'')^{2}+\epsilon(2\tau-\epsilon)\right)}{\tau+\sqrt{\tau^{2}-\left((r_{\min}'')^{2}+\epsilon(2\tau-\epsilon)\right)}},\qquad\tilde{r}_{\delta,b}^{2}:=\min\left\{ \delta^{2}+\epsilon(2\beta-\epsilon),\frac{1}{2}\tilde{r}_{b}^{2}\right\} .
\end{align*}
Then $\mathbb{X}$ is homotopy equivalent to the ambient \v{C}ech
complex $\textrm{\v{C}ech}_{\mathbb{R}^{d}}(\mathcal{X},r)$.

\end{theorem}

A similar approach also gives the homotopy equivalence between the target space and the Vietoris-Rips complex when the target space has positive reach.

\begin{theorem} \label{thm:homotopy_rips}

Let $\mathbb{X}\subset\mathbb{R}^{d}$ be a subset with reach $\tau>0$
and let $\mathcal{X}\subset\mathbb{R}^{d}$ be a closed discrete set of points. Let
$\{r_{x}>0:x\in\mathcal{X}\}$ be a set of radii indexed by $x\in\mathcal{X}$
with $r_{\min}:=\min_{x\in\mathcal{X}}\{r_{x}\}$ and $r_{\max}:=\max_{x\in\mathcal{X}}\{r_{x}\}$,
and let $\epsilon:=\max\{d_{\mathbb{X}}(x):x\in\mathcal{X}\}$. Suppose
$\mathbb{X}$ is covered by the union of balls centered at $x\in\mathcal{X}$
and radius $\delta$ as  
\begin{equation}
\mathbb{X}\subset\bigcup_{x\in\mathcal{X}}\mathbb{B}_{\mathbb{R}}(x,\delta).\label{eq:homotopy_rips_covering}
\end{equation}
Suppose that the maximum radius $r_{\max}$ is bounded as 
\begin{equation}
r_{\max}\leq\sqrt{\frac{d+1}{2d}}\left(\tau-\epsilon\right).\label{eq:homotopy_rips_condition_radius}
\end{equation}
Also, suppose $\delta$ satisfies the following condition: 
\begin{align}
	& \sqrt{\tilde{r}_{b}^{2}(r_{\max})-(2\tau^{2}-\tilde{r}_{b}^{2}(r_{\max}))\left(\frac{\tau}{\sqrt{\tau^{2}-\tilde{r}_{\delta,b}^{2}(r_{\max})}}-1\right)}+2\delta\leq2r_{\min},\nonumber \\
	& \sqrt{\frac{d}{2(d+1)}}\left(\sqrt{\tilde{r}_{b}^{2}(r_{\min}'')-(2\tau^{2}-\tilde{r}_{b}^{2}(r_{\min}''))\left(\frac{\tau}{\sqrt{\tau^{2}-\tilde{r}_{\delta,b}^{2}(r_{\min}'')}}-1\right)}+2\delta\right)\leq r_{\min}'',\label{eq:homotopy_rips_condition_covering}
\end{align}
where 
\begin{align*}
& r_{\min}'':=\sqrt{\tau^{2}-\epsilon(2\tau-\epsilon)-\frac{(2\tau^{2}-r_{\min}^{2}-\epsilon(2\tau-\epsilon))^{2}}{4\tau^{2}}},\\
& \tilde{r}_{b}^{2}(t):=\frac{2\tau\left(t^{2}+\epsilon(2\tau-\epsilon)\right)}{\tau+\sqrt{\tau^{2}-\left(t^{2}+\epsilon(2\tau-\epsilon)\right)}},\qquad\tilde{r}_{\delta,b}^{2}(t):=\min\left\{ \delta^{2}+\epsilon(2\tau-\epsilon),\frac{1}{2}\tilde{r}_{b}^{2}(t)\right\} .
\end{align*}
Then $\mathbb{X}$ is homotopy equivalent to the Vietoris-Rips complex
$\textrm{Rips}(\mathcal{X},r)$.
\end{theorem}

\begin{remark}
	Compared to the restricted \v{C}ech complex (Corollary \ref{cor:nerve_contractible_covering}), the covering condition in \eqref{eq:homotopy_cech_covering} or \eqref{eq:homotopy_rips_covering} is more critical for the ambient \v{C}ech complex (Theorem~\ref{thm:homotopy_cech}) or the Vietoris-Rips complex (Theorem~\ref{thm:homotopy_rips}). Although the restricted \v{C}ech complex $\textrm{\v{C}ech}_{\mathbb{X}}(\mathcal{X},r)$ is still homotopy equivalent to the union of restricted balls  $\bigcup_{x\in\mathcal{X}}\mathbb{B}_{\mathbb{X}}(x,r_{x})$ without the covering condition in \eqref{eq:nerve_covering}, such homotopy equivalence does not hold for the ambient \v{C}ech complex or the Vietoris-Rips complex. This is since the upper triangle of the diagram in \eqref{eq:homotopy_diagram} only holds under the covering condition in \eqref{eq:homotopy_cech_covering} or \eqref{eq:homotopy_rips_covering}. Furthermore, the covering condition in \eqref{eq:homotopy_cech_covering} or \eqref{eq:homotopy_rips_covering} is denser in that $\delta<r_{x}$ for all $x\in\mathcal{X}$, to construct an additional homotopy equivalence on the lower triangle of the diagram in \eqref{eq:homotopy_diagram}.
\end{remark}

The homotopy equivalences in Theorem~\ref{thm:homotopy_cech} and \ref{thm:homotopy_rips} for the positive reach case is extended to the positive $\mu$-reach case by applying Corollary \ref{cor:defretract_doubleoffset} with the double offset of the target space. Corollary \ref{cor:homotopy_cech_rips_mureach} shows that when the double offset of the target space of positive $\mu$-reach is densely covered by the data points and if they are not too far apart, either the ambient \v{C}ech complex or the Vietoris-Rips complex can be used to recover the homotopy type of $\mathbb{X}$.

\begin{corollary}
\label{cor:homotopy_cech_rips_mureach}
Let $\mathbb{X}\subset\mathbb{R}^{d}$ be a subset with positive $\mu$-reach
$\tau^{\mu}>0$ and let $\mathcal{X}\subset\mathbb{R}^{d}$ be a set
of points. Let $\{r_{x}>0:x\in\mathcal{X}\}$ be a set of radii indexed
by $x\in\mathcal{X}$ with $r_{\min}:=\min_{x\in\mathcal{X}}\{r_{x}\}$
and $r_{\max}:=\max_{x\in\mathcal{X}}\{r_{x}\}$. Let $s,t,\epsilon\geq0$
with $\frac{t}{\mu} < s < \tau^{\mu}$, and let $\mathbb{Y}:=(((\mathbb{X}^{s})^{\complement})^{t})^{\complement}$
be the double offset, with $d_{\mathbb{Y}}(x)\leq\epsilon$ for all
$x\in\mathcal{X}$. Suppose $\mathbb{Y}$ is covered by the union
of balls centered at $x\in\mathcal{X}$ and radius $\delta$ as 
\[
\mathbb{Y}\subset\bigcup_{x\in\mathcal{X}}\mathbb{B}_{\mathbb{R}}(x,\delta).
\]
\begin{enumerate}
	\item[(i)] Suppose $r_{\max}\leq t-\epsilon$, and $\delta$ satisfies the following condition: 
	\begin{align*}
	& \delta+\sqrt{r_{\max}^{2}-\tilde{l}^{2}+\epsilon(2t-\epsilon)-((t-\epsilon)^{2}-r_{\max}^{2}+\tilde{l}^{2}+(t-\epsilon_{\tilde{l}})^{2})\left(\frac{t}{\sqrt{t^{2}-\tilde{r}_{\delta,c}}}-1\right)}\\
	& \leq r_{\min},\\
	& \sqrt{\frac{d}{2(d+1)}}\frac{r_{\max}}{r_{\min}}\left(\sqrt{\tilde{r}_{b}^{2}-(2t^{2}-\tilde{r}_{b}^{2})\left(\frac{t}{\sqrt{t^{2}-\tilde{r}_{\delta,b}^{2}}}-1\right)}+2\delta\right)\leq r_{\min}'',
	\end{align*}
	where 
	\begin{align*}
	& \tilde{l}:=\frac{1}{2}\left(r_{\min}-t+\sqrt{(t-\epsilon)^{2}-r_{\max}^{2}}-\delta\right),\qquad\epsilon_{\tilde{l}}:=t-\sqrt{(t-\epsilon)^{2}-r_{\max}^{2}}+\tilde{l},\\
	& \tilde{r}_{\delta,c}^{2}:=\min\left\{ \delta^{2}+\epsilon(2t-\epsilon),\frac{1}{2}(r_{\max}^{2}-\tilde{l}^{2}+\epsilon(2t-\epsilon)+\epsilon_{\tilde{l}}(2t-\epsilon_{\tilde{l}}))\right\} ,\\
	& r_{\min}'':=\sqrt{t^{2}-\epsilon(2t-\epsilon)-\frac{(2t^{2}-r_{\min}^{2}-\epsilon(2t-\epsilon))^{2}}{4t^{2}}},\\
	& \tilde{r}_{b}^{2}:=\frac{2t\left((r_{\min}'')^{2}+\epsilon(2t-\epsilon)\right)}{t+\sqrt{t^{2}-\left((r_{\min}'')^{2}+\epsilon(2t-\epsilon)\right)}},\qquad\tilde{r}_{\delta,b}^{2}:=\min\left\{ \delta^{2}+\epsilon(2t-\epsilon),\frac{1}{2}\tilde{r}_{b}^{2}\right\} .
	\end{align*}
	Then $\mathbb{X}$ is homotopy equivalent to the ambient \v{C}ech
	complex $\textrm{\v{C}ech}_{\mathbb{R}^{d}}(\mathcal{X},r)$. 
	\item[(ii)] Suppose $r_{\max}\leq\sqrt{\frac{d+1}{2d}}\left(t-\epsilon\right)$, and $\delta$ satisfies the following condition: 
	\begin{align*}
		& \sqrt{\tilde{r}_{b}^{2}(r_{\max})-(2t^{2}-\tilde{r}_{b}^{2}(r_{\max}))\left(\frac{t}{\sqrt{t^{2}-\tilde{r}_{\delta,b}^{2}(r_{\max})}}-1\right)}+2\delta\leq2r_{\min},\\
		& \sqrt{\frac{d}{2(d+1)}}\left(\sqrt{\tilde{r}_{b}^{2}(r_{\min}'')-(2t^{2}-\tilde{r}_{b}^{2}(r_{\min}''))\left(\frac{t}{\sqrt{t^{2}-\tilde{r}_{\delta,b}^{2}(r_{\min}'')}}-1\right)}+2\delta\right)\leq r_{\min}'',
	\end{align*}
	where 
	\begin{align*}
	& r_{\min}'':=\sqrt{t^{2}-\epsilon(2t-\epsilon)-\frac{(2t^{2}-r_{\min}^{2}-\epsilon(2t-\epsilon))^{2}}{4t^{2}}},\\
	& \tilde{r}_{b}^{2}(t):=\frac{2t\left(t^{2}+\epsilon(2t-\epsilon)\right)}{t+\sqrt{t^{2}-\left(t^{2}+\epsilon(2t-\epsilon)\right)}},\qquad\tilde{r}_{\delta,b}^{2}(t):=\min\left\{ \delta^{2}+\epsilon(2t-\epsilon),\frac{1}{2}\tilde{r}_{b}^{2}(t)\right\} .
	\end{align*}
	Then $\mathbb{X}$ is homotopy equivalent to the Vietoris-Rips complex
	$\textrm{Rips}(\mathcal{X},r)$. 
\end{enumerate}
\end{corollary}

We end this section by introducing a sampling condition in which we can guarantee the covering conditions in Corollary~\ref{cor:nerve_contractible_covering} and Theorem~\ref{thm:homotopy_cech}, ~\ref{thm:homotopy_rips} are satisfied. Let $P$ be the sampling distribution on $\mathbb{X}$. We assume that there exist positive constants $a, b$ and $\epsilon_{0}$ such that, for all $x \in \mathbb{X}$, the following inequality holds:
\begin{equation} \label{eq::ab-condition}
	P(\mathbb{B}_{\mathbb{R}^{d}}(x,\epsilon)) \geq a \epsilon^{b}, \quad  \text{for all } \epsilon \in(0,\epsilon_0).    
\end{equation}
This condition on $P$ is also known as the $(a,b)$-condition or the standard condition \cite{CuevasR2004, Cuevas2009, ChazalFLMRW2018}. It is satisfied, for example, if $\mathbb{X}$ is a smooth manifold of dimension $b$ and $P$ has a density with respect to the Hausdorff measure on it bounded from below by $a$. 

Under this condition, we have the following covering lemma.
\begin{lemma}
	\label{lem:density_covering_probability}
	Let $\{X_1, \dots, X_n\}$ be an i.i.d. sample from the distribution $P$ and let $\{r_{n}=(r_{n,1},\dots,r_{n,n})\}_{n\in\mathbb{N}}$
	be a triangular array of positive numbers such that, for each $n$,
	\begin{equation}
		2\left(\frac{\log n}{an}\right)^{1/b} \leq \min_{i}r_{n,i} \leq  2\epsilon_{0}.
	\end{equation}
	Then, the probability that
	the sample is a $r_{n}$-covering of $\mathbb{X}$ is
	bounded as 
	\begin{equation}
		P\left(\mathbb{X}\subset\bigcup_{i=1}^{n}\mathbb{B}_{\mathbb{R}^{d}}(X_{i},r_{n,i})\right)\geq1-\frac{1}{2^b\log n}.\label{eq:density_covering_whp}
	\end{equation}
\end{lemma}

\subsection{Conditions for homotopy reconstruction}
\label{subsec:homotopy_condition}

In this subsection, we discuss the tightness of the conditions we have identified for guaranteeing the homotopy equivalence of the target space and the \v{C}ech complex and the Vietoris-Rips complex. We first argue that the maximum radius conditions in \eqref{eq:homotopy_cech_condition_radius} and \eqref{eq:homotopy_rips_condition_radius} are tight, as Example \ref{ex:homotopy_maximum_radius_tight} shows that the \v{C}ech complex fails to be homotopy equivalent to $\mathbb{X}$ when $r_{\max} > \tau-d_{\mathbb{X}}(x)$ and the Vietoris-Rips complex fails to be homotopy equivalent to $\mathbb{X}$ when $r_{\max} > \sqrt{\frac{d+1}{2d}}\left(\tau-d_{\mathbb{X}}(x)\right)$ and $d\leq 2$.

\begin{example}
	\label{ex:homotopy_maximum_radius_tight}
	
	Let $\epsilon\in[0,1)$ be fixed. Let $\mathbb{X}\subset\mathbb{R}^{d}$
	be the unit sphere in $\mathbb{R}^{d}$, and let $\mathcal{X}=\{x_{1},\ldots,x_{n}\}\subset(1-\epsilon)\mathbb{X}$
	be a finite set of points on the sphere centered at $0$ and of radius
	$1-\epsilon$. Suppose that for some $\delta>\epsilon$, $\mathbb{X}$
	is covered by $\delta$-balls centered at $\mathcal{X}$, that is,  $\mathbb{X}\subset\bigcup_{x\in\mathcal{X}}\mathbb{B}_{\mathbb{R}}(x,\delta)$.
	The reach of $\mathbb{X}$ equals to its radius $1$.
	
	For the ambient \v{C}ech complex, if $r\in\left(0,1-\epsilon\right]^{n}$
	and condition \eqref{eq:homotopy_cech_condition_covering} is
	satisfied, then $\mathbb{X}$ is homotopy equivalent to $\textrm{\v{C}ech}_{\mathbb{X}}\left(\mathcal{X},r\right)$
	by Theorem~\ref{thm:homotopy_cech}. Now, suppose that $r_{\min}>1-\epsilon$. Then $0\in\mathbb{B}_{\mathbb{R}^{d}}(x_{i},r_{x_{i}})$
	for all $i$, hence for any $y\in\bigcup_{i=1}^{n}\mathbb{B}_{\mathbb{R}^{d}}(x_{i},r_{x_{i}})$,
	a line segment connecting $0$ and $y$ is contained in $\bigcup_{i=1}^{n}\mathbb{B}_{\mathbb{R}^{d}}(x_{i},r_{x_{i}})$
	as well. Hence $\bigcup_{i=1}^{n}\mathbb{B}_{\mathbb{R}^{d}}(x_{i},r_{x_{i}})$
	is contractible, and then from the usual Nerve Theorem, $\textrm{\v{C}ech}_{\mathbb{\mathbb{R}}^{d}}\left(\mathcal{X},r\right)\simeq\bigcup_{i=1}^{n}\mathbb{B}_{\mathbb{R}^{d}}(x_{i},r_{x_{i}})\simeq0$.
	On the other hand, the $d-1$-th homology group of $\mathbb{X}$ is $H_{d-1}(\mathbb{X})=\mathbb{Z}$,
	so $\mathbb{X}$ and $\textrm{\v{C}ech}_{\mathbb{\mathbb{R}}^{d}}\left(\mathcal{X},r\right)$
	are not homotopy equivalent.
	
	For the Vietoris-Rips complex, if $r\in\left(0,\sqrt{\frac{d+1}{2d}}(1-\epsilon)\right]^{d+1}$
	and condition \eqref{eq:homotopy_rips_condition_covering} is
	satisfied, then $\mathbb{X}$ is homotopy equivalent to $\textrm{Rips}_{\mathbb{X}}\left(\mathcal{X},r\right)$
	by Theorem~\ref{thm:homotopy_rips}. Now, suppose each $r_{x_{i}}$ is
	equal to some $r>\sqrt{\frac{d+1}{2d}}(1-\epsilon)$, and further suppose that $d\leq2$ and $\delta<\frac{1}{2(1-\epsilon)}r_{0}-\frac{\sqrt{3}}{4}$.
	When $d=1$, then the Vietoris-Rips complex equals the ambient \v{C}ech complex,
	hence from the above argument, $\textrm{Rips}\left(\mathcal{X},r\right)=\textrm{\v{C}ech}_{\mathbb{\mathbb{R}}^{d}}\left(\mathcal{X},r\right)\simeq0$.
	When $d=2$, then $\textrm{Rips}\left(\mathcal{X},r\right)\cong\textrm{Rips}\left(\frac{1}{1-\epsilon}\mathcal{X},\frac{1}{1-\epsilon}r_{0}\right)$
	and $\frac{1}{1-\epsilon}\mathcal{X}\subset\mathbb{X}\subset\bigcup_{i=1}^{n}\mathbb{B}_{\mathbb{R}^{d}}(\frac{1}{1-\epsilon}x_{i},\delta)$
	holds. Then $\frac{1}{1-\epsilon}r_{0}-2\delta>\frac{\sqrt{3}}{2}$,
	and hence from Proposition 3.8, Corollary 4.5, Proposition 5.2 of
	\cite{AdamaszekA2017}, either $\textrm{Rips}\left(\mathcal{X},r\right)\simeq S^{2l+1}$
	for some $l\geq1$ or $\textrm{Rips}\left(\mathcal{X},r\right)\simeq\vee^{c}S^{2l}$
	for some $l\geq1$ and $c\geq0$. In either case, $H_{1}(\textrm{Rips}\left(\mathcal{X},r\right))=0$.
	However, the $d-1$-th homology group of $\mathbb{X}$ is $H_{d-1}(\mathbb{X})=\mathbb{Z}$,
	so $\mathbb{X}$ and $\textrm{Rips}\left(\mathcal{X},r\right)$ are
	not homotopy equivalent.
	
\end{example}

We then rephrase the conditions on $\epsilon:=\max\{d_{\mathbb{X}}(x):x\in\mathcal{X}\}$ and the covering radius $\delta$ in \eqref{eq:homotopy_cech_condition_covering} and  \eqref{eq:homotopy_rips_condition_covering} in terms of the Hausdorff distance $d_{H}(\mathbb{X},\mathcal{X})$. For simplicity, we consider the case when all the radii $r_{x}$'s are equal, and we denote that common value as $r$. In general, the Hausdorff distance $d_{H}(\mathbb{X},\mathcal{X})$ gives a bound for both $\epsilon$ and $\delta$, that is, $\epsilon,\delta \leq d_{H}(\mathbb{X},\mathcal{X})$. Let $\rho:=\frac{d_{H}(\mathbb{X},\mathcal{X})}{\tau}$. For the \v{C}ech complex, a sufficient condition for \eqref{eq:homotopy_cech_condition_covering} is that for some $\frac{r}{\tau}\in(0,1]$,
\begin{align}
& \rho+\sqrt{(\frac{r}{\tau})^{2}-\tilde{l}^{2}+\rho(2-\rho)-((1-\rho)^{2}-(\frac{r}{\tau})^{2}+\tilde{l}^{2}+(1-\rho_{\tilde{l}})^{2})\left(\frac{1}{\sqrt{1-\tilde{r}_{\delta,c}^{2}}}-1\right)}\leq\frac{r}{\tau},\nonumber \\
& \sqrt{\frac{d}{2(d+1)}}\left(\sqrt{\tilde{r}_{b}^{2}-(2-\tilde{r}_{b}^{2})\left(\frac{1}{\sqrt{1-\tilde{r}_{\delta,b}^{2}}}-1\right)}+2\rho\right)\leq r_{\min}'',\label{eq:homotopy_condition_cech_general}
\end{align}
where 
\begin{align*}
& \tilde{l}:=\frac{1}{2}\left(\frac{r}{\tau}-1+\sqrt{(1-\rho)^{2}-(\frac{r}{\tau})^{2}}-\rho\right),\qquad\rho_{\tilde{l}}:=1-\sqrt{(1-\rho)^{2}-(\frac{r}{\tau})^{2}}+\tilde{l},\\
& \tilde{r}_{\delta,c}^{2}:=\min\left\{ 2\rho,\frac{1}{2}((\frac{r}{\tau})^{2}-\tilde{l}^{2}+\rho(2-\rho)+\rho_{\tilde{l}}(2-\rho_{\tilde{l}}))\right\} ,\\
& r_{\min}'':=\sqrt{1-\rho(2-\rho)-\frac{(2-(\frac{r}{\tau})^{2}-\rho(2-\rho))^{2}}{4}},\\
& \tilde{r}_{b}^{2}:=\frac{2\left((r_{\min}'')^{2}+\rho(2-\rho)\right)}{1+\sqrt{1-\left((r_{\min}'')^{2}+\rho(2-\rho)\right)}},\qquad\tilde{r}_{\delta,b}^{2}:=\min\left\{ 2\rho,\frac{1}{2}\tilde{r}_{b}^{2}\right\} .
\end{align*}
And for the Vietoris-Rips complex, the sufficient condition for \eqref{eq:homotopy_rips_condition_covering} is
\begin{align}
	& \sqrt{\tilde{r}_{b}^{2}(r_{0})-(2-\tilde{r}_{b}^{2}(r_{0}))\left(\frac{1}{\sqrt{1-\tilde{r}_{\delta,b}^{2}(r_{0})}}-1\right)}+2\rho\leq2r_{0},\nonumber \\
	& \sqrt{\frac{d}{2(d+1)}}\left(\sqrt{\tilde{r}_{b}^{2}(r_{\min}'')-(2-\tilde{r}_{b}^{2}(r_{\min}''))\left(\frac{1}{\sqrt{1-\tilde{r}_{\delta,b}^{2}(r_{\min}'')}}-1\right)}+2\rho\right)\leq r_{\min}'',\label{eq:homotopy_condition_rips_general}
\end{align}
where 
\begin{align*}
& r_{0}:=\sqrt{\frac{d+1}{2d}}(1-\rho),\qquad r_{\min}'':=\sqrt{1-\rho(2-\rho)-\frac{(2-\frac{d+1}{2d}(1-\rho)^{2}-\rho(2-\rho))^{2}}{4}},\\
& \tilde{r}_{b}^{2}(t):=\frac{2\left(t^{2}+\rho(2-\rho)\right)}{1+\sqrt{1-\left(t^{2}+\rho(2-\rho)\right)}},\qquad\tilde{r}_{\delta,b}^{2}(t):=\min\left\{ 2\rho,\frac{1}{2}\tilde{r}_{b}^{2}(t)\right\} .
\end{align*}
With the aid of a computer program, we can check that \eqref{eq:homotopy_condition_cech_general} is equivalent to
$\rho \leq0.01591\cdots$,
and \eqref{eq:homotopy_condition_rips_general} is equivalent to 
$\rho \leq0.07856\cdots$.

Now, we consider two specific cases. First, we consider the noiseless case $\mathcal{X}\subset\mathbb{X}$, that is, the data points lie in the target space. For that case, $\epsilon=0$ and $\delta \leq d_{H}(\mathbb{X},\mathcal{X})$. For the \v{C}ech complex, the sufficient condition for \eqref{eq:homotopy_cech_condition_covering} is that for some $\frac{r}{\tau}\in(0,1]$, 
\begin{align}
& \rho+\sqrt{(\frac{r}{\tau})^{2}-\tilde{l}^{2}-(1-(\frac{r}{\tau})^{2}+\tilde{l}^{2}+(1-\rho_{\tilde{l}})^{2})\left(\frac{1}{\sqrt{1-\tilde{r}_{\delta,c}^{2}}}-1\right)}\leq\frac{r}{\tau},\nonumber \\
& \sqrt{\frac{d}{2(d+1)}}\left(\sqrt{\tilde{r}_{b}^{2}-(2-\tilde{r}_{b}^{2})\left(\frac{1}{\sqrt{1-\tilde{r}_{\delta,b}^{2}}}-1\right)}+2\rho\right)\leq r_{\min}'',\label{eq:homotopy_condition_cech_noiseless}
\end{align}
where 
\begin{align*}
& \tilde{l}:=\frac{1}{2}\left(\frac{r}{\tau}-1+\sqrt{1-(\frac{r}{\tau})^{2}}-\rho\right),\qquad\rho_{\tilde{l}}:=1-\sqrt{1-(\frac{r}{\tau})^{2}}+\tilde{l},\\
& \tilde{r}_{\delta,c}^{2}:=\min\left\{ \rho^{2},\frac{1}{2}((\frac{r}{\tau})^{2}-\tilde{l}^{2}+\rho_{\tilde{l}}(2-\rho_{\tilde{l}}))\right\} ,\\
& r_{\min}'':=\sqrt{1-\frac{(2-(\frac{r}{\tau})^{2})^{2}}{4}},\qquad\tilde{r}_{b}^{2}:=\frac{2(r_{\min}'')^{2}}{1+\sqrt{1-(r_{\min}'')^{2}}},\qquad\tilde{r}_{\delta,b}^{2}:=\min\left\{ \rho^{2},\frac{1}{2}\tilde{r}_{b}^{2}\right\} .
\end{align*}
For the Vietoris-Rips complex, a sufficient condition for \eqref{eq:homotopy_rips_condition_covering}
is
\begin{align}
	& \sqrt{\tilde{r}_{b}^{2}(r_{0})-(2-\tilde{r}_{b}^{2}(r_{0}))\left(\frac{1}{\sqrt{1-\tilde{r}_{\delta,b}^{2}(r_{0})}}-1\right)}+2\rho\leq2r_{0},\nonumber \\
	& \sqrt{\frac{d}{2(d+1)}}\left(\sqrt{\tilde{r}_{b}^{2}(r_{\min}'')-(2-\tilde{r}_{b}^{2}(r_{\min}''))\left(\frac{1}{\sqrt{1-\tilde{r}_{\delta,b}^{2}(r_{\min}'')}}-1\right)}+2\rho\right)\leq r_{\min}'',\label{eq:homotopy_condition_rips_noiseless}
\end{align}
where 
\begin{align*}
& r_{0}:=\sqrt{\frac{d+1}{2d}}(1-\rho),\qquad r_{\min}'':=\sqrt{1-\frac{(2-\frac{d+1}{2d}(1-\rho)^{2})^{2}}{4}},\\
& \tilde{r}_{b}^{2}(t):=\frac{2t^{2}}{1+\sqrt{1-t^{2}}},\qquad\tilde{r}_{\delta,b}^{2}(t):=\min\left\{ \rho^{2},\frac{1}{2}\tilde{r}_{b}^{2}(t)\right\} .
\end{align*}
With the aid of a computer program, we can check that \eqref{eq:homotopy_condition_cech_noiseless} is equivalent to
$\rho \leq0.02994\cdots$,
 and \eqref{eq:homotopy_condition_rips_noiseless} is equivalent to 
$\rho \leq0.1117\cdots$.

Second, we consider the asymptotic case, where we sample more and more points and $\mathcal{X}$ forms a dense cover of $\mathbb{X}$, that is,  $\sup_{y\in\mathbb{X}}\inf_{x\in\mathcal{X}}\left\Vert x-y\right\Vert \to0$. Still, we have a noisy sample distribution, that is, $\sup_{x\in\mathcal{X}}\inf_{y\in\mathbb{X}}\left\Vert x-y\right\Vert \nrightarrow0$, so the Hausdorff distance $d_{H}(\mathbb{X},\mathcal{X})$ need not go to $0$. In this case, $\delta\to0$ and $\epsilon\leq d_{H}(\mathbb{X},\mathcal{X})$. For the \v{C}ech complex, a sufficient condition for \eqref{eq:homotopy_cech_condition_covering} is that for some $\frac{r}{\tau}\in(0,1]$, 
\begin{align}
& \sqrt{(\frac{r}{\tau})^{2}-\tilde{l}^{2}+\rho(2-\rho)-((1-\rho)^{2}-(\frac{r}{\tau})^{2}+\tilde{l}^{2}+(1-\rho_{\tilde{l}})^{2})\left(\frac{1}{\sqrt{1-\tilde{r}_{\delta,c}^{2}}}-1\right)}\leq\frac{r}{\tau},\nonumber \\
& \sqrt{\frac{d}{2(d+1)}}\sqrt{\tilde{r}_{b}^{2}-(2-\tilde{r}_{b}^{2})\left(\frac{1}{\sqrt{1-\tilde{r}_{\delta,b}^{2}}}-1\right)}\leq r_{\min}'',\label{eq:homotopy_condition_cech_asymptotic}
\end{align}
where 
\begin{align*}
& \tilde{l}:=\frac{1}{2}\left(\frac{r}{\tau}-1+\sqrt{(1-\rho)^{2}-(\frac{r}{\tau})^{2}}\right),\qquad\rho_{\tilde{l}}:=1-\sqrt{(1-\rho)^{2}-(\frac{r}{\tau})^{2}}+\tilde{l},\\
& \tilde{r}_{\delta,c}^{2}:=\min\left\{ \rho(2-\rho),\frac{1}{2}((\frac{r}{\tau})^{2}-\tilde{l}^{2}+\rho(2-\rho)+\rho_{\tilde{l}}(2-\rho_{\tilde{l}}))\right\} ,\\
& r_{\min}'':=\sqrt{1-\rho(2-\rho)-\frac{(2-(\frac{r}{\tau})^{2}-\rho(2-\rho))^{2}}{4}},\\
& \tilde{r}_{b}^{2}:=\frac{2\left((r_{\min}'')^{2}+\rho(2-\rho)\right)}{1+\sqrt{1-\left((r_{\min}'')^{2}+\rho(2-\rho)\right)}},\qquad\tilde{r}_{\delta,b}^{2}:=\min\left\{ \rho(2-\rho),\frac{1}{2}\tilde{r}_{b}^{2}\right\} .
\end{align*}
And for the Vietoris-Rips complex, a sufficient condition for \eqref{eq:homotopy_rips_condition_covering}
is 
\begin{align}
	& \sqrt{\tilde{r}_{b}^{2}(r_{0})-(2-\tilde{r}_{b}^{2}(r_{0}))\left(\frac{1}{\sqrt{1-\tilde{r}_{\delta,b}^{2}(r_{0})}}-1\right)}\leq2r_{0},\nonumber \\
	& \sqrt{\frac{d}{2(d+1)}}\sqrt{\tilde{r}_{b}^{2}(r_{\min}'')-(2-\tilde{r}_{b}^{2}(r_{\min}''))\left(\frac{1}{\sqrt{1-\tilde{r}_{\delta,b}^{2}(r_{\min}'')}}-1\right)}\leq r_{\min}'',\label{eq:homotopy_condition_rips_asymptotic}
\end{align}
where
\begin{align*}
& r_{0}:=\sqrt{\frac{d+1}{2d}}(1-\rho),\qquad r_{\min}'':=\sqrt{1-\rho(2-\rho)-\frac{(2-\frac{d+1}{2d}(1-\rho)^{2}-\rho(2-\rho))^{2}}{4}},\\
& \tilde{r}_{b}^{2}(t):=\frac{2\left(t^{2}+\rho(2-\rho)\right)}{1+\sqrt{1-\left(t^{2}+\rho(2-\rho)\right)}},\qquad\tilde{r}_{\delta,b}^{2}(t):=\min\left\{ \rho(2-\rho),\frac{1}{2}\tilde{r}_{b}^{2}(t)\right\} .
\end{align*}
With the aid of a computer program, we can check that \eqref{eq:homotopy_condition_cech_asymptotic} is equivalent to 
$\rho \leq0.03440\cdots$,
 and \eqref{eq:homotopy_condition_rips_asymptotic} is equivalent to
$\rho \leq0.07712\cdots$.

\section{Discussion and Conclusion}
\label{sec:conclusion}


Above we have provided  conditions under which the ambient \v{C}ech
complex $\textrm{\v{C}ech}_{\mathbb{R}^{d}}(\mathcal{X},r)$ and the
Rips complex $\textrm{Rips}(\mathcal{X},r)$ are homotopy equivalent to the
target space $\mathbb{X}$ when the target space $\mathbb{X}$ has positive $\mu$-reach
$\tau^{\mu}$ and the data points $\mathcal{X}$ being contained in
the $\epsilon$-offset $\mathbb{X}^{\epsilon}$ of $\mathbb{X}$.
In this section, we further discuss our
results and compare them with existing ones. For the comparison purpose,
we consider the case when all the radii $r_{x}$'s are equal, and
we denote the common value as $r$.  In these settings, an analogous
homotopy equivalence between the ambient \v{C}ech complex $\textrm{\v{C}ech}_{\mathbb{R}^{d}}(\mathcal{X},r)$
and the target space $\mathbb{X}$ is presented in \cite{AttaliLS2013}
and \cite{NiyogiSW2008}.

First, we compare the upper bound for the maximum parameter value
$r$ in $\textrm{\v{C}ech}_{\mathbb{R}^{d}}(\mathcal{X},r)$ or $\textrm{Rips}(\mathcal{X},r)$.
When $\mu=1$ so that  $\tau^{\mu}=\tau$,
our result suggests that the homotopy equivalences hold when $r\leq\tau-\epsilon$
for $\textrm{\v{C}ech}_{\mathbb{R}^{d}}(\mathcal{X},r)$ and $r\leq\sqrt{\frac{d+1}{2d}}(\tau-\epsilon)$
for $\textrm{Rips}(\mathcal{X},r)$. As we have seen in Example \ref{ex:homotopy_maximum_radius_tight}, these bounds
are optimal bounds. In \cite{NiyogiSW2008}, such a bound for $\textrm{\v{C}ech}_{\mathbb{R}^{d}}(\mathcal{X},r)$
 is $\frac{(\tau+\epsilon)+\sqrt{\tau^{2}+\epsilon^{2}-6\tau\epsilon}}{2}$ (see Proposition 7.1).
Then our bound is strictly sharper than this when $\epsilon>0$ since
\[
\frac{(\tau+\epsilon)+\sqrt{\tau^{2}+\epsilon^{2}-6\tau\epsilon}}{2}<\frac{(\tau+\epsilon)+\sqrt{\tau^{2}+9\epsilon^{2}-6\tau\epsilon}}{2}=\tau-\epsilon.
\]
In \cite{AttaliLS2013}, a necessary condition for $\textrm{\v{C}ech}_{\mathbb{R}^{d}}(\mathcal{X},r)$
in Section 5.3 is $r\leq\tau-3\epsilon$, so our upper bound is strictly
better when $\epsilon>0$.

Second, we compare the condition for the maximum possible ratio of
the Hausdorff distance $d_{H}(\mathbb{X},\mathcal{X})$ and the $\mu$-reach
$\tau^{\mu}$.
For this case, as we have seen in Section \ref{subsec:homotopy_condition},
we can check that $\textrm{\v{C}ech}_{\mathbb{R}^{d}}(\mathcal{X},r)$
is homotopy equivalent to $\mathbb{X}$ when 
$\frac{d_{H}(\mathbb{X},\mathcal{X})}{\tau}\leq0.01591\cdots$.
This result is worse than $3-\sqrt{8}\approx0.1716\cdots$ in Proposition
7.1 in \cite{NiyogiSW2008} or $\frac{-3+\sqrt{22}}{13}\approx0.1300\cdots$
in Section 5.3 in \cite{AttaliLS2013}.
Again from Section \ref{subsec:homotopy_condition},
we can check that $\textrm{Rips}(\mathcal{X},r)$ is homotopy
equivalent to $\mathbb{X}$ when 
$\frac{d_{H}(\mathbb{X},\mathcal{X})}{\tau}\leq0.07856\cdots$.
This result is better than $\frac{2\sqrt{2-\sqrt{2}}-\sqrt{2}}{2+\sqrt{2}}\approx0.03412\cdots$
in Section 5.3 in \cite{AttaliLS2013}.

Then we consider two specific cases. In the noiseless case $\mathcal{X}\subset\mathbb{X}$, the data points lie in the target space.
In this case, as we have seen in Section \ref{subsec:homotopy_condition}, we can verify that $\textrm{\v{C}ech}_{\mathbb{R}^{d}}(\mathcal{X},r)$
is homotopy equivalent to $\mathbb{X}$ when
$\frac{d_{H}(\mathbb{X},\mathcal{X})}{\tau}\leq0.02994\cdots$,
and $\textrm{Rips}(\mathcal{X},r)$ is homotopy equivalent to $\mathbb{X}$
when
$\frac{d_{H}(\mathbb{X},\mathcal{X})}{\tau}\leq0.1117\cdots$. 

In the asymptotics case, as we sample more and more points from the target
space, $\mathcal{X}$ forms a dense cover on $\mathbb{X}$, that is, $\sup_{y\in\mathbb{X}}\inf_{x\in\mathcal{X}}\left\Vert x-y\right\Vert \to0$.
For this case, as we have seen in Section \ref{subsec:homotopy_condition},
we can check that $\textrm{\v{C}ech}_{\mathbb{R}^{d}}(\mathcal{X},r)$
is homotopy equivalent to $\mathbb{X}$ when 
$\frac{d_{H}(\mathbb{X},\mathcal{X})}{\tau}\leq0.03440\cdots$,
and $\textrm{Rips}(\mathcal{X},r)$ is homotopy equivalent to $\mathbb{X}$
when
$\frac{d_{H}(\mathbb{X},\mathcal{X})}{\tau}\leq0.07712\cdots$. 

Finally, we emphasize that our result also allows the radii $\{r_{x}\}_{x\in\mathcal{X}}$
to vary across the points $x\in\mathcal{X}$. Considering different radii is of practical interest if each data point has different importance. For example, one might want to use large radii on the flat and sparse region, while to use small radii on the spiky and dense region. However, there remain significant technical difficulties to allow for a different radius per each data point. As it can be seen in Figure \ref{fig:eg_unionballs_nonhomotopic}, an uneven distribution of radii might lead to nonhomotopic between the \v{C}ech complex (or the Vietoris-Rips complex) and the target space. This situation has been studied in \cite{ChazalL2008} for the union of balls under the reach condition, but not the Vietoris-Rips complex or under the $\mu$-reach case. Theorem~\ref{thm:homotopy_rips} or  Corollary \ref{cor:homotopy_cech_rips_mureach} first tackles this homotopy reconstruction problem with different radii for the Vietoris-Rips complex or under the $\mu$-reach condition.



\bibliography{myref}

\appendix

\section{More background}
\label{app:background_more}

Throughout the paper, for $x\in\mathbb{R}$, we use the notation $(x)_{+}:=\max\{x,0\}$. Also, for a subset $\mathbb{X}\subset\mathbb{R}^{d}$ and a point $x\in\mathbb{R}^{d}\backslash Med(\mathbb{X})$, let $\pi_{\mathbb{X}}:\mathbb{R}^{d}\to\mathbb{X}$ be the projection function to $\mathbb{X}$, that is, $\pi_{\mathbb{X}}(x)=\arg\min_{y\in\mathbb{X}}\left\Vert x-y \right\Vert$.

We first review some topological equivalent relations. Two spaces
$X,Y$ are \emph{homeomorphic} and write $X\cong Y$ if there exists
a bijective function $f:X\to Y$ that is continuous and its inverse
$f^{-1}:Y\to X$ is also continuous. Sometimes the homeomorphic equivalence
is too fine and we need a coarser equivalent relation. A \emph{homotopy}
between two continuous functions $f,g:X\to Y$ is a continuous function
$H:X\times[0,1]\to Y$ such that $H(x,0)=f(x)$ and $H(x,1)=g(x)$
for all $x\in X$. We say that $f$ and $g$ are \emph{homotopic}
if there exists a homotopy between them, and write $f\simeq g$. Two
spaces $X,Y$ are \emph{homotopy equivalent} if there exists $f:X\to Y$
and $g:Y\to X$ such that $g\circ f\simeq id_{X}$ and $f\circ g\simeq id_{Y}$.
Then homotopy equivalence is coarser than the homeomorphic equivalence,
in that if $X\cong Y$ then $X\simeq Y$.

A space $X$ \emph{deformation retracts} to a subset $A\subset X$
if there exists a homotopy $H:X\times[0,1]\to X$ such that for all
$x\in X$, $H(x,0)=x$ and $H(x,1)\in A$, and for all $x\in A$ and
$t\in[0,1]$ $H(x,t)=x$. A space $X$ is \emph{contractible} if $X$
is homotopy equivalent to a point $x$. We can see that if a space
$X$ deformation retracts to a point $x\in X$ then $X$ is contractible,
but the converse is not always true.


The following theorem is from \cite{Lee2013}. This theorem is used for constructing the homotopy map giving a deformation retract from $\mathbb{X}^{r}$ to $\mathbb{X}$ in Theorem~\ref{thm:defretract_mureach}.

\begin{theorem}[Fundamental theorem on flows]
	\label{thm:background_fundamental_flow}
	Let $W$ be a smooth vector field on a smooth manifold $M$. Then
	there is a unique maximal flow $\psi:\mathbb{D}\subset M\times\mathbb{R}\to M$
	with $\frac{d}{dt}\psi(x,t)=W(\psi(x,t))$. In particular, the maximal
	domain $\mathbb{D}$ is open in $M\times\mathbb{R}$.
\end{theorem}

The following Lemma 3.3 in \cite{ChazalO2008} shows that the homotopy
equivalence in the Nerve Theorem commutes well with the inclusion
maps. 
This lemma plays a critical role in showing the homotopy equivalence of the target space $\mathbb{X}$ and a simplicial complex $\mathcal{S}$ in Lemma~\ref{lem:homotopy_simplexnerve}.
For a space $\mathbb{X}$, its open cover $\mathcal{U}=\{U_{i}\}_{i\in I}$
is a good cover if its intersection is either empty or contractible,
that is, for each $\sigma\subset I$, the set $\bigcap_{i\in\sigma}U_{i}$
is either empty or contractible.

\begin{lemma}[Lemma 3.3 in \cite{ChazalO2008}]
	
	\label{lem:background_nerve_interleaving}
	
	Let $\mathbb{X}\subset\mathbb{X}'$ be two paracompact spaces, and
	let $\mathcal{U}=\{U_{i}\}_{i\in I}$, $\mathcal{U}'=\{U'_{i}\}_{i\in I'}$,
	be good covers of $\mathbb{X}$ and $\mathbb{X}'$, respectively,
	with $I\subset I'$ and $U_{i}\subset U'_{i}$ for all $i\in I$.
	Then, there exist homotopy equivalences $\mathcal{N}\mathcal{U}\to\mathbb{X}$
	and $\mathcal{N}\mathcal{U}'\to\mathbb{X}'$ that commutes with inclusions
	$\mathbb{X}\hookrightarrow\mathbb{X}'$ and $\mathcal{N}\mathcal{U}\hookrightarrow\mathcal{N}\mathcal{U}'$
	at homotopy level, that is, the following diagram commutes: 
	\[
	\xymatrix{X\ar[r]^{\imath_{\mathbb{X}\to\mathbb{X}'}}\ar[d] & \mathbb{X}'\ar[d]\\
		\mathcal{NU}\ar[r]^{\imath_{\mathcal{N}\mathcal{U}\to\mathcal{N\mathcal{U}}'}}\ar[u] & \mathcal{NU}'\ar[u]
	}
	.
	\]
	
\end{lemma}

\subsection{The reach and the distance function}
\label{app:background_more_reach}

Sometimes we are interested in the local behavior of the target space.
For that, the reach can be defined locally at a point. The reach of
a closed set $\mathbb{X}\subset\mathbb{R}^{d}$ at $q\in\mathbb{X}$,
denoted by $\tau_{\mathbb{X}}(q)$, is defined as 
\[
\tau_{\mathbb{X}}(q):=d\left(q,{\rm Med}(\mathbb{X})\right)=\inf_{x\in\mathrm{Med}(\mathbb{X})}||q-x||.
\]
Then it is direct that the reach is the infimum of the reach at each
point, that is, 
\[
\tau_{\mathbb{X}}=\inf_{q\in\mathbb{X}}\tau_{\mathbb{X}}(q).
\]
The positive condition for the reach imposes smoothness on the target
space. In particular, the reach condition enforces a lower bound on
the inner product between a difference vector between two points and
a normal vector at one of these points, and provides a Lipschitz continuity
on the projection function to the target space.

\begin{theorem}[Theorem 4.8 (7) and (8) in \cite{Federer1959}]
	\label{thm:background_reach_projection}
	Let $\mathbb{X}\subset\mathbb{R}^{d}$ be a subset.
	\begin{enumerate}
		\item[(i)]  If $x\in\mathbb{X}$ and $u\in\mathbb{R}^{d}\backslash Med(\mathbb{X})$,
		then 
		\[
		\left\langle u-\pi_{\mathbb{X}}(u),\pi_{\mathbb{X}}(u)-x\right\rangle \geq-\frac{\left\Vert x-\pi_{\mathbb{X}}(u)\right\Vert ^{2}d_{\mathbb{X}}(u)}{2\tau_{\mathbb{X}}(u)}.
		\]
		\item[(ii)]  If $\epsilon,\tau\in\mathbb{R}$ with $0<\epsilon<\tau$ and $x,y\in\mathbb{R}^{d}\backslash Med(\mathbb{X})$
		with $d_{\mathbb{X}}(x),d_{\mathbb{X}}(u)\leq\epsilon$ and $\tau_{\mathbb{X}}(x),\tau_{\mathbb{X}}(u)\geq\tau$,
		then 
		\[
		\left\Vert \pi_{\mathbb{X}}(x)-\pi_{\mathbb{X}}(u)\right\Vert \leq\frac{\tau}{\tau-\epsilon}\left\Vert x-u\right\Vert .
		\]
	\end{enumerate}
\end{theorem}
In the case of submanifolds, one can reformulate the definition of the reach in the following manner.
\begin{proposition}[Theorem 4.18 in \cite{Federer1959}] \label{thm:background_reach_manifold} If $\mathbb{X}\subset\mathbb{R}^{d}$ is a submanifold of $\mathbb{R}^{d}$, then
	\begin{equation}
	\tau_{\mathbb{X}}=\underset{q_{1}\neq q_{2}\in \mathbb{X}}{\inf}~\frac{\|q_{1}-q_{2}\|_{2}^{2}}{2d(q_{2}-q_{1},T_{q_{1}}\mathbb{X})}.\label{eq:background_reach_supremum}
	\end{equation}
\end{proposition}
Above, $T_q \mathbb{X}$ denotes the tangent space of $\mathbb{X}$ at $q \in \mathbb{X}$. 
This formulation has the advantage of involving only points on $\mathbb{X}$ and its tangent spaces, while \eqref{eq:background_reach_medial_axis} uses the distance to the medial axis $Med(\mathbb{X})$, which is a global quantity.
\begin{figure}
	\centering
	\includegraphics{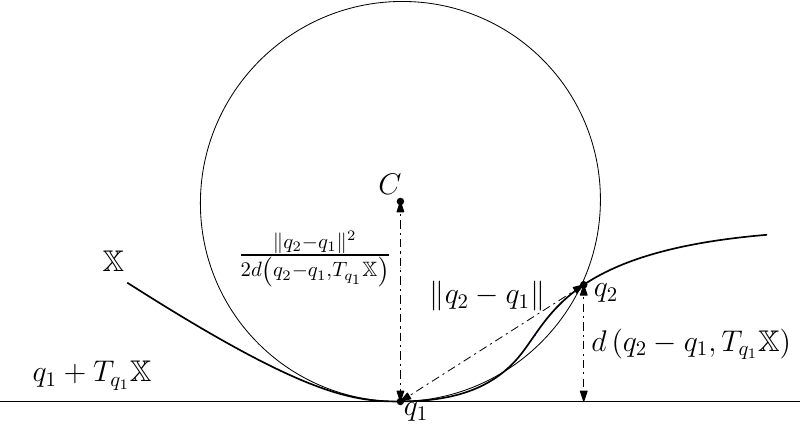}
	\caption{Geometric interpretation of the quantities involved in \eqref{eq:background_reach_supremum}.}\label{fig:background_reach_tangent_ball}
\end{figure}
The ratio appearing in \eqref{eq:background_reach_supremum} has a clear geometric meaning, as it corresponds to the radius of the ambient ball tangent to $\mathbb{X}$ at $q_{1}$ and passing through $q_{2}$. See Figure~\ref{fig:background_reach_tangent_ball}.
Hence, the reach gives a lower bound on the radii of curvature of $\mathbb{X}$.
Equivalently, $\tau_\mathbb{X}^{-1}$ is an upper bound on the directional curvature of $\mathbb{X}$.
\begin{proposition}[Proposition 6.1 in \cite{NiyogiSW2008}]
	\label{thm:background_reach_bound_second_derivative}
	Let $\mathbb{X} \subset \mathbb{R}^{d}$ be a submanifold, and $\gamma_{p,v}$ an arc-length parametrized geodesic of $\mathbb{X}$. Then for all $t >0$,
	\[
	\left\Vert\gamma_{p,v}''(t) \right\Vert \leq 1/ \tau_{\mathbb{X}}.
	\]
\end{proposition}

The reach further provides an upper bound on the injectivity radius and the sectional curvature of $\mathbb{X}$; see Proposition A.1. part (ii) and (iii) respectively, in \cite{AamariKCMRW2019}. Hence the reach is a quantity that characterizes the overall structure of $\mathbb{X}$ and, as argued in \cite{AamariKCMRW2019}, captures structural properties of $\mathbb{X}$ of both global and local nature. In particular, assuming a lower bound on the reach prevents the manifold from being nearly self-intersecting or from having portions with very high curvature \cite[Theorem 3.4]{AamariKCMRW2019}.

We end this section with the bound on the distance function, which is Lemma 3.2 in \cite{ChazalCL2009}. We enhance this bound in Lemma~\ref{lem:defretract_distance_bound_derivative}.

\begin{lemma}
	
	\label{lem:background_distance_concavity}
	
	Let $\mathbb{X}\subset\mathbb{R}^{d}$ be a closed set and $x\in Med_{\mu}(\mathbb{X})$.
	Then for any $y\in\mathbb{R}^{d}$, the distance $d_{\mathbb{X}}(y)$
	is bounded as 
	\[
	d_{\mathbb{X}}(y)^{2}\leq d_{\mathbb{X}}(x)^{2}+\left\Vert x-y\right\Vert ^{2}+2d_{\mathbb{X}}(x)\left\Vert \nabla_{\mathbb{X}}(x)\right\Vert \left\Vert x-y\right\Vert .
	\]
	
\end{lemma}

\begin{theorem}[Theorem 4.1 in \cite{ChazalCLT2007}]
	
	\label{thm:background_reach_offset}
	
	Let $\mathbb{X}\subset\mathbb{R}^{d}$ be a closed set with $\mu$-reach
	$\tau_{\mathbb{X}}^{\mu}>0$. Then for $r\in(0,\tau_{\mathbb{X}}^{\mu})$,
	\[
	\tau_{(\mathbb{X}^{r})^{\complement}}\geq\mu r.
	\]
	
\end{theorem}

\section{Projection on a Positive Reach set}


The positive reach condition imposes smoothness in geometry and topology.
In particular, given two points $x,u\in\mathbb{R}^{d}$ where $u$
has the unique projection $\pi_{\mathbb{X}}(u)\in\mathbb{X}$ to the
target space $\mathbb{X}$, the positive reach condition gives a bound
on the distance $\left\Vert x-\pi_{\mathbb{X}}(u)\right\Vert $ from
the projection $\pi_{\mathbb{X}}(u)$ to $x$ in terms of the distance
$\left\Vert x-u\right\Vert $. When $x$ also has the unique projection
$\pi_{\mathbb{X}}(x)$ and $\tau\geq\tau_{\mathbb{X}}(\pi_{\mathbb{X}}(x)),\tau_{\mathbb{X}}(\pi_{\mathbb{X}}(u))$,
a direct consequence of Theorem~\ref{thm:background_reach_projection} (ii)
gives a bound as 
\begin{equation}
\left\Vert x-\pi_{\mathbb{X}}(u)\right\Vert \leq\left\Vert d_{\mathbb{X}}(x)-\pi_{\mathbb{X}}(u)\right\Vert +\left\Vert x-\pi_{\mathbb{X}}(x)\right\Vert \leq\frac{\tau}{\tau-d_{\mathbb{X}}(u)}\left\Vert x-u\right\Vert +d_{\mathbb{X}}(x).\label{eq:projreach_naivebound}
\end{equation}
This section is devoted to improve \eqref{eq:projreach_naivebound} and give
a tighter bound on $\left\Vert x-\pi_{\mathbb{X}}(u)\right\Vert $.
An improved bound is given in Lemma~\ref{lem:projreach_distance_projection_point}.
Then, for the case where the distance $d_{\mathbb{X}}(u)$ from $u$ to $\mathbb{X}$ is not directly known, 
a bound for a general case is given in Lemma~\ref{lem:projreach_distance_projection_general}
in terms of $\left\Vert x-u\right\Vert $, and a bound for a case
when $u$ is in a convex hull of a set of points $x_{1},\ldots,x_{k}$
is given in Lemma~\ref{lem:projreach_distance_projection_center}
in terms of $\left\Vert x-x_{i}\right\Vert $'s for $i=1,\ldots,k$.
Lemma~\ref{lem:projreach_distance_projection_general} and \ref{lem:projreach_distance_projection_center}
play key roles in showing the homotopy equivalence of the target
space $\mathbb{X}$ to the restricted \v{C}ech complex (Corollary~\ref{cor:nerve_contractible_covering}), the ambient \v{C}ech
complex (Theorem~\ref{thm:homotopy_cech}), and the Vietoris-Rips complex (Theorem~\ref{thm:homotopy_rips}).

The proof of Theorem~\ref{thm:background_reach_projection} (ii) in
\cite{Federer1959} is based on Theorem~\ref{thm:background_reach_projection}
(i). To improve \eqref{eq:projreach_naivebound}, we generalize Theorem
\ref{thm:background_reach_projection} (i) to the case where the point
$x$ need not be on the target space $\mathbb{X}$, in Claim~\ref{claim:projreach_innerproduct_asymmetric}.
The proof is similar to the proof of Theorem~\ref{thm:background_reach_projection} (i) in \cite{Federer1959}.

\begin{claim}
	\label{claim:projreach_innerproduct_asymmetric}
	
	Let $\tau>0$ and $\mathbb{X}\subset\mathbb{R}^{d}$ be a set. Let
	$x\in\mathbb{R}^{d}$ and $u\in\mathbb{R}^{d}\backslash Med(\mathbb{X})$
	with ${\rm reach}(\mathbb{X},\pi_{\mathbb{X}}(u))\geq\tau$ and $d_{\mathbb{X}}(x)\leq\tau$.
	Then 
	\[
	\left\langle u-\pi_{\mathbb{X}}(u),\pi_{\mathbb{X}}(u)-x\right\rangle \geq-\frac{\left\Vert x-\pi_{\mathbb{X}}(u)\right\Vert ^{2}d_{\mathbb{X}}(u)}{2\tau}-d_{\mathbb{X}}(x)d_{\mathbb{X}}(u)\left(1-\frac{d_{\mathbb{X}}(x)}{2\tau}\right).
	\]
	
\end{claim}

\begin{proof}[Proof of Claim~\ref{claim:projreach_innerproduct_asymmetric}]
	
	If $\pi_{\mathbb{X}}(u)=u$, then $d_{\mathbb{X}}(u)=0$ and the inequality
	trivially holds. Assume $\pi_{\mathbb{X}}(u)\neq u$, and we will
	find a lower bound for $\left\langle u-\pi_{\mathbb{X}}(u),\pi_{\mathbb{X}}(u)-x\right\rangle $.
	Let 
	\[
	v=\frac{u-\pi_{\mathbb{X}}(u)}{d_{\mathbb{X}}(u)},\qquad S=\{t>0:\pi_{\mathbb{X}}(\pi_{\mathbb{X}}(u)+tv)=\pi_{\mathbb{X}}(u)\}.
	\]
	Then $\left\Vert u-\pi_{\mathbb{X}}(u)\right\Vert \in S$ implies
	$\sup S>0$, and then Theorem 4.8 (6) in \cite{Federer1959} implies
	that 
	\begin{equation}
	\sup S\geq\tau.\label{eq:projreach_innerproduct_asymmetric_sup}
	\end{equation}
	Now, if $t<\sup S$, then $\left\Vert \pi_{\mathbb{X}}(u)+tv-x\right\Vert $ is lower bounded using $\pi_{\mathbb{X}}(\pi_{\mathbb{X}}(u)+tv)=\pi_{\mathbb{X}}(u)$
	as (see Figure~\ref{fig:projreach_innerproduct_asymmetric_norm_bound} for a graphical illustration)
	\begin{align}
	\left\Vert \pi_{\mathbb{X}}(u)+tv-x\right\Vert  & \geq\left\Vert \pi_{\mathbb{X}}(u)+tv-\pi_{\mathbb{X}}(x)\right\Vert -\left\Vert x-\pi_{\mathbb{X}}(x)\right\Vert \nonumber\\
	& \geq\left\Vert \pi_{\mathbb{X}}(u)+tv-\pi_{\mathbb{X}}(\pi_{\mathbb{X}}(u)+tv)\right\Vert -d_{\mathbb{X}}(x) \nonumber\\
	& =\left\Vert \pi_{\mathbb{X}}(u)+tv-\pi_{\mathbb{X}}(u)\right\Vert -d_{\mathbb{X}}(x)=t-d_{\mathbb{X}}(x).\label{eq:projreach_innerproduct_asymmetric_norm_bound}
	\end{align}
	And since both $\left\Vert \pi_{\mathbb{X}}(u)+tv-x\right\Vert $
	and $t-d_{\mathbb{X}}(x)$ are continuous on $t$, so for all $t\leq\sup S$,
	$\left\Vert \pi_{\mathbb{X}}(u)+tv-x\right\Vert $ is lower bounded
	as 
	\[
	\left\Vert \pi_{\mathbb{X}}(u)+tv-x\right\Vert \geq t-d_{\mathbb{X}}(x).
	\]
	Then additionally under $t\geq d_{\mathbb{X}}(x)$, squaring and expanding
	gives 
	\[
	\left\Vert \pi_{\mathbb{X}}(u)-x\right\Vert ^{2}+2t\left\langle v,\pi_{\mathbb{X}}(u)-x\right\rangle +t^{2}\geq(t-d_{\mathbb{X}}(x))^{2}.
	\]
	Rearranging this gives that if $d_{\mathbb{X}}(x)\leq t\leq\sup S$,
	then $\left\langle v,\pi_{\mathbb{X}}(u)-x\right\rangle $ is lower
	bounded as 
	\begin{equation}
	\left\langle v,\pi_{\mathbb{X}}(u)-x\right\rangle \geq-\frac{\left\Vert \pi_{\mathbb{X}}(u)-x\right\Vert ^{2}}{2t}-d_{\mathbb{X}}(x)\left(1-\frac{d_{\mathbb{X}}(x)}{2t}\right).\label{eq:projreach_innerproduct_asymmetric_innerproduct_bound}
	\end{equation}
	Hence\eqref{eq:projreach_innerproduct_asymmetric_sup} and \eqref{eq:projreach_innerproduct_asymmetric_innerproduct_bound}
	implies that, under the condition $d_{\mathbb{X}}(x)\leq\tau$, applying
	$v=\frac{u-\pi_{\mathbb{X}}(u)}{\left\Vert u-\pi_{\mathbb{X}}(u)\right\Vert }$
	and $t=\tau$ gives a lower bound for $\left\langle u-\pi_{\mathbb{X}}(u),\pi_{\mathbb{X}}(u)-x\right\rangle $
	as 
	\[
	\left\langle u-\pi_{\mathbb{X}}(u),\pi_{\mathbb{X}}(u)-x\right\rangle \geq-\frac{\left\Vert \pi_{\mathbb{X}}(u)-x\right\Vert ^{2}d_{\mathbb{X}}(u)}{2\tau}-d_{\mathbb{X}}(x)d_{\mathbb{X}}(u)\left(1-\frac{d_{\mathbb{X}}(x)}{2\tau}\right).
	\]
	
\end{proof}

\begin{figure}
	\centering
	\includegraphics{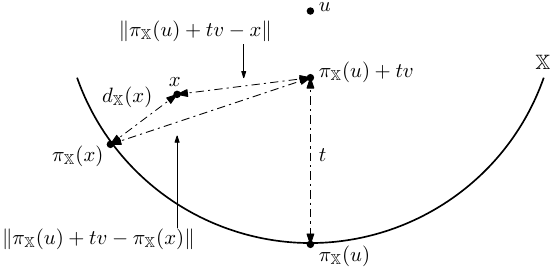}
	\caption{The graphical illustration involved in \eqref{eq:projreach_innerproduct_asymmetric_norm_bound}.}
	\label{fig:projreach_innerproduct_asymmetric_norm_bound}
\end{figure}


Now we are ready to pose a bound for $\left\Vert x-\pi_{\mathbb{X}}(u)\right\Vert $
in Lemma~\ref{lem:projreach_distance_projection_point} that is improved
from \eqref{eq:projreach_naivebound}. The proof is similar to the proof
of Theorem~\ref{thm:background_reach_projection} (ii) in \cite{Federer1959},
while we use Claim~\ref{claim:projreach_innerproduct_asymmetric}
instead of Theorem~\ref{thm:background_reach_projection} (i).

\begin{lemma}
	\label{lem:projreach_distance_projection_point}
	Let $\tau>0$ and $\mathbb{X}\subset\mathbb{R}^{d}$ be a set. Let
	$\epsilon\geq0$, $x\in\mathbb{R}^{d}$, and $u\in\mathbb{R}^{d}\backslash Med(\mathbb{X})$
	with ${\rm reach}(\mathbb{X},\pi_{\mathbb{X}}(u))\geq\tau$ and $d_{\mathbb{X}}(x)\leq\epsilon\leq\tau$.
	Then 
	\begin{equation}
	\left\Vert x-\pi_{\mathbb{X}}(u)\right\Vert \leq\sqrt{\frac{\tau}{\tau-d_{\mathbb{X}}(u)}\left(\left\Vert x-u\right\Vert ^{2}-d_{\mathbb{X}}(u)\left(d_{\mathbb{X}}(u)-2\epsilon+\frac{\epsilon^{2}}{\tau}\right)\right)}.\label{eq:projreach_distance_projection_point}
	\end{equation}
	
\end{lemma}

\begin{proof}[Proof of Lemma~\ref{lem:projreach_distance_projection_point}]
	
	From Claim~\ref{claim:projreach_innerproduct_asymmetric}, $\left\langle u-\pi_{\mathbb{X}}(u),\pi_{\mathbb{X}}(u)-x\right\rangle $
	is lower bounded as 
	\[
	\left\langle u-\pi_{\mathbb{X}}(u),\pi_{\mathbb{X}}(u)-x\right\rangle \geq-\frac{\left\Vert \pi_{\mathbb{X}}(u)-x\right\Vert ^{2}d_{\mathbb{X}}(u)}{2\tau}-d_{\mathbb{X}}(x)d_{\mathbb{X}}(u)\left(1-\frac{d_{\mathbb{X}}(x)}{2\tau}\right).
	\]
	Hence $\left\Vert x-u\right\Vert ^{2}$ can be expanded and lower
	bounded as 
	\begin{align*}
	\left\Vert x-u\right\Vert ^{2} & =\left\Vert u-\pi_{\mathbb{X}}(u)\right\Vert ^{2}+\left\Vert \pi_{\mathbb{X}}(u)-x\right\Vert ^{2}+2\left\langle u-\pi_{\mathbb{X}}(u),\,\pi_{\mathbb{X}}(u)-x\right\rangle \\
	& \geq d_{\mathbb{X}}(u)^{2}+\left\Vert \pi_{\mathbb{X}}(u)-x\right\Vert ^{2}\left(1-\frac{d_{\mathbb{X}}(u)}{\tau}\right)-2d_{\mathbb{X}}(x)d_{\mathbb{X}}(u)\left(1-\frac{d_{\mathbb{X}}(x)}{2\tau}\right).
	\end{align*}
	By rearranging, we obtain an upper bound for $\left\Vert x-\pi_{\mathbb{X}}(u)\right\Vert $
	as 
	\[
	\left\Vert x-\pi_{\mathbb{X}}(u)\right\Vert \leq\sqrt{\frac{\tau}{\tau-d_{\mathbb{X}}(u)}\left(\left\Vert x-u\right\Vert ^{2}-d_{\mathbb{X}}(u)\left(d_{\mathbb{X}}(u)-2d_{\mathbb{X}}(x)+\frac{d_{\mathbb{X}}(x)^{2}}{\tau}\right)\right)}.
	\]
	And $d_{\mathbb{X}}(x)\leq\epsilon\leq\tau$ implies $2d_{\mathbb{X}}(x)-\frac{d_{\mathbb{X}}(x)^{2}}{\tau}\leq2\epsilon-\frac{\epsilon^{2}}{\tau}$,
	and hence $\left\Vert x-\pi_{\mathbb{X}}(u)\right\Vert $ is further
	upper bounded as 
	\[
	\left\Vert x-\pi_{\mathbb{X}}(u)\right\Vert \leq\sqrt{\frac{\tau}{\tau-d_{\mathbb{X}}(u)}\left(\left\Vert x-u\right\Vert ^{2}-d_{\mathbb{X}}(u)\left(d_{\mathbb{X}}(u)-2\epsilon+\frac{\epsilon^{2}}{\tau}\right)\right)}.
	\]
\end{proof}

Note that by setting $\epsilon=d_{\mathbb{X}}(x)$, the bound of \eqref{eq:projreach_distance_projection_point}
is upper bounded by the bound of \eqref{eq:projreach_naivebound} as
\begin{align*}
& \sqrt{\frac{\tau}{\tau-d_{\mathbb{X}}(u)}\left(\left\Vert x-u\right\Vert ^{2}-d_{\mathbb{X}}(u)\left(d_{\mathbb{X}}(u)-2\epsilon+\frac{\epsilon^{2}}{\tau}\right)\right)}\\
& =\sqrt{\frac{\tau}{\tau-d_{\mathbb{X}}(u)}\left\Vert x-u\right\Vert ^{2}-\frac{\tau}{\tau-d_{\mathbb{X}}(u)}(d_{\mathbb{X}}(u)-\epsilon)^{2}+\epsilon^{2}}\\
& \leq\sqrt{\frac{\tau}{\tau-d_{\mathbb{X}}(u)}\left\Vert x-u\right\Vert ^{2}+\epsilon^{2}}\leq\sqrt{\frac{\tau}{\tau-d_{\mathbb{X}}(u)}}\left\Vert x-u\right\Vert +\epsilon\\
& \leq\frac{\tau}{\tau-d_{\mathbb{X}}(u)}\left\Vert x-u\right\Vert +\epsilon,
\end{align*}
so \eqref{eq:projreach_distance_projection_point} is indeed tighter than \eqref{eq:projreach_naivebound}.

For many cases, we don't have direct access to the distance $d_{\mathbb{X}}(u)$
from $u$ to $\mathbb{X}$ but only through a bound $d_{\mathbb{X}}(u)\leq\epsilon_{u}$
for some $\epsilon_{u}\geq0$. For this case, we need to maximize
the bound of \eqref{eq:projreach_distance_projection_point} with
respect to $d_{\mathbb{X}}(u)$. As a result, the bound for $\left\Vert x-\pi_{\mathbb{X}}(u)\right\Vert $
is expressed in terms of $\left\Vert x-u\right\Vert $ and $\epsilon_{u}$
in Lemma~\ref{lem:projreach_distance_projection_general}.
Lemma~\ref{lem:projreach_distance_projection_general} plays a key role in showing the interleaving relationship between the restricted \v{C}ech complex and the ambient \v{C}ech complex in Lemma~\ref{lem:homotopy_interleaving_ambientcech_restrictedcech}.

\begin{lemma} \label{lem:projreach_distance_projection_general}
	
	Let $\tau>0$ and $\mathbb{X}\subset\mathbb{R}^{d}$ be a set. Let
	$\epsilon\geq0$, $x\in\mathbb{R}^{d}$, and $u\in\mathbb{R}^{d}\backslash Med(\mathbb{X})$
	with ${\rm reach}(\mathbb{X},\pi_{\mathbb{X}}(u))\geq\tau$ and $d_{\mathbb{X}}(x)\leq\epsilon\leq\tau$.
	\begin{enumerate}
		\item[(i)] Let $\epsilon_{u}\in\mathbb{R}$ and suppose that $d_{\mathbb{X}}(u)\leq\epsilon_{u}<\tau$.
		Then 
		\[
		\left\Vert x-\pi_{\mathbb{X}}(u)\right\Vert \leq\sqrt{\left\Vert x-u\right\Vert ^{2}+\tilde{r}^{2}-(\tau^{2}+(\tau-\epsilon)^{2}-\left\Vert x-u\right\Vert ^{2}-\tilde{r}^{2})\left(\frac{\tau}{\sqrt{\tau^{2}-\tilde{r}^{2}}}-1\right)},
		\]
		where 
		\begin{align*}
		\tilde{r}^{2} & :=\min\left\{ \left\Vert x-u\right\Vert ^{2}+\epsilon(2\tau-\epsilon),\tau^{2}-(\tau-\epsilon_{u})^{2}\right\} .
		\end{align*}
		\item[(ii)] Suppose that $\left\Vert x-u\right\Vert \leq\tau-\epsilon$, then
		\[
		\left\Vert x-\pi_{\mathbb{X}}(u)\right\Vert \leq\sqrt{\frac{2\tau\left(\left\Vert u-x\right\Vert ^{2}+\epsilon\left(2\tau-\epsilon\right)\right)}{\tau+\sqrt{\tau^{2}-\left(\left\Vert u-x\right\Vert ^{2}+\epsilon\left(2\tau-\epsilon\right)\right)}}-\epsilon\left(2\tau-\epsilon\right)}.
		\]
	\end{enumerate}
\end{lemma}

\begin{proof}[Proof of Lemma~\ref{lem:projreach_distance_projection_general}]
	
	(i)
	
	First, considering Lemma~\ref{lem:projreach_distance_projection_point}
	gives an upper bound for $\left\Vert x-\pi_{\mathbb{X}}(u)\right\Vert $
	as 
	\begin{equation}
	\left\Vert x-\pi_{\mathbb{X}}(u)\right\Vert \leq\sqrt{\frac{\tau}{\tau-d_{\mathbb{X}}(u)}\left(\left\Vert u-x\right\Vert ^{2}-d_{\mathbb{X}}(u)\left(d_{\mathbb{X}}(u)-2\epsilon+\frac{\epsilon^{2}}{\tau}\right)\right)}.\label{eq:projreach_distance_projection_general_initial}
	\end{equation}
	We further bound \eqref{eq:projreach_distance_projection_general_initial}
	by regarding \eqref{eq:projreach_distance_projection_general_initial}
	as a function of $d_{\mathbb{X}}(u)$ and maximize with respect to
	$d_{\mathbb{X}}(u)$. Let $\tilde{r}_{u}:=\left\Vert u-x\right\Vert $,
	$\tilde{r}_{\mathbb{X}}:=2\epsilon-\frac{\epsilon^{2}}{\tau}$, and
	$\tilde{r}_{yz}:=\sqrt{\tau^{2}-(\tau-\epsilon_{u})^{2}}$ for convenience,
	so that $\epsilon_{u}=\tau-\sqrt{\tau^{2}-\tilde{r}_{yz}^{2}}$. Now,
	consider the function 
	\[
	t\in\left[0,\tau-\sqrt{\tau^{2}-\tilde{r}_{yz}^{2}}\right]\mapsto f(t):=\frac{\tau(\tilde{r}_{u}^{2}-t^{2}+\tilde{r}_{\mathbb{X}}t)}{\tau-t}.
	\]
	Then its derivative is 
	\[
	f'(t)=\frac{\tau(t^{2}-2\tau t+\tilde{r}_{u}^{2}+\tilde{r}_{\mathbb{X}}\tau)}{(\tau-t)^{2}}.
	\]
	Hence if $\tilde{r}_{u}^{2}+\tilde{r}_{\mathbb{X}}\tau\geq\tau^{2}$
	then $f'(t)\geq0$ holds for all $t\in\left[0,\tau-\sqrt{\tau^{2}-\tilde{r}_{yz}^{2}}\right]$,
	and if $\tilde{r}_{u}^{2}+\tilde{r}_{\mathbb{X}}\tau\leq\tau^{2}$
	then $f'(t)\geq0$ if and only if $t\leq\tau-\sqrt{\tau^{2}-(\tilde{r}_{u}^{2}+\tilde{r}_{\mathbb{X}}\tau)}$.
	Hence if $\tilde{r}_{yz}^{2}\leq\tilde{r}_{u}^{2}+\tilde{r}_{\mathbb{X}}\tau$
	then $f$ attains its maximum at $t=\tau-\sqrt{\tau^{2}-\tilde{r}_{yz}^{2}}$,
	and if $\tilde{r}_{u}^{2}+\tilde{r}_{\mathbb{X}}\tau\leq\tilde{r}_{yz}^{2}$
	then $f$ attains its maximum at $t=\tau-\sqrt{\tau^{2}-(\tilde{r}_{u}^{2}+\tilde{r}_{\mathbb{X}}\tau)}$.
	Hence by letting $\tilde{r}:=\min\{\sqrt{\tilde{r}_{u}^{2}+\tilde{r}_{\mathbb{X}}\tau},\tilde{r}_{yz}\}$,
	for all $t\in\left[0,\tau-\sqrt{\tau^{2}-\tilde{r}_{yz}^{2}}\right]$,
	\[
	f(t)\leq f\left(\tau-\sqrt{\tau^{2}-\tilde{r}^{2}}\right)=\frac{\tau\left(\tilde{r}_{u}^{2}+\tilde{r}^{2}-2\tau^{2}+\tilde{r}_{\mathbb{X}}\tau+(2\tau-\tilde{r}_{\mathbb{X}})\sqrt{\tau^{2}-\tilde{r}^{2}}\right)}{\sqrt{\tau^{2}-\tilde{r}^{2}}}.
	\]
	Hence \eqref{eq:projreach_distance_projection_general_initial} is
	correspondingly further upper bounded as 
	\begin{align}
	\left\Vert x-\pi_{\mathbb{X}}(u)\right\Vert  & \leq\sqrt{\frac{\tau}{\tau-d_{\mathbb{X}}(u)}\left(\left\Vert u-x\right\Vert ^{2}-d_{\mathbb{X}}(u)\left(d_{\mathbb{X}}(u)-2\epsilon+\frac{\epsilon^{2}}{\tau}\right)\right)}\nonumber \\
	& \leq\sqrt{\frac{\tau\left(\tilde{r}_{u}^{2}+\tilde{r}^{2}-2\tau^{2}+\tilde{r}_{\mathbb{X}}\tau+(2\tau-\tilde{r}_{\mathbb{X}})\sqrt{\tau^{2}-\tilde{r}^{2}}\right)}{\sqrt{\tau^{2}-\tilde{r}^{2}}}}\label{eq:projreach_distance_projection_general_maxexpand}\\
	& =\sqrt{\tilde{r}_{u}^{2}+\tilde{r}^{2}-(2\tau^{2}-\tilde{r}_{u}^{2}-\tilde{r}^{2}-\tilde{r}_{\mathbb{X}}\tau)\left(\frac{\tau}{\sqrt{\tau^{2}-\tilde{r}^{2}}}-1\right)}\nonumber \\
	& =\sqrt{\left\Vert x-u\right\Vert ^{2}+\tilde{r}^{2}-(\tau^{2}+(\tau-\epsilon)^{2}-\left\Vert x-u\right\Vert ^{2}-\tilde{r}^{2})\left(\frac{\tau}{\sqrt{\tau^{2}-\tilde{r}^{2}}}-1\right)},\nonumber 
	\end{align}
	where 
	\begin{align*}
	\tilde{r}^{2} & :=\min\left\{ \left\Vert x-u\right\Vert ^{2}+\epsilon(2\tau-\epsilon),\tau^{2}-(\tau-\epsilon_{u})^{2}\right\} .
	\end{align*}
	
	(ii)
	
	Since $d_{\mathbb{X}}(u)<\tau$, we can set $\epsilon_{u}\to\tau$,
	which implies $\tilde{r}_{yz}^{2}\to\tau^{2}$. And under $\left\Vert x-u\right\Vert \leq\tau-\epsilon$,
	$\tilde{r}_{u}^{2}+\tilde{r}_{\mathbb{X}}\tau\leq(\tau-\epsilon)^{2}+\epsilon(2\tau-\epsilon)=\tau^{2}$,
	and hence $\tilde{r}^{2}\to\tilde{r}_{u}^{2}+\tilde{r}_{\mathbb{X}}\tau$.
	Then \eqref{eq:projreach_distance_projection_general_maxexpand} converges
	to
	\[
	\frac{\tau\left(\tilde{r}_{u}^{2}+\tilde{r}^{2}-2\tau^{2}+\tilde{r}_{\mathbb{X}}\tau+(2\tau-\tilde{r}_{\mathbb{X}})\sqrt{\tau^{2}-\tilde{r}^{2}}\right)}{\sqrt{\tau^{2}-\tilde{r}^{2}}}\to\tau\left(2\tau-\tilde{r}_{\mathbb{X}}-2\sqrt{\tau^{2}-(\tilde{r}_{u}^{2}+\tilde{r}_{\mathbb{X}}\tau)}\right),
	\]
	and hence $\left\Vert x-\pi_{\mathbb{X}}(u)\right\Vert $ is correspondingly
	bounded as 
	\begin{align*}
	\left\Vert x-\pi_{\mathbb{X}}(u)\right\Vert  & \leq\sqrt{\tau\left(2\tau-\tilde{r}_{\mathbb{X}}-2\sqrt{\tau^{2}-(\tilde{r}_{u}^{2}+\tilde{r}_{\mathbb{X}}\tau)}\right)}\\
	& =\sqrt{\frac{2\tau(\tilde{r}_{u}^{2}+\tilde{r}_{\mathbb{X}}\tau)}{\tau+\sqrt{\tau^{2}-(\tilde{r}_{u}^{2}+\tilde{r}_{\mathbb{X}}\tau)}}-\tilde{r}_{\mathbb{X}}\tau}\\
	& =\sqrt{\frac{2\tau\left(\left\Vert u-x\right\Vert ^{2}+\epsilon\left(2\tau-\epsilon\right)\right)}{\tau+\sqrt{\tau^{2}-\left(\left\Vert u-x\right\Vert ^{2}+\epsilon\left(2\tau-\epsilon\right)\right)}}-\epsilon\left(2\tau-\epsilon\right)}.
	\end{align*}
\end{proof}

Now, we consider the case when the point $u$ is in a convex hull of a set of points $x_{1},\ldots,x_{k}$. 
We first start with a simple calculation of the distance from one vertex of a simplex to another point lying on the opposite side.
\begin{claim}
	\label{claim:projreach_distance_midpoint}
	Let $x,x_{1},\ldots,x_{k}\in\mathbb{R}^{d}$ and $\lambda_{1},\ldots,\lambda_{k}\in[0,1]$
	with $\sum_{i=1}^{k}\lambda_{i}=1$. Then 
	\begin{align}
	\left\Vert x-\sum_{i=1}^{k}\lambda_{i}x_{i}\right\Vert  & =\sqrt{\sum_{i=1}^{k}\lambda_{i}\left\Vert x-x_{i}\right\Vert ^{2}-\sum_{1\leq i<j\leq k}\lambda_{i}\lambda_{j}\left\Vert x_{i}-x_{j}\right\Vert ^{2}}\label{eq:projreach_distance_midpoint_pairwise}\\
	& =\sqrt{\sum_{i=1}^{k}\lambda_{i}\left\Vert x-x_{i}\right\Vert ^{2}-\sum_{i=1}^{k}\lambda_{i}\left\Vert x_{i}-\sum_{j=1}^{k}\lambda_{j}x_{j}\right\Vert ^{2}}.\label{eq:projreach_distance_midpoint_midpoint}
	\end{align}
\end{claim}


\begin{proof}[Proof of Claim~\ref{claim:projreach_distance_midpoint}]
	The distance from $\sum_{i=1}^{k}\lambda_{i}x_{i}$ to $x$ can be
	expanded as 
	\begin{align}
	\left\Vert x-\sum_{i=1}^{k}\lambda_{i}x_{i}\right\Vert ^{2} & =\left\Vert \sum_{i=1}^{k}\lambda_{i}(x-x_{i})\right\Vert ^{2}\nonumber \\
	& =\sum_{i=1}^{k}\lambda_{i}^{2}\left\Vert x-x_{i}\right\Vert ^{2}+2\sum_{1\leq i<j\leq k}\lambda_{i}\lambda_{j}\left\langle x-x_{i},x-x_{j}\right\rangle .\label{eq:projreach_distance_midpoint_expansion}
	\end{align}
	Then applying the identity $2\left\langle x-x_{i},x-x_{j}\right\rangle =\left\Vert x-x_{i}\right\Vert ^{2}+\left\Vert x-x_{j}\right\Vert ^{2}-\left\Vert x_{i}-x_{j}\right\Vert ^{2}$
	to \eqref{eq:projreach_distance_midpoint_expansion} gives 
	\begin{equation}
	\left\Vert x-\sum_{i=1}^{k}\lambda_{i}x_{i}\right\Vert ^{2}=\sum_{i=1}^{k}\lambda_{i}\left\Vert x-x_{i}\right\Vert ^{2}-\sum_{1\leq i<j\leq k}\lambda_{i}\lambda_{j}\left\Vert x_{i}-x_{j}\right\Vert ^{2},\label{eq:projreach_distance_midpoint_pairwise_square}
	\end{equation}
	which gives \eqref{eq:projreach_distance_midpoint_pairwise}. Then,
	applying $x=\sum_{i=1}^{k}\lambda_{i}x_{i}$ to \eqref{eq:projreach_distance_midpoint_pairwise_square}
	gives 
	\begin{equation}
	\sum_{i=1}^{k}\lambda_{i}\left\Vert x_{i}-\sum_{j=1}^{k}\lambda_{j}x_{j}\right\Vert ^{2}=\sum_{1\leq i<j\leq k}\lambda_{i}\lambda_{j}\left\Vert x_{i}-x_{j}\right\Vert ^{2},\label{eq:projreach_distance_midpoint_pairwise_midpoint}
	\end{equation}
	and combining \eqref{eq:projreach_distance_midpoint_pairwise_square}
	and \eqref{eq:projreach_distance_midpoint_pairwise_midpoint} gives
	\eqref{eq:projreach_distance_midpoint_midpoint}.
\end{proof}

If the points $x_{1},\ldots,x_{k}$ are not far from the target space $\mathbb{X}$, then so is the convex
combination $u$. Lemma~\ref{lem:projreach_distance_segment} gives
a bound on the distance from any point $u$ on the convex hull to
its projection $\pi_{\mathbb{X}}(u)$ on $\mathbb{X}$. See Figure
\ref{fig:projreach_distance_segment}.

\begin{lemma}
	\label{lem:projreach_distance_segment}
	
	Let $\tau>0$ and $\mathbb{X}\subset\mathbb{R}^{d}$ be a subset with reach $\tau_{\mathbb{X}}\geq\tau$.
	Let $\epsilon_{1},\ldots,\epsilon_{k}\geq0$ and $x_{1},\ldots,x_{k}\in\mathbb{R}^{d}$
	with $d_{\mathbb{X}}(x_{i})\leq\epsilon_{i}\leq\tau$ for $i=1,\ldots,k$.
	Let $\lambda_{1},\ldots,\lambda_{k}\in[0,1]$ with $\sum_{i=1}^{k}\lambda_{i}=1$,
	and let $u:=\sum_{i=1}^{k}\lambda_{i}x_{i}$ be such that $d_{\mathbb{X}}(u)<\tau$. Then 
	\begin{align*}
	\left\Vert u-\pi_{\mathbb{X}}(u)\right\Vert  & \leq\tau-\sqrt{\left(\sum_{i=1}^{k}\lambda_{i}(\tau-\epsilon_{i})^{2}-\sum_{1\leq i<j\leq k}\lambda_{i}\lambda_{j}\left\Vert x_{i}-x_{j}\right\Vert ^{2}\right)_{+}}\\
	& =\tau-\sqrt{\left(\sum_{i=1}^{k}\lambda_{i}(\tau-\epsilon_{i})^{2}-\sum_{i=1}^{k}\lambda_{i}\left\Vert x_{i}-\sum_{j=1}^{k}\lambda_{j}x_{j}\right\Vert ^{2}\right)_{+}}.
	\end{align*}
	
\end{lemma}

\begin{figure}
	\centering
	\includegraphics{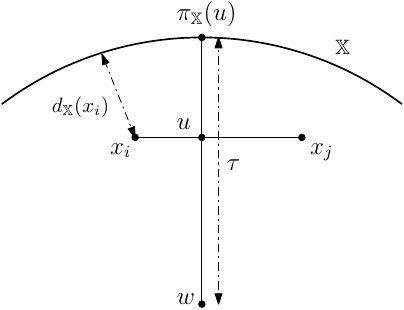}
	\caption{Bound on the distance from any point on the segment to its projection
		on $\mathbb{X}$, as in Lemma~\ref{lem:projreach_distance_segment}.}
	\label{fig:projreach_distance_segment}
\end{figure}

\begin{proof}[Proof of Lemma~\ref{lem:projreach_distance_segment}]
	
	If $\pi_{\mathbb{X}}(u)=u$, then there is nothing to prove. Now,
	suppose $\pi_{\mathbb{X}}(u)\neq u$, and let $w:=\pi_{\mathbb{X}}(u)+\tau\frac{u-\pi_{\mathbb{X}}(u)}{\left\Vert u-\pi_{\mathbb{X}}(u)\right\Vert }$.
	Then, we have that $\left\Vert w-\pi_{\mathbb{X}}(u)\right\Vert =\tau$
	and $w-u=\left(\frac{\tau-\left\Vert \pi_{\mathbb{X}}(u)-u\right\Vert }{\left\Vert \pi_{\mathbb{X}}(u)-u\right\Vert }\right)(u-\pi_{\mathbb{X}}(u))$.
	Since $\left\Vert u-\pi_{\mathbb{X}}(u)\right\Vert <\tau$, it follows
	that $\left\langle w-u,u-\pi_{\mathbb{X}}(u)\right\rangle =\left\Vert w-u\right\Vert \left\Vert u-\pi_{\mathbb{X}}(u)\right\Vert $
	and $\left\Vert u-\pi_{\mathbb{X}}(u)\right\Vert +\left\Vert w-u\right\Vert =\left\Vert w-\pi_{\mathbb{X}}(u)\right\Vert $,
	as in Figure~\ref{fig:projreach_distance_segment}. Since $\pi_{\mathbb{X}}\left(\pi_{\mathbb{X}}(u)+\left\Vert u-\pi_{\mathbb{X}}(u)\right\Vert \frac{u-\pi_{\mathbb{X}}(u)}{\left\Vert u-\pi_{\mathbb{X}}(u)\right\Vert }\right)=\pi_{\mathbb{X}}(u)$
	and $\pi_{\mathbb{X}}(u)+r\frac{u-\pi_{\mathbb{X}}(u)}{\left\Vert u-\pi_{\mathbb{X}}(u)\right\Vert }\notin Med(\mathbb{X})$
	for all $r<\tau$, Theorem 4.8 (2) and (6) in \cite{Federer1959}
	imply that 
	\[
	\pi_{\mathbb{X}}\left(\pi_{\mathbb{X}}(u)+r\frac{u-\pi_{\mathbb{X}}(u)}{\left\Vert u-\pi_{\mathbb{X}}(u)\right\Vert }\right)=\pi_{\mathbb{X}}(u)
	\]
	for all $r<\tau$. Thus, $\mathbb{B}(w,\tau)\cap\mathbb{X}=\emptyset$
	and we can conclude that $\left\Vert w-x_{i}\right\Vert \geq\tau-d_{\mathbb{X}}(x_{i})\geq\tau-\epsilon_{i}$
	for $i=1,\ldots,k$. Applying Claim~\ref{claim:projreach_distance_midpoint}
	to $\left\Vert w-u\right\Vert $ gives a lower bound for $\left\Vert w-u\right\Vert $
	as 
	\begin{align*}
	\left\Vert w-u\right\Vert  & =\sqrt{\sum_{i=1}^{k}\lambda_{i}\left\Vert x_{i}-w\right\Vert ^{2}-\sum_{1\leq i<j\leq k}\lambda_{i}\lambda_{j}\left\Vert x_{i}-x_{j}\right\Vert ^{2}}\\
	& \geq\sqrt{\left(\sum_{i=1}^{k}\lambda_{i}(\tau-\epsilon_{i})^{2}-\sum_{1\leq i<j\leq k}\lambda_{i}\lambda_{j}\left\Vert x_{i}-x_{j}\right\Vert ^{2}\right)_{+}}.
	\end{align*}
	Then, applying $\left\Vert u-\pi_{\mathbb{X}}(u)\right\Vert =\left\Vert w-\pi_{\mathbb{X}}(u)\right\Vert -\left\Vert w-u\right\Vert $
	and Claim~\ref{claim:projreach_distance_midpoint} gives upper bounds
	for $\left\Vert u-\pi_{\mathbb{X}}(u)\right\Vert $ as 
	\begin{align*}
	\left\Vert u-\pi_{\mathbb{X}}(u)\right\Vert  & \leq\tau-\sqrt{\left(\sum_{i=1}^{k}\lambda_{i}(\tau-\epsilon_{i})^{2}-\sum_{1\leq i<j\leq k}\lambda_{i}\lambda_{j}\left\Vert x_{i}-x_{j}\right\Vert ^{2}\right)_{+}}\\
	& =\tau-\sqrt{\left(\sum_{i=1}^{k}\lambda_{i}(\tau-\epsilon_{i})^{2}-\sum_{i=1}^{k}\lambda_{i}\left\Vert x_{i}-\sum_{j=1}^{k}\lambda_{j}x_{j}\right\Vert ^{2}\right)_{+}}.
	\end{align*}
\end{proof}

Now, combining Lemma~\ref{lem:projreach_distance_projection_general}
and \ref{lem:projreach_distance_segment} and gives the bound on $\left\Vert x-u\right\Vert $
in terms of $\left\Vert x-x_{i}\right\Vert $'s for $i=1,\ldots,k$,
in Lemma~\ref{lem:projreach_distance_projection_center}.
Lemma~\ref{lem:projreach_distance_projection_center} plays a key role in showing the contractibility of the intersection of restricted balls in Theorem~\ref{thm:nerve_contractible}.

\begin{lemma} \label{lem:projreach_distance_projection_center}
	
	Let $\tau>0$ and $\mathbb{X}\subset\mathbb{R}^{d}$ be a subset with
	reach $\tau_{\mathbb{X}}\geq\tau$. Let $\epsilon,\epsilon_{1},\ldots,\epsilon_{k}\geq0$
	and $x,x_{1},\ldots,x_{k}\in\mathbb{R}^{d}$ with $d_{\mathbb{X}}(x)\leq\epsilon\leq\tau$,
	$d_{\mathbb{X}}(x_{i})\leq\epsilon_{i}\leq\tau$, and 
	\[
	\left\Vert x-x_{i}\right\Vert <\sqrt{\left(\tau-\epsilon\right)^{2}+\left(\tau-\epsilon_{i}\right)^{2}},
	\]
	for each $i=1,\ldots,k$. Let $\lambda_{1},\ldots,\lambda_{k}\in[0,1]$
	with $\sum_{i=1}^{k}\lambda_{i}=1$, and let $u:=\sum_{i=1}^{k}\lambda_{i}x_{i}$.
	Then
	$d_{\mathbb{X}}(u)<\tau$ and
	\[
	\left\Vert x-\pi_{\mathbb{X}}(u)\right\Vert \leq\sqrt{\tilde{r}_{x}^{2}-(\tau^{2}-\tilde{r}_{x}^{2}+(\tau-\epsilon)^{2})\left(\frac{\tau}{\sqrt{\tau^{2}-\tilde{r}_{u,x}^{2}}}-1\right)},
	\]
	where 
	\begin{align*}
		\tilde{r}_{x}^{2} & :=\sum_{i=1}^{k}\lambda_{i}\left(\left\Vert x_{i}-x\right\Vert ^{2}+\epsilon_{i}(2\tau-\epsilon_{i})\right),\\
		\tilde{r}_{u,x}^{2} & :=\min\left\{ \sum_{i=1}^{k}\lambda_{i}\left(\left\Vert x_{i}-u\right\Vert ^{2}+\epsilon_{i}(2\tau-\epsilon_{i})\right),\frac{1}{2}(\tilde{r}_{x}^{2}+\epsilon(2\tau-\epsilon))\right\} .
	\end{align*}
\end{lemma}

\begin{proof}[Proof of Lemma~\ref{lem:projreach_distance_projection_center}]

	Let $r:=\sqrt{\sum_{i=1}^{k}\lambda_{i}\left\Vert x_{i}-x\right\Vert ^{2}}$,
	then 
	\[
	r<\sqrt{\left(\tau-\epsilon\right)^{2}+\sum_{i=1}^{k}\lambda_{i}\left(\tau-\epsilon_{i}\right)^{2}}.
	\]
	Then from Claim~\ref{claim:projreach_distance_midpoint},
	\begin{equation}
		\left\Vert x-u\right\Vert =\sqrt{\sum_{i=1}^{k}\lambda_{i}\left\Vert x_{i}-x\right\Vert ^{2}-\sum_{i=1}^{k}\lambda_{i}\left\Vert x_{i}-u\right\Vert ^{2}}=\sqrt{r^{2}-\sum_{i=1}^{k}\lambda_{i}\left\Vert x_{i}-u\right\Vert ^{2}},\label{eq:projreach_distance_midpoint_support}
	\end{equation}
	whose rearrangement gives 
	\[
	\left\Vert x-u\right\Vert ^{2}+\sum_{i=1}^{k}\lambda_{i}\left\Vert x_{i}-u\right\Vert ^{2}=r^{2}.
	\]
	Then $r<\sqrt{\left(\tau-\epsilon\right)^{2}+\sum_{i=1}^{k}\lambda_{i}\left(\tau-\epsilon_{i}\right)^{2}}$
	implies that either $\left\Vert x-u\right\Vert <\tau-\epsilon$ or
	$\left\Vert x_{i}-u\right\Vert <\tau-\epsilon_{i}$ for some $i,$and
	hence $d_{\mathbb{X}}(u)<\tau$ holds for either case. Hence applying Lemma
	\ref{lem:projreach_distance_segment} implies that 
	\begin{equation}
	\left\Vert u-\pi_{\mathbb{X}}(u)\right\Vert \leq\tau-\sqrt{\left(\tau^{2}-\sum_{i=1}^{k}\lambda_{i}\epsilon_{i}(2\tau-\epsilon_{i})-\sum_{i=1}^{k}\lambda_{i}\left\Vert x_{i}-u\right\Vert ^{2}\right)_{+}}.\label{eq:projreach_distance_segment_support}
	\end{equation}
	Hence, combining \eqref{eq:projreach_distance_segment_support} and
	Lemma~\ref{lem:projreach_distance_projection_point} gives 
	\begin{equation}
	\left\Vert x-\pi_{\mathbb{X}}(u)\right\Vert \leq\sqrt{\left\Vert x-u\right\Vert ^{2}+\tilde{r}^{2}-(\tau^{2}+(\tau-\epsilon)^{2}-\left\Vert x-u\right\Vert ^{2}-\tilde{r}^{2})\left(\frac{\tau}{\sqrt{\tau^{2}-\tilde{r}^{2}}}-1\right)},\label{eq:projreach_distance_projection_center_initial}
	\end{equation}
	where 
	\begin{align*}
	\tilde{r}^{2} & :=\min\left\{ \left\Vert x-u\right\Vert ^{2}+\epsilon(2\tau-\epsilon),\min\left\{ \sum_{i=1}^{k}\lambda_{i}\epsilon_{i}(2\tau-\epsilon_{i})+\sum_{i=1}^{k}\lambda_{i}\left\Vert x_{i}-u\right\Vert ^{2},\tau^{2}\right\} \right\} .
	\end{align*}
	For further bounding \eqref{eq:projreach_distance_projection_center_initial},
	use the following notations for convenience: 
	\[
	\tilde{r}_{u}:=\left\Vert u-x\right\Vert ,\qquad\tilde{r}_{\mathbb{X}}:=2\epsilon-\frac{\epsilon^{2}}{\tau},\qquad\tilde{r}_{yz}:=\sqrt{\sum_{i=1}^{k}\lambda_{i}\epsilon_{i}(2\tau-\epsilon_{i})+\sum_{i=1}^{k}\lambda_{i}\left\Vert x_{i}-u\right\Vert ^{2}}.
	\]
	First, note that $\tilde{r}_{u}^{2}+\tilde{r}_{yz}^{2}$
	can be expanded as 
	\begin{align}
	\tilde{r}_{u}^{2}+\tilde{r}_{yz}^{2} & =\left\Vert x-u\right\Vert ^{2}+\sum_{i=1}^{k}\lambda_{i}\epsilon_{i}(2\tau-\epsilon_{i})+\sum_{i=1}^{k}\lambda_{i}\left\Vert x_{i}-u\right\Vert ^{2}\nonumber \\
	& =\left(r^{2}-\sum_{i=1}^{k}\lambda_{i}\left\Vert x_{i}-u\right\Vert ^{2}\right)+\sum_{i=1}^{k}\lambda_{i}\epsilon_{i}(2\tau-\epsilon_{i})+\sum_{i=1}^{k}\lambda_{i}\left\Vert x_{i}-u\right\Vert ^{2}\nonumber \\
	& =r^{2}+\sum_{i=1}^{k}\lambda_{i}\epsilon_{i}(2\tau-\epsilon_{i})\nonumber \\
	& =\sum_{i=1}^{k}\lambda_{i}\left(\left\Vert x_{i}-x\right\Vert ^{2}+\epsilon_{i}(2\tau-\epsilon_{i})\right)=:\tilde{r}_{x}^{2}.\label{eq:projreach_distance_projection_center_sum}
	\end{align}
	Then, $\left\Vert x-x_{i}\right\Vert <\sqrt{(\tau-\epsilon)^{2}+(\tau-\epsilon_{i})^{2}}$
	implies that $\tilde{r}_{u}^{2}+\tilde{r}_{\mathbb{X}}\tau+\tilde{r}_{yz}^{2}$
	is bounded as 
	\begin{align*}
	\tilde{r}_{u}^{2}+\tilde{r}_{\mathbb{X}}\tau+\tilde{r}_{yz}^{2} & =\sum_{i=1}^{k}\lambda_{i}\left(\left\Vert x_{i}-x\right\Vert ^{2}+\epsilon_{i}(2\tau-\epsilon_{i})\right)+\epsilon(2\tau-\epsilon)\\
	& <\sum_{i=1}^{k}\lambda_{i}\left((\tau-\epsilon)^{2}+(\tau-\epsilon_{i})^{2}+\epsilon_{i}(2\tau-\epsilon_{i})\right)+\epsilon(2\tau-\epsilon)=2\tau^{2},
	\end{align*}
	and hence 
	\[
	\tilde{r}^{2}=\min\left\{ \tilde{r}_{u}^{2}+\tilde{r}_{\mathbb{X}}\tau,\tilde{r}_{yz}^{2}\right\} .
	\]
	Now, we split into cases whether $\tilde{r}_{yz}\leq\sqrt{\tilde{r}_{u}^{2}+\tilde{r}_{\mathbb{X}}\tau}$
	or $\sqrt{\tilde{r}_{u}^{2}+\tilde{r}_{\mathbb{X}}\tau}\leq\tilde{r}_{yz}$.
	For $\tilde{r}_{yz}\leq\sqrt{\tilde{r}_{u}^{2}+\tilde{r}_{\mathbb{X}}\tau}$
	case, applying $\tilde{r}=\tilde{r}_{yz}$ to \eqref{eq:projreach_distance_projection_center_initial}
	gives 
	\begin{equation}
	\left\Vert x-\pi_{\mathbb{X}}(u)\right\Vert \leq\sqrt{\tilde{r}_{u}^{2}+\tilde{r}_{yz}^{2}-(2\tau^{2}-\tilde{r}_{u}^{2}-\tilde{r}_{yz}^{2}-\tilde{r}_{\mathbb{X}}\tau)\left(\frac{\tau}{\sqrt{\tau^{2}-\tilde{r}_{yz}^{2}}}-1\right)}.\label{eq:projreach_distance_projection_center_first}
	\end{equation}
	For $\sqrt{\tilde{r}_{u}^{2}+\tilde{r}_{\mathbb{X}}\tau}\leq\tilde{r}_{yz}$
	case, applying $\tilde{r}=\sqrt{\tilde{r}_{u}^{2}+\tilde{r}_{\mathbb{X}}\tau}$
	to \eqref{eq:projreach_distance_projection_center_initial} gives
	\begin{align*}
	\left\Vert x-\pi_{\mathbb{X}}(u)\right\Vert  & \leq\sqrt{2\tilde{r}_{u}^{2}+\tilde{r}_{\mathbb{X}}\tau-(2\tau^{2}-2\tilde{r}_{u}^{2}-2\tilde{r}_{\mathbb{X}}\tau)\left(\frac{\tau}{\sqrt{\tau^{2}-(\tilde{r}_{u}^{2}+\tilde{r}_{\mathbb{X}}\tau)}}-1\right)}\\
	& =\sqrt{2\tau^{2}-\tilde{r}_{\mathbb{X}}\tau-2\tau\sqrt{\tau^{2}-(\tilde{r}_{u}^{2}+\tilde{r}_{\mathbb{X}}\tau)}}
	\end{align*}
	Then RHS is an increasing function of $\tilde{r}_{u}$. Hence applying
	$\tilde{r}_{u}^{2}\leq\frac{1}{2}(\tilde{r}_{u}^{2}+\tilde{r}_{yz}^{2}-\tilde{r}_{\mathbb{X}}\tau)$
	gives 
	\begin{equation}
	\left\Vert x-\pi_{\mathbb{X}}(u)\right\Vert \leq\sqrt{\tilde{r}_{u}^{2}+\tilde{r}_{yz}^{2}-(2\tau^{2}-\tilde{r}_{u}^{2}-\tilde{r}_{yz}^{2}-\tilde{r}_{\mathbb{X}}\tau)\left(\frac{\tau}{\sqrt{\tau^{2}-\frac{1}{2}(\tilde{r}_{u}^{2}+\tilde{r}_{yz}^{2}+\tilde{r}_{\mathbb{X}}\tau)}}-1\right)}.\label{eq:projreach_distance_projection_center_second}
	\end{equation}
	Hence, combining \eqref{eq:projreach_distance_projection_center_first}
	and \eqref{eq:projreach_distance_projection_center_second} gives
	that 
	\begin{align*}
	& \left\Vert x-\pi_{\mathbb{X}}(u)\right\Vert \\
	& \leq\sqrt{\tilde{r}_{u}^{2}+\tilde{r}_{yz}^{2}-(2\tau^{2}-\tilde{r}_{u}^{2}-\tilde{r}_{yz}^{2}-\tilde{r}_{\mathbb{X}}\tau)\left(\frac{\tau}{\sqrt{\tau^{2}-\max\left\{ \tilde{r}_{yz}^{2},\frac{1}{2}(\tilde{r}_{u}^{2}+\tilde{r}_{yz}^{2}+\tilde{r}_{\mathbb{X}}\tau)\right\} }}-1\right)}.
	\end{align*}
	Then, \eqref{eq:projreach_distance_projection_center_sum} gives that
	$\left\Vert x-\pi_{\mathbb{X}}(u)\right\Vert $ is upper bounded as
	\[
	\left\Vert x-\pi_{\mathbb{X}}(u)\right\Vert \leq\sqrt{\tilde{r}_{x}^{2}-(\tau^{2}-\tilde{r}_{x}^{2}+(\tau-\epsilon)^{2})\left(\frac{\tau}{\sqrt{\tau^{2}-\tilde{r}_{u,x}^{2}}}-1\right)},
	\]
	where 
	\begin{align*}
	\tilde{r}_{x}^{2} & :=\sum_{i=1}^{k}\lambda_{i}\left(\left\Vert x_{i}-x\right\Vert ^{2}+\epsilon_{i}(2\tau-\epsilon_{i})\right),\\
	\tilde{r}_{u,x}^{2} & :=\min\left\{ \sum_{i=1}^{k}\lambda_{i}\left(\left\Vert x_{i}-u\right\Vert ^{2}+\epsilon_{i}(2\tau-\epsilon_{i})\right),\frac{1}{2}(\tilde{r}_{x}^{2}+\epsilon(2\tau-\epsilon)\right\} .
	\end{align*}

\end{proof}

\section{Proofs for Section~\ref{sec:nerve}}
\label{app:nerve_proof}

This section provides the proofs for Section~\ref{sec:nerve}, in
particular focuses on proving Theorem~\ref{thm:nerve_contractible}.
To show the contractibility of the intersection of restricted balls $\bigcap_{x\in I}\mathbb{B}_{\mathbb{X}}(x,r_{x})$,
we fix a point $y_{0}\in\bigcap_{x\in I}\mathbb{B}_{\mathbb{X}}(x,r_{x})$
and show that $\bigcap_{x\in I}\mathbb{B}_{\mathbb{X}}(x,r_{x})$
deformation retracts to $\{y_{0}\}$. This deformation retract is
constructed by sending each $y\in\bigcap_{x\in I}\mathbb{B}_{\mathbb{X}}(x,r_{x})$
to $y_{0}$ via a curve $\pi_{\mathbb{X}}(l_{y_{0},y})$, where $l_{y_{0},y}$
is the line segment joining $y_{0}$ and $y$ in $\mathbb{R}^{d}$.
The key part is to show that $\pi_{\mathbb{X}}(l_{y_{0},y})\subset\bigcap_{x\in I}\mathbb{B}_{\mathbb{X}}(x,r_{x})$,
or equivalently, to show that for all $x\in I$ and $t\in[0,1]$,
\[
\left\Vert x-\pi_{\mathbb{X}}(ty_{0}+(1-t)y)\right\Vert <r_{x}.
\]
This is implied from Lemma~\ref{lem:projreach_distance_projection_center},
as $\left\Vert x-y_{0}\right\Vert ,\left\Vert x-y\right\Vert <r_{x}$
implies that 
\[
\left\Vert x-\pi_{\mathbb{X}}(ty_{0}+(1-t)y)\right\Vert \leq\sqrt{t\left\Vert x-y_{0}\right\Vert ^{2}+(1-t)\left\Vert x-y\right\Vert ^{2}}<r_{x}.
\]
We restate Theorem~\ref{thm:nerve_contractible} and formally write
its proof below.

\textbf{Theorem~\ref{thm:nerve_contractible}.} \textit{
Let $\mathbb{X} \subset\mathbb{R}^{d}$ be a subset with reach $\tau>0$ and let $\mathcal{X}\subset\mathbb{R}^{d}$ be a set of points. Let $\{r_x > 0: x \in \mathcal{X}\}$ be a set of radii indexed by $x\in\mathcal{X}$. Then, if $r_x \leq \sqrt{\tau^{2}+(\tau-d_{\mathbb{X}}(x))^{2}}$ for all $x\in\mathcal{X}$, any nonempty intersection of restricted balls $\bigcap_{x\in I} \mathbb{B}_{\mathbb{X}}(x,r_{x})$ for $I\subset\mathcal{X}$ is contractible.}

\begin{proof}[Proof of Theorem~\ref{thm:nerve_contractible}]
	Fix $x\in I$ and let $y_{1},y_{2}\in\mathbb{B}_{\mathbb{X}}(x,r_{x})$.
	Let $l_{y_{1},y_{2}}:[0,1]\to\mathbb{B}_{\mathbb{R}^{d}}(x,r_{x})$
	with $l_{y_{1},y_{2}}(t)=ty_{1}+(1-t)y_{2}$ be the line segment joining
	$y_{1}$ and $y_{2}$. Then $d_{\mathbb{X}}(y_{i})=0$ and $\left\Vert x-y_{i}\right\Vert <r_{x}\leq\sqrt{(\tau-d_{\mathbb{X}}(x))^{2}+\tau^{2}}$
	for $i=1,2$, hence Lemma~\ref{lem:projreach_distance_projection_center}
	implies that $d_{\mathbb{X}}(l_{y_{1},y_{2}}(t))<\tau$ for all $t\in[0,1]$.
	Hence $\pi_{\mathbb{X}}(l_{y_{1},y_{2}}(t))\in\mathbb{X}$ is uniquely
	defined for each $t\in[0,1]$. And hence the curve $\gamma_{y_{1},y_{2}}:[0,1]\to\mathbb{X}$
	defined as $\gamma(t):=\pi_{\mathbb{X}}(l(t))$ is well-defined.
	
	Now we argue that $\gamma_{y_{1},y_{2}}(t)\in\mathbb{B}_{\mathbb{X}}(x,r_{x})$
	for all $t\in[0,1]$, in other words, we show $\left\Vert x-\gamma_{y_{1},y_{2}}(t)\right\Vert <r_{x}$.
	Again, applying Lemma~\ref{lem:projreach_distance_projection_center}
	gives the bound for$\left\Vert x-\gamma_{y_{1},y_{2}}(t)\right\Vert $
	as 
	\[
	\left\Vert x-\gamma_{y_{1},y_{2}}(t)\right\Vert \leq\sqrt{t\left\Vert x-y_{1}\right\Vert ^{2}+(1-t)\left\Vert x-y_{2}\right\Vert ^{2}}<r_{x}.
	\]
	Hence $\gamma_{y_{1},y_{2}}(t)\in\mathbb{B}_{\mathbb{X}}(x,r_{x})$
	for all $t\in[0,1]$.
	
	Now, fix $y_{0}\in\underset{x\in I}{\bigcap}\mathbb{B}_{\mathbb{X}}(x,r_{x})$,
	and define the map $F:\left(\underset{x\in I}{\bigcap}\mathbb{B}_{\mathbb{X}}(x,r_{x})\right)\times[0,1]\to\underset{x\in I}{\bigcap}\mathbb{B}_{\mathbb{X}}(x,r_{x})$
	as $F(y,t)=\gamma_{y_{0},y}(t)$. As we have shown above, $\gamma_{y_{0},y}(t)\in\underset{x\in I}{\bigcap}\mathbb{B}_{\mathbb{X}}(x,r_{x})$
	for all $t$, so $F$ is a well-defined. Also, for any $y,\tilde{y}\in\underset{x\in I}{\bigcap}\mathbb{B}_{\mathbb{X}}(x,r_{x})$
	and $t,\tilde{t}\in[0,1]$, Theorem \eqref{thm:background_reach_projection}
	(ii) implies
	\begin{align*}
	\left\Vert F(y,t)-F(\tilde{y},\tilde{t})\right\Vert  & =\left\Vert \gamma_{y_{0},y}(t)-\gamma_{y_{0},\tilde{y}}(\tilde{t})\right\Vert =\left\Vert \pi_{\mathbb{X}}(l_{y_{0},y}(t))-\pi_{\mathbb{X}}(l_{y_{0},\tilde{y}}(\tilde{t}))\right\Vert \\
	& \leq\frac{\tau\left\Vert l_{y_{0},y}(t)-l_{y_{0},\tilde{y}}(\tilde{t})\right\Vert }{\tau-\max\left\{ d_{\mathbb{X}}(l_{y_{0},y}(t)),d_{\mathbb{X}}(l_{y_{0},\tilde{y}}(\tilde{t}))\right\} }.
	\end{align*}
	Then $(\tilde{y},\tilde{t})\to(y,t)$ implies $l_{y_{0},\tilde{y}}(\tilde{t})=\tilde{t}y_{0}+(1-\tilde{t})\tilde{y}\to ty_{0}+(1-t)y=l_{y_{0},y}(t)$
	and $d_{\mathbb{X}}(l_{y_{0},\tilde{y}}(\tilde{t}))\to d_{\mathbb{X}}(l_{y_{0},y}(t))<\tau$,
	and hence $\left\Vert F(y,t)-F(\tilde{y},\tilde{t})\right\Vert \to0$,
	and hence $F$ is continuous.
	
	Now, for all $y\in\underset{x\in I}{\bigcap}\mathbb{B}_{\mathbb{X}}(x,r_{x})$,
	$F(y,0)=y$ and $F(y,1)=y_{0}$, and for all $t\in[0,1]$, $F(y_{0},t)=y_{0}$.
	Hence the intersection $\underset{x\in I}{\bigcap}\mathbb{B}_{\mathbb{X}}(x,r_{x})$
	deformation retracts to a point $\{y_{0}\}$. And hence the intersection
	$\underset{x\in I}{\bigcap}\mathbb{B}_{\mathbb{X}}(x,r_{x})$ is contractible.
\end{proof}

Then Corollary~\ref{cor:nerve_contractible_covering} is a direct
application of Theorem~\ref{thm:nerve_contractible} to Nerve Theorem
(Theorem~\ref{thm:background_nerve}).

\textbf{Corollary~\ref{cor:nerve_contractible_covering} (Nerve Theorem
	on the restricted balls).} \textit{Under the same condition of Theorem~\ref{thm:nerve_contractible},
	suppose $r_{x}\leq\sqrt{\tau^{2}+(\tau-d_{\mathbb{X}}(x))^{2}}$ for
	all $x\in\mathcal{X}$, then the union of restricted balls $\bigcup_{x\in\mathcal{X}}\mathbb{B}_{\mathbb{X}}(x,r_{x})$
	is homotopy equivalent to the restricted \v{C}ech complex $\textrm{\v{C}ech}_{\mathbb{X}}(\mathcal{X},r)$.
	If, in addition, the union of restricted balls covers the target space
	$\mathbb{X}$, that is, 
	\begin{equation}
	\mathbb{X}\subset\bigcup_{x\in\mathcal{X}}\mathbb{B}_{\mathbb{X}}(x,r_{x}),\label{eq:nerve_covering_proof}
	\end{equation}
	then $\mathbb{X}$ is homotopy equivalent to the restricted \v{C}ech
	complex $\textrm{\v{C}ech}_{\mathbb{X}}(\mathcal{X},r)$.}

\begin{proof}[Proof of Corollary~\ref{cor:nerve_contractible_covering}]
	
	Let $\mathcal{U}:=\left\{ \mathbb{B}_{\mathbb{X}}(x,r_{x})\right\} _{x\in\mathcal{X}}$ be the collection of balls.
	Then since $\mathbb{R}^{d}$ is paracompact, $\bigcup_{x\in\mathcal{X}}\mathbb{B}_{\mathbb{X}}(x,r_{x})\subset\mathbb{R}^{d}$
	is paracompact as well. And from Theorem~\ref{thm:nerve_contractible},
	any nonempty finite intersection of $\mathcal{U}$ is contractible.
	Hence from Nerve Theorem (Theorem~\ref{thm:background_nerve}), $\bigcup_{x\in\mathcal{X}}\mathbb{B}_{\mathbb{X}}(x,r_{x})$
	is homotopy equivalent to the nerve $\mathcal{N}(\mathcal{U})$, which
	is the restricted \v{C}ech complex $\textrm{\v{C}ech}_{\mathbb{X}}(\mathcal{X},r)$.
	Further, under the covering condition \eqref{eq:nerve_covering_proof},
	$\mathbb{X}=\bigcup_{x\in\mathcal{X}}\mathbb{B}_{\mathbb{X}}(x,r_{x})$,
	so $\mathbb{X}$ is homotopy equivalent to $\textrm{\v{C}ech}_{\mathbb{X}}(\mathcal{X},r)$
	as well.
	
\end{proof}

\section{Proofs for Section~\ref{sec:defretract}}
\label{app:defretract_proof}

This section provides the proofs for Section~\ref{sec:defretract},
in particular focuses on proving Theorem~\ref{thm:defretract_mureach}.
As in Section~\ref{subsec:background_reach}, we use the following
notation: For a closed set $\mathbb{X}$ and $x\in\mathbb{R}^{d}\backslash\mathbb{X}$,
let $\Gamma_{\mathbb{X}}(x)$ be the set of points in $\mathbb{X}$
closest to $x$, and $\Theta_{\mathbb{X}}(x)$ be the center of the
smallest ball enclosing $\Gamma_{\mathbb{X}}(x)$.

To prove the deformation retract, we proceed similar to \cite{Grove1993}.
We find a continuous vector field $W:\mathbb{X}^{r}\backslash\mathbb{X}\to\mathbb{R}^{d}$
that satisfies 
\begin{equation}
\sup_{x\in\mathbb{X}^{r}\backslash\mathbb{X}}\left\langle W(x),\nabla_{\mathbb{X}}(x)\right\rangle <0.\label{eq:defretract_innerproduct}
\end{equation}
And use the flow $\psi$ generated from $W$, that is, $\frac{d}{dt}\psi(x,t)=W(\psi(x,t))$
to find a homotopy map giving a deformation retract from $\mathbb{X}^{r}$
to $\mathbb{X}$. To use $\psi$ as a homotopy map, we show that the
distance function $d_{\mathbb{X}}$ decreases to $0$ in a finite
time on the integral curve $\psi^{x}(\cdot):=\psi(x,\cdot)$: sufficiently,
we show that 
\begin{equation}
\sup_{(x,s):\psi(x,s)\in\mathbb{X}^{r}\backslash\mathbb{X}}\limsup_{h\to0}\frac{d_{\mathbb{X}}(\psi(x,s+h))-d_{\mathbb{X}}(\psi(x,s))}{h}<0.\label{eq:defretract_derivative}
\end{equation}

To construct a vector field $W$ satisfying \eqref{eq:defretract_innerproduct},
the generalized gradient of the distance function $\nabla_{\mathbb{X}}$
is necessarily required not to change too much. Lemma~\ref{lem:defretract_innerproduct_bound}
asserts that the generalized gradient of the distance function is
not necessarily continuous but it also does not change too much in
terms of the inner product geometry.
\begin{lemma}
	
	\label{lem:defretract_innerproduct_bound}
	
	Let $\mathbb{X}$ be a closed set and $x\in\mathbb{R}^{d}\backslash\mathbb{X}$.
	Then 
	\[
	\liminf_{y\to x}\left\langle \nabla_{\mathbb{X}}(y),\nabla_{\mathbb{X}}(x)\right\rangle \geq\left\Vert \nabla_{\mathbb{X}}(x)\right\Vert ^{2}.
	\]
	
\end{lemma}

\begin{proof}
	
	Fix a small $\epsilon>0$. Let $\mathcal{F}_{\mathbb{X}}(x)$ be the
	radius of the ball enclosing $\Gamma_{\mathbb{X}}(x)$, and consider
	a compact set $K:=\partial\mathbb{B}_{\mathbb{R}^{d}}(x,d_{\mathbb{X}}(x))\backslash\mathbb{B}_{\mathbb{R}^{d}}(\Theta_{\mathbb{X}}(x),\mathcal{F}_{\mathbb{X}}(x)+\epsilon)$.
	Since $\partial\mathbb{B}_{\mathbb{R}^{d}}(x,d_{\mathbb{X}}(x))\cap\mathbb{X}\subset\mathbb{B}_{\mathbb{R}^{d}}(\Theta_{\mathbb{X}}(x),\mathcal{F}_{\mathbb{X}}(x)+\epsilon)$,
	$K$ does not intersect with $\mathbb{X}$, i.e. $K\subset\mathbb{R}^{d}\backslash\mathbb{X}$.
	Then since $\mathbb{R}^{d}\backslash\mathbb{X}$ is an open set, for
	each $q\in K$, there exists $r_{q}>0$ such that $\mathbb{B}_{\mathbb{R}^{d}}(q,2r_{q})\subset\mathbb{R}^{d}\backslash\mathbb{X}$.
	And $\left\{ \mathbb{B}_{\mathbb{R}^{d}}(q,r_{q})\right\} _{q\in K}$
	covers $K$, hence there exists a finite subcover $\left\{ \mathbb{B}_{\mathbb{R}^{d}}(q_{i},r_{q_{i}})\right\} _{i=1}^{m}$
	that covers $K$. Now, let $\Delta r:=\min\{r_{q_{i}}:1\leq i\leq m\}$,
	then $K^{\Delta r}\subset\bigcup_{i=1}^{m}\mathbb{B}_{\mathbb{R}^{d}}(q_{i},2r_{q_{i}})\subset\mathbb{R}^{d}\backslash\mathbb{X}$
	holds.
	
	Now, from the expansion 
	\[
	\left\Vert q-\Theta_{\mathbb{X}}(x)\right\Vert ^{2}=\left\Vert x-q\right\Vert ^{2}+\left\Vert x-\Theta_{\mathbb{X}}(x)\right\Vert ^{2}-2\left\langle x-q,x-\Theta_{\mathbb{X}}(x)\right\rangle ,
	\]
	it is implied that if $x\in\partial\mathbb{B}_{\mathbb{R}^{d}}(x,d_{\mathbb{X}}(x))$,
	then $x\in K$ is equivalent to 
	\begin{align*}
		\left\langle x-q,x-\Theta_{\mathbb{X}}(x)\right\rangle  & \leq\frac{1}{2}\left(d_{\mathbb{X}}(x)^{2}+d_{\mathbb{X}}(x)^{2}\left\Vert \nabla_{\mathbb{X}}(x)\right\Vert ^{2}-(\mathcal{F}_{\mathbb{X}}(x)+\epsilon)^{2}\right)\\
		& =d_{\mathbb{X}}(x)^{2}\left\Vert \nabla_{\mathbb{X}}(x)\right\Vert ^{2}-\epsilon\mathcal{F}_{\mathbb{X}}(x)-\frac{1}{2}\epsilon^{2},
	\end{align*}
	using $\mathcal{F}_{\mathbb{X}}(x)^{2}=d_{\mathbb{X}}(x)^{2}\left(1-\left\Vert \nabla_{\mathbb{X}}(x)\right\Vert ^{2}\right)$.
	Hence $q\in\mathbb{B}_{\mathbb{R}^{d}}(x,d_{\mathbb{X}}(x)+\Delta r)\backslash\mathbb{B}_{\mathbb{R}^{d}}(x,d_{\mathbb{X}}(x))$
	with $\left\langle x-q,x-\Theta_{\mathbb{X}}(x)\right\rangle \leq d_{\mathbb{X}}(x)^{2}\left\Vert \nabla_{\mathbb{X}}(x)\right\Vert ^{2}-\epsilon\mathcal{F}_{\mathbb{X}}(x)-\frac{1}{2}\epsilon^{2}$
	implies that $q\in K^{\Delta r}$. Then from $\mathbb{B}_{\mathbb{R}^{d}}(x,d_{\mathbb{X}}(x))\cap\mathbb{X}=\emptyset$
	and $K^{\Delta r}\cap\mathbb{X}=\emptyset$, for any $q\in\mathbb{B}_{\mathbb{R}^{d}}(x,d_{\mathbb{X}}(x)+\Delta r)\cap\mathbb{X}$,
	\begin{equation}
	\left\langle x-q,x-\Theta_{\mathbb{X}}(x)\right\rangle \geq d_{\mathbb{X}}(x)^{2}\left\Vert \nabla_{\mathbb{X}}(x)\right\Vert ^{2}-\epsilon\mathcal{F}_{\mathbb{X}}(x)-\frac{1}{2}\epsilon^{2}.\label{eq:defretract_innerproduct_bound_ball}
	\end{equation}
	Now, for $\delta\in(0,\frac{1}{2}\Delta r)$, suppose $y\in\mathbb{R}^{d}$
	with $\left\Vert y-x\right\Vert <\delta$. Then $d_{\mathbb{X}}(y)<d_{\mathbb{X}}(x)+\delta$,
	so $\Gamma_{\mathbb{X}}(y)\subset\overline{\mathbb{B}_{\mathbb{R}^{d}}(y,d_{\mathbb{X}}(y))}\subset\mathbb{B}_{\mathbb{R}^{d}}(y,d_{\mathbb{X}}(x)+\delta)\subset\mathbb{B}_{\mathbb{R}^{d}}(x,d_{\mathbb{X}}(x)+2\delta)\subset\mathbb{B}_{\mathbb{R}^{d}}(x,d_{\mathbb{X}}(x)+\Delta r).$
	Hence for any $q\in\Gamma_{\mathbb{X}}(y)$, \eqref{eq:defretract_innerproduct_bound_ball}
	implies that 
	\begin{align}
		\left\langle y-q,x-\Theta_{\mathbb{X}}(x)\right\rangle  & =\left\langle x-q,x-\Theta_{\mathbb{X}}(x)\right\rangle +\left\langle y-x,x-\Theta_{\mathbb{X}}(x)\right\rangle \nonumber \\
		& \geq\left(d_{\mathbb{X}}(x)^{2}\left\Vert \nabla_{\mathbb{X}}(x)\right\Vert ^{2}-\epsilon\mathcal{F}_{\mathbb{X}}(x)-\frac{1}{2}\epsilon^{2}\right)-\left\Vert y-x\right\Vert \left\Vert x-\Theta_{\mathbb{X}}(x)\right\Vert \nonumber \\
		& \geq d_{\mathbb{X}}(x)^{2}\left\Vert \nabla_{\mathbb{X}}(x)\right\Vert ^{2}-\epsilon\mathcal{F}_{\mathbb{X}}(x)-\frac{1}{2}\epsilon^{2}-\delta d_{\mathbb{X}}(x)\left\Vert \nabla_{\mathbb{X}}(x)\right\Vert .\label{eq:defretract_innerproduct_bound_point}
	\end{align}
	Since $\Theta_{\mathbb{X}}(y)$ is in the convex hull of $\Gamma_{\mathbb{X}}(y)$,
	there exists $q_{1},\ldots,q_{k}\in\Gamma_{\mathbb{X}}(y)$ and $\lambda_{1},\ldots,\lambda_{k}\in[0,1]$
	with $\sum_{i=1}^{k}\lambda_{i}q_{i}=\Theta_{\mathbb{X}}(y)$. Then \eqref{eq:defretract_innerproduct_bound_point}
	implies that 
	\[
	\left\langle y-\Theta_{\mathbb{X}}(y),x-\Theta_{\mathbb{X}}(x)\right\rangle \geq d_{\mathbb{X}}(x)^{2}\left\Vert \nabla_{\mathbb{X}}(x)\right\Vert ^{2}-\epsilon\mathcal{F}_{\mathbb{X}}(x)-\frac{1}{2}\epsilon^{2}-\delta d_{\mathbb{X}}(x)\left\Vert \nabla_{\mathbb{X}}(x)\right\Vert ,
	\]
	which implies that 
	\begin{align*}
		\left\langle \nabla_{\mathbb{X}}(y),\nabla_{\mathbb{X}}(x)\right\rangle  & \geq\frac{d_{\mathbb{X}}(x)^{2}\left\Vert \nabla_{\mathbb{X}}(x)\right\Vert ^{2}-\epsilon\mathcal{F}_{\mathbb{X}}(x)-\frac{1}{2}\epsilon^{2}-\delta d_{\mathbb{X}}(x)\left\Vert \nabla_{\mathbb{X}}(x)\right\Vert }{d_{\mathbb{X}}(x)d_{\mathbb{X}}(y)}\\
		& \geq\frac{d_{\mathbb{X}}(x)^{2}\left\Vert \nabla_{\mathbb{X}}(x)\right\Vert ^{2}-\epsilon\mathcal{F}_{\mathbb{X}}(x)-\frac{1}{2}\epsilon^{2}-\delta d_{\mathbb{X}}(x)\left\Vert \nabla_{\mathbb{X}}(x)\right\Vert }{d_{\mathbb{X}}(x)(d_{\mathbb{X}}(x)+\delta)}.
	\end{align*}
	Hence by sending $\delta\to0$, 
	\[
	\liminf_{y\to x}\left\langle \nabla_{\mathbb{X}}(y),\nabla_{\mathbb{X}}(x)\right\rangle \geq\left\Vert \nabla_{\mathbb{X}}(x)\right\Vert ^{2}-\frac{\epsilon\mathcal{F}_{\mathbb{X}}(x)+\frac{1}{2}\epsilon^{2}}{d_{\mathbb{X}}(x)^{2}}.
	\]
	And since the choice of $\epsilon$ was arbitrary small, 
	\[
	\liminf_{y\to x}\left\langle \nabla_{\mathbb{X}}(y),\nabla_{\mathbb{X}}(x)\right\rangle \geq\left\Vert \nabla_{\mathbb{X}}(x)\right\Vert ^{2}.
	\]
	
\end{proof}

To show \eqref{eq:defretract_derivative}, we generally show that
for $x,y\in\mathbb{R}^{d}$, the distance function is bounded as 
\[
d_{\mathbb{X}}(y)^{2}\leq d_{\mathbb{X}}(x)^{2}+\left\Vert x-y\right\Vert ^{2}+2d_{\mathbb{X}}(x)\left\langle y-x,\nabla_{\mathbb{X}}(x)\right\rangle .
\]
This implies that for $x\in\mathbb{R}^{d}\backslash\mathbb{X}$ and
any differentiable curve $\gamma$ with $\gamma(s)=x$, 
\[
\limsup_{h\to0}\frac{d_{\mathbb{X}}(\gamma(s+h))-d_{\mathbb{X}}(\gamma(s))}{h}\leq\left\langle \gamma'(s),\nabla_{\mathbb{X}}(x)\right\rangle ,
\]
and hence together with Lemma~\ref{lem:defretract_innerproduct_bound}
implies \eqref{eq:defretract_derivative}. Lemma~\ref{lem:defretract_distance_bound_derivative}
bounds the distance function values that are close to each other and
improves Lemma~\ref{lem:background_distance_concavity}.
\begin{lemma}
	\label{lem:defretract_distance_bound_derivative}
	
	Let $\mathbb{X}$ be a closed set and $x,y\in\mathbb{R}^{d}$. Then
	the distance $d_{\mathbb{X}}(y)$ is bounded as 
	\begin{equation}
	d_{\mathbb{X}}(y)^{2}\leq d_{\mathbb{X}}(x)^{2}+\left\Vert x-y\right\Vert ^{2}+2d_{\mathbb{X}}(x)\left\langle y-x,\nabla_{\mathbb{X}}(x)\right\rangle .\label{eq:defretract_distance_bound}
	\end{equation}
	In particular, suppose $x\in\mathbb{R}^{d}\backslash\mathbb{X}$ and
	let $\gamma$ be a differentiable curve with $\gamma(s)=x$, then
	\begin{equation}
	\limsup_{h\to0}\frac{d_{\mathbb{X}}(\gamma(s+h))-d_{\mathbb{X}}(\gamma(s))}{h}\leq\left\langle \gamma'(s),\nabla_{\mathbb{X}}(x)\right\rangle .\label{eq:defretract_distance_derivative}
	\end{equation}
	
\end{lemma}

\begin{proof}
	
	We first show \eqref{eq:defretract_distance_bound}. For any $q\in\Gamma_{\mathbb{X}}(x)$,
	\begin{align}
		\left\Vert y-q\right\Vert ^{2} & =\left\Vert x-q\right\Vert ^{2}+\left\Vert y-x\right\Vert ^{2}+2\left\langle y-x,x-q\right\rangle \nonumber \\
		& =\left\Vert x-q\right\Vert ^{2}+\left\Vert y-x\right\Vert ^{2}+2\left\langle y-x,x-\Theta_{\mathbb{X}}(x)\right\rangle +2\left\langle y-x,\Theta_{\mathbb{X}}(x)-q\right\rangle .\label{eq:defretract_distance_bound_expand}
	\end{align}
	Then since $\Theta_{\mathbb{X}}(x)$ is in the convex hull of $\Gamma_{\mathbb{X}}(x)$,
	there exists $q_{1},\ldots,q_{k}\in\Gamma_{\mathbb{X}}(x)$ and $\lambda_{1},\ldots,\lambda_{k}\in[0,1]$
	with $\sum_{i=1}^{k}\lambda_{i}q_{i}=\Theta_{\mathbb{X}}(x)$. Then
	$2\sum_{i=1}^{k}\lambda_{i}\left\langle y-x,\Theta_{\mathbb{X}}(x)-q_{i}\right\rangle =0$,
	so there exists $q_{j}\in\Gamma_{\mathbb{X}}(x)$ with 
	\[
	\left\langle y-x,\Theta_{\mathbb{X}}(x)-q_{j}\right\rangle \leq0,
	\]
	and applying this to \eqref{eq:defretract_distance_bound_expand}
	gives 
	\begin{align*}
		\left\Vert y-q_{j}\right\Vert ^{2} & \leq\left\Vert x-q_{j}\right\Vert ^{2}+\left\Vert y-x\right\Vert ^{2}+2\left\langle y-x,x-\Theta_{\mathbb{X}}(x)\right\rangle \\
		& =\left\Vert x-q_{j}\right\Vert ^{2}+\left\Vert y-x\right\Vert ^{2}+2d_{\mathbb{X}}(x)\left\langle y-x,\nabla_{\mathbb{X}}(x)\right\rangle .
	\end{align*}
	Then applying $d_{\mathbb{X}}(y)\leq\left\Vert y-q_{j}\right\Vert $
	and $d_{\mathbb{X}}(x)=\left\Vert x-q_{j}\right\Vert $ gives \eqref{eq:defretract_distance_bound}.
	
	For \eqref{eq:defretract_distance_derivative}, from $d_{\mathbb{X}}(\gamma(s))=d_{\mathbb{X}}(x)>0$, applying $\gamma(s+h)$
	and $\gamma(s)$ to \eqref{eq:defretract_distance_bound} and rearranging
	gives 
	\[
	\frac{d_{\mathbb{X}}(\gamma(s+h))-d_{\mathbb{X}}(\gamma(s))}{h}\leq\frac{\frac{1}{h}\left\Vert \gamma(s+h)-\gamma(s)\right\Vert ^{2}+2d_{\mathbb{X}}(\gamma(s))\left\langle \frac{1}{h}(\gamma(s+h)-\gamma(s)),\nabla_{\mathbb{X}}(x)\right\rangle }{d_{\mathbb{X}}(\gamma(s+h))+d_{\mathbb{X}}(\gamma(s))}.
	\]
	Then from $\gamma$ being differentiable, as $h\to0$, $\gamma(s+h)\to\gamma(s)$,
	$\frac{1}{h}\left\Vert \gamma(s+h)-\gamma(s)\right\Vert ^{2}\to0$,
	and $\frac{1}{h}(\gamma(s+h)-\gamma(s))\to\gamma'(s)$ hold, and \eqref{eq:defretract_distance_derivative}
	follows.
	
\end{proof}

Now, from Lemma~\ref{lem:defretract_innerproduct_bound} and Lemma
\ref{lem:defretract_distance_bound_derivative}, we can construct
a vector field $W$ with the desired properties, and build a homotopy
map from $W$. We restate Theorem~\ref{thm:defretract_mureach} and
formally write its proof below.

\textbf{Theorem~\ref{thm:defretract_mureach}.} \textit{Let $\mathbb{X}\subset\mathbb{R}^{d}$
	be a subset with positive $\mu$-reach $\tau^{\mu}>0$. For $r\leq\tau^{\mu}$,
	the $r$-offset $\mathbb{X}^{r}$ deformation retracts to $\mathbb{X}$.
	In particular, $\mathbb{X}$ and $\mathbb{X}^{r}$ are homotopy equivalent.}

\begin{proof}[Proof of Theorem~\ref{thm:defretract_mureach}]

	For each $x\in\mathbb{X}^{r}\backslash\mathbb{X}$, let $W_{x}:U_{x}\to\mathbb{R}^{d}$ be a vector field on a small neighborhood $U_{x}$ of $x$ defined as a constant $W_{x}(y)=-\nabla_{\mathbb{X}}(x)$. Then from Lemma~\ref{lem:defretract_innerproduct_bound} and the $\mu$-reach condition, $U_{x}$ can be chosen arbitrary small so that 
	\begin{equation}
	\left\langle W_{x}(y),\nabla_{\mathbb{X}}(y)\right\rangle \leq-\frac{\mu}{2}\left\Vert \nabla_{\mathbb{X}}(y)\right\Vert .\label{eq:defretract_mureach_innerproduct_pointwise}
	\end{equation}
	From the open cover $\{U_{x}\}\supset\mathbb{X}^{r}\backslash\mathbb{X}$,
	take a locally finite covering $\{U_{x_{i}}\}_{i\in\mathbb{N}}$,
	and take a smooth partition of unity $\{\rho_{i}\}_{i\in\mathbb{N}}$
	subordinate to $\{U_{x_{i}}\}_{i\in\mathbb{N}}$. Then we define a
	vector field $W:\mathbb{X}^{r}\backslash\mathbb{X}\to\mathbb{R}^{d}$
	as 
	\[
	W=\sum_{i\in\mathbb{N}}\rho_{i}W_{i}.
	\]
	Then note that for any $x\in\mathbb{X}^{r}\backslash\mathbb{X}$,
	\eqref{eq:defretract_mureach_innerproduct_pointwise} implies 
	\begin{equation}
	\left\langle W(x),\nabla_{\mathbb{X}}(x)\right\rangle =\left\langle \sum_{i\in\mathbb{N}}\rho_{i}W_{i}(x),\nabla_{\mathbb{X}}(x)\right\rangle \leq-\frac{\mu}{2}\left\Vert \nabla_{\mathbb{X}}(x)\right\Vert .\label{eq:defretract_mureach_innerproduct_sum}
	\end{equation}
	Then, since $\mathbb{X}^{r}\backslash\mathbb{X}$ is an open set in
	$\mathbb{R}^{d}$ and $W$ is smooth on $\mathbb{X}^{r}\backslash\mathbb{X}$,
	Fundamental Theorem of flows (Theorem~\ref{thm:background_fundamental_flow}
	implies that there exists a domain $\mathbb{D}\subset(\mathbb{X}^{r}\backslash\mathbb{X})\times[0,\infty)$
	and the unique smooth flow $\psi:\mathbb{D}\times[0,\infty)\to\mathbb{X}^{r}\backslash\mathbb{X}$,
	i.e. $\frac{d}{ds}\psi(x,s)=W(\psi(x,s))$. Such $\mathbb{D}$ is
	a maximal in that $\mathbb{D}$ is the set of points $(x,s)$ such
	that the integral curve $\psi^{x}(\cdot):=\psi(x,\cdot)$ exists on
	an interval containing $[0,s]$.
	
	Now, Lemma~\ref{lem:defretract_distance_bound_derivative} and \eqref{eq:defretract_mureach_innerproduct_sum}
	implies 
	\begin{align}
		\limsup_{h\to0}\frac{d_{\mathbb{X}}(\psi(x,s+h))-d_{\mathbb{X}}(\psi(x,s))}{h} & \leq\left\langle \nabla_{\mathbb{X}}(\psi(x,s)),\frac{d}{ds}\psi(x,s)\right\rangle \nonumber \\
		& =\left\langle \nabla_{\mathbb{X}}(\psi(x,s)),W(\psi(x,s))\right\rangle \nonumber \\
		& \leq-\frac{\mu}{2}\left\Vert \nabla_{\mathbb{X}}(x)\right\Vert \leq-\frac{\mu^{2}}{2}.\label{eq:defretract_mureach_integralcurve_diff_bound}
	\end{align}
	Hence $d_{\mathbb{X}}$ is strictly decreasing along the integral
	curve $\psi^{x}$.
	
	Let $\mathbb{D}_{x}:=\{s\in[0,\infty):(x,s)\in\mathbb{D}\}$, then
	$\mathbb{D}_{x}$ is a connected open set in $[0,\infty)$, and \eqref{eq:defretract_mureach_integralcurve_diff_bound}
	implies that $s_{x}:=\sup\mathbb{D}_{x}\leq\frac{d_{\mathbb{X}}(x)}{\mu^{2}/2}\leq\frac{2\tau^{\mu}}{\mu^{2}}<\infty$,
	so $\mathbb{D}_{x}=[0,s_{x})$. And since $\left\Vert (\psi^{x})'(s)\right\Vert =\left\Vert W(\psi^{x}(s))\right\Vert <\infty$
	for all $s\in[0,s_{x})$, $\psi^{x}([0,s_{x}))$ is a curve of finite
	length. Hence there exists a limit $\psi^{x}(s_{x}):=\lim_{s\to s_{x}}\psi^{x}(s)$,
	and we can extend $\psi^{x}$ continuously on $[0,\infty)$ by $\psi^{x}(s)=\psi^{x}(s_{x})$
	if $s\geq s_{x}$. Using this, we extend $\psi$ on $\mathbb{X}^{r}\times[0,\infty)$
	as 
	\[
	\psi(x,s)=\begin{cases}
	\psi^{x}(s_{x}), & \text{if }x\in\mathbb{X}^{r}\backslash\mathbb{X},\\
	x, & \text{if }x\in\mathbb{X}.
	\end{cases}
	\]
	Now we show that $\psi$ is continuous on $\mathbb{X}^{r}\times[0,\infty)$.
	If $(x_{0},s_{0})\in\mathbb{D}$, then $\psi$ is smooth on an open
	set $\mathbb{D}$, so $\psi$ is continuous at $(x_{0},s_{0})$. When
	$x_{0}\in\mathbb{X}^{r}\backslash\mathbb{X}$ and $s_{0}\notin\mathbb{D}_{x}$,
	let $\rho_{x}>0$ be small enough so that $B(x_{0},\rho_{x})\times\left[0,s_{x_{0}}-\frac{\mu^{2}}{8}\epsilon\right]\subset\mathbb{D}$
	and $\left|\psi(x,s)-\psi(x_{0},s)\right|\leq\frac{\mu^{2}\epsilon}{8}$
	for all $(x,s)\in B(x_{0},\rho_{x})\times\left[0,s_{x_{0}}-\frac{\mu^{2}}{8}\epsilon\right]$.
	Then $s_{0}\geq s_{x_{0}}$ holds. And for any $x\in B(x_{0},\rho_{x})$,
	\begin{align*}
		& d_{\mathbb{X}}\left(\psi\left(x,s_{x_{0}}-\frac{\mu^{2}}{8}\epsilon\right)\right)\\
		& \leq\left|d_{\mathbb{X}}\left(\psi\left(x,s_{x_{0}}-\frac{\mu^{2}}{8}\epsilon\right)\right)-d_{\mathbb{X}}\left(\psi\left(x_{0},s_{x_{0}}-\frac{\mu^{2}}{8}\epsilon\right)\right)\right|\\
		& \quad+\left|d_{\mathbb{X}}\left(\psi\left(x_{0},s_{x_{0}}-\frac{\mu^{2}}{8}\epsilon\right)\right)-d_{\mathbb{X}}(\psi(x_{0},s_{x_{0}}))\right|+d_{\mathbb{X}}(\psi(x_{0},s_{x_{0}}))\\
		& \leq\frac{\mu^{2}\epsilon}{4},
	\end{align*}
	then \eqref{eq:defretract_mureach_integralcurve_diff_bound} implies
	that $\left|s_{x}-\left(s_{x_{0}}-\frac{\mu^{2}}{8}\epsilon\right)\right|\leq\frac{\epsilon}{2}$.
	Hence for $(x,s)$ with $\left\Vert x-x_{0}\right\Vert <\rho_{x}$
	and $\left\Vert s-s_{0}\right\Vert <\frac{\mu^{2}}{8}\epsilon$, $s_{0}\geq s_{x_{0}}$
	and $s\geq s_{0}-\frac{\mu^{2}}{8}\epsilon\geq s_{x_{0}}-\frac{\mu^{2}}{8}\epsilon$
	imply 
	\begin{align*}
		& \left\Vert \psi(x,s)-\psi(x_{0},s_{0})\right\Vert =\left\Vert \psi(x,\min\{s,s_{x}\})-\psi(x_{0},s_{x_{0}})\right\Vert \\
		& \leq\left\Vert \psi(x,\min\{s,s_{x}\})-\psi\left(x,s_{x_{0}}-\frac{\mu^{2}}{8}\epsilon\right)\right\Vert +\left\Vert \psi\left(x,s_{x_{0}}-\frac{\mu^{2}}{8}\epsilon\right)-\psi\left(x_{0},s_{x_{0}}-\frac{\mu^{2}}{8}\epsilon\right)\right\Vert \\
		& \quad+\left\Vert \psi\left(x_{0},s_{x_{0}}-\frac{\mu^{2}}{8}\epsilon\right)-\psi\left(x_{0},s_{x_{0}}\right)\right\Vert \\
		& \leq\frac{\epsilon}{2}+\frac{\mu^{2}\epsilon}{8}+\frac{\mu^{2}\epsilon}{8}<\epsilon.
	\end{align*}
	Hence $\psi$ is continuous at $(x_{0},s_{0})$. When $x_{0}\in\mathbb{X}$,
	let $x\in B(x_{0},\frac{\mu^{2}}{4}\epsilon)$. Then $d_{\mathbb{X}}(x)<\frac{\mu^{2}}{4}\epsilon$,
	so $s_{x}<\frac{\epsilon}{2}$. Then for any $x\in B(x_{0},\frac{\mu^{2}}{4}\epsilon)$
	and for any $s\geq0$, 
	\begin{align*}
		\left\Vert \psi(x,s)-\psi(x_{0},s)\right\Vert  & =\left\Vert \psi(x,s)-x_{0}\right\Vert =\left\Vert \psi(x,s)-x\right\Vert +\left\Vert x-x_{0}\right\Vert \\
		& <\frac{\epsilon}{2}+\frac{\mu^{2}\epsilon}{4}\leq\epsilon.
	\end{align*}
	Hence $\psi$ is continuous at $(x_{0},s_{0})$.
	
	Now, we define the deformation retract $H:\mathbb{X}^{r}\times[0,1]\to\mathbb{X}^{r}$
	as $H(x,t)=\psi\left(x,\frac{2}{\mu^{2}}t\right)$. Then, $H(x,0)=x$
	for all $x\in\mathbb{X}^{r}$ and $H(x,t)=x$ for all $x\in\mathbb{X}$.
	Also, since $s_{x}\leq\frac{2}{\mu^{2}}$ for all $x$, $H(x,1)=\psi\left(x,\frac{2}{\mu^{2}}\right)\in\mathbb{X}$
	for all $x\in\mathbb{X}^{r}$. Also, $H$ is continuous since $\psi$
	is continuous. Hence $H$ gives the deformation retract from $\mathbb{X}^{r}$
	to $\mathbb{X}$.
\end{proof}

For showing Lemma~\ref{lem:defretract_complement}, we directly construct
a map $\psi:\mathbb{X}^{\complement}\to(\mathbb{X}^{r})^{\complement}$ as 
\[
\psi(x)=\begin{cases}
x+\frac{\left(r-d_{\mathbb{X}}(x)\right)}{d_{\mathbb{X}}(x)}(x-\pi_{\mathbb{X}}(x)), & \text{if }x\in\mathbb{X}^{r}\backslash\mathbb{X},\\
x, & \text{if }x\notin\mathbb{X}^{r},
\end{cases}
\]
and show that $\mathbb{X}$ deformation retracts to $(\mathbb{X}^{r})^{\complement}$ using the map $\psi$.
We restate Lemma~\ref{lem:defretract_complement} and formally write
its proof below.

\textbf{Lemma~\ref{lem:defretract_complement}.} \textit{Let $\mathbb{X}\subset\mathbb{R}^{d}$
	be a subset with positive reach $\tau>0$. For $r\leq\tau$, $\mathbb{X}^{\complement}$
	deformation retracts to $(\mathbb{X}^{r})^{\complement}$. In particular,
	$\mathbb{X}^{\complement}$ and $(\mathbb{X}^{r})^{\complement}$
	are homotopy equivalent.}

\begin{proof}[Proof of Lemma~\ref{lem:defretract_complement}]

For $x\in\mathbb{X}^{\complement}$, let the function $\psi:\mathbb{X}^{\complement}\to(\mathbb{X}^{r})^{\complement}$
be 
\[
\psi(x)=\begin{cases}
x+\frac{\left(r-d_{\mathbb{X}}(x)\right)}{d_{\mathbb{X}}(x)}(x-\pi_{\mathbb{X}}(x)), & \text{if }x\in\mathbb{X}^{r}\backslash\mathbb{X},\\
x, & \text{if }x\notin\mathbb{X}^{r}.
\end{cases}
\]
From $r\leq\tau$, this function is well defined, and $\psi$ is continuous
on $\mathbb{X}^{\complement}$. Also, if $x\notin\mathbb{X}^{r}$
then $\psi(x)=x\notin\mathbb{X}^{r}$ and if $x\in\mathbb{X}^{r}\backslash\mathbb{X}$,
\[
\left\Vert \psi(x)-\pi_{\mathbb{X}}(x)\right\Vert =\left(1+\frac{r-d_{\mathbb{X}}(x)}{d_{\mathbb{X}}(x)}\right)\left\Vert x-\pi_{\mathbb{X}}(x)\right\Vert =r,
\]
so $\psi(x)\notin\mathbb{X}^{r}$. Hence in any case, $\psi(x)\in(\mathbb{X}^{r})^{\complement}$.
And if $x\in(\mathbb{X}^{r})^{\complement}$, then $\psi(x)=x$ on
$(\mathbb{X}^{r})^{\complement}$.

Now, we define the deformation retract $H:\mathbb{X}^{\complement}\times[0,1]\to\mathbb{X}^{\complement}$
as $H(x,t)=(1-t)x+t\psi(x)$. Then, $H(x,0)=x$ for all $x\in\mathbb{X}^{\complement}$,
$H(x,t)=x$ for all $x\in(\mathbb{X}^{r})^{\complement}$, and $H(x,1)\in(\mathbb{X}^{r})^{\complement}$
for all $x\in\mathbb{X}^{r}$. Also, $H$ is continuous since $\psi$
is continuous. Hence $H$ gives the deformation retract from $\mathbb{X}^{\complement}$
to $(\mathbb{X}^{r})^{\complement}$.

\end{proof}

\textbf{Corollary~\ref{cor:defretract_doubleoffset}.} \textit{Let
	$\mathbb{X}\subset\mathbb{R}^{d}$ be a subset with positive $\mu$-reach
	$\tau^{\mu}>0$. For $r,s>0$ with $s\leq r$, let $\mathbb{X}^{r,s}:=(((\mathbb{X}^{r})^{\complement})^{s})^{\complement}$
	be the double offset of $\mathbb{X}$. If $r<\tau^{\mu}$ and $s<\mu r$,
	then $\mathbb{X}^{r,s}$ and $\mathbb{X}$ are homotopy equivalent, and
	the reach of $\mathbb{X}^{r,s}$ is greater than or equal to $s$,
	that is, $\tau_{\mathbb{X}^{r,s}}\geq s$.}

\begin{proof}[Proof of Corollary~\ref{cor:defretract_doubleoffset}]

	Since $\mathbb{X}$ has a positive $\mu$-reach $\tau^{\mu}$ and
	$r<\tau^{\mu}$, Theorem~\ref{thm:defretract_mureach} implies that
	$\mathbb{X}$ and $\mathbb{X}^{r}$ are homotopy equivalent. Then,
	by $r<\tau^{\mu}$ and Theorem~\ref{thm:background_reach_offset},
	\[
	\tau_{(\mathbb{X}^{r})^{\complement}}\geq\mu r.
	\]
	Then, from Lemma~\ref{lem:defretract_complement} and $s<\mu r$,
	$\mathbb{X}^{r}=((\mathbb{X}^{r})^{\complement})^{\complement}$ and
	$\mathbb{X}^{r,s}=(((\mathbb{X}^{r})^{\complement})^{s})^{\complement}$
	are homotopy equivalent. And hence, $\mathbb{X}$ and $\mathbb{X}^{r,s}$
	are homotopy equivalent as well. Also, again by $s<\mu r$ and Theorem
	\ref{thm:background_reach_offset}, 
	\[
	\tau_{\mathbb{X}^{r,s}}=\tau_{(((\mathbb{X}^{r})^{\complement})^{s})^{\complement}}\geq s.
	\]

\end{proof}

\section{Proofs for Section~\ref{sec:homotopy}}
\label{app:homotopy_proof}

This section provides the proofs for Section~\ref{sec:homotopy}, and in particular, focuses on proving Theorem~\ref{thm:homotopy_cech}
and \ref{thm:homotopy_rips}. For this section, we define some notions
for (weighted) barycenter and radius. For a given set of radii $r=\{r_{x}:x\in\mathcal{X}\}$,
Let (weighted) barycenter ${\rm bc}_{r}:2^{\mathcal{X}}\to\mathbb{R}$
and (weighted) radius ${\rm Rad}_{r}:2^{\mathcal{X}}\to\mathbb{R}$
be functions defined on a simplex defined as 
\begin{align}
{\rm bc}_{r}(\sigma) & =\arg\min_{y\in\mathbb{R}^{d}}\max_{x\in\sigma}\frac{\left\Vert x-y\right\Vert }{r_{x}},\nonumber \\
{\rm Rad}_{r}(\sigma) & =\max_{x\in\sigma}\left\Vert x-{\rm bc}_{r}(\sigma)\right\Vert .\label{eq:homotopy_barycenter_radius_r}
\end{align}
And we drop the notation $r$ when $r_{x}$'s are all equal, i.e.
\begin{align}
{\rm bc}(\sigma) & =\arg\min_{y\in\mathbb{R}^{d}}\max_{x\in\sigma}\left\Vert x-y\right\Vert ,\nonumber \\
{\rm Rad}(\sigma) & =\max_{x\in\sigma}\left\Vert x-{\rm bc}(\sigma)\right\Vert .\label{eq:homotopy_barycenter_radius}
\end{align}
${\rm bc}(\sigma)$ is the usual center of the smallest enclosing
ball, and ${\rm Rad}(\sigma)$ is its radius. Also, note that ${\rm bc}_{r}(\sigma)\in\bigcap_{x\in\sigma}\mathbb{B}_{\mathbb{R}^{d}}(x,r_{x})$
if and only if $\min_{y\in\mathbb{R}^{d}}\max_{x\in\sigma}\frac{\left\Vert x-y\right\Vert }{r_{x}}<1$.

We first extend the interleaving relationship of the ambient \v{C}ech
complex and the Vietoris-Rips complex in \eqref{eq:background_interleaving_rips_ambientcech}
to the different radii case in Lemma~\ref{lem:homotopy_interleaving_rips_ambientcech}.
The proof is similar to the proof of Theorem 2.5 in \cite{deSilvaG2007}
but modified to adapt to different radii case.

\textbf{Lemma~\ref{lem:homotopy_interleaving_rips_ambientcech}.} \textit{
	Let $\mathcal{X}\subset\mathbb{R}^{d}$ be a set of points. Let $\{r_{x}>0:x\in\mathcal{X}\}$
	be a set of radii indexed by $x\in\mathcal{X}$. Then, 
	\[
	\textrm{\v{C}ech}_{\mathbb{R}^{d}}(\mathcal{X},r)\subset\textrm{Rips}(\mathcal{X},r)\subset\textrm{\v{C}ech}_{\mathbb{R}^{d}}\left(\mathcal{X},\sqrt{\frac{2d}{d+1}}r\right).
	\]
}

\begin{proof}[Proof of Lemma~\ref{lem:homotopy_interleaving_rips_ambientcech}]
	
	The first inclusion $\textrm{\v{C}ech}_{\mathbb{R}^{d}}(\mathcal{X},r)\subset\textrm{Rips}(\mathcal{X},r)$
	is trivial.
	
	For the second inclusion, the equivalent statement is as follows:
	if $\sigma=\{x_{0},\ldots,x_{k}\}\subset\mathbb{R}^{d}$ satisfies
	that $\left\Vert x_{i}-x_{j}\right\Vert <r_{i}+r_{j}$ for all $0\leq i,j\leq k$,
	then the intersection of the balls $\bigcap_{i=0}^{k}\mathbb{B}_{\mathbb{R}^{d}}\left(x_{i},\sqrt{\frac{2d}{d+1}}r_{x_{i}}\right)$
	is nonempty.
	
	We first prove this for the case $k\leq d$. Consider a function $f:\mathbb{R}^{d}\to\mathbb{R}$
	defined as 
	\[
	f(y)=\max_{0\leq i\leq k}\frac{\left\Vert x_{i}-y\right\Vert }{r_{x_{i}}}.
	\]
	This is continuous, and $f(y)\to\infty$ as $y\to\infty$, so $f$
	has a global minimum $f(y_{0})$ at $y_{0}:={\rm bc}_{r}(\sigma)$.
	Then $y_{0}\in co(\sigma)$ where $co(\sigma)$ is the convex hull
	of $\sigma$, since otherwise $\left\Vert x_{i}-\pi_{co(\sigma)}(y_{0})\right\Vert <\left\Vert x_{i}-y_{0}\right\Vert $
	for each $x_{i}\in\mathcal{X}$, contradicting the minimality of $y_{0}$.
	
	Let $\hat{x}_{i}:=x_{i}-y_{0}$ be the translated $x_{i}$'s. We can
	find a convex combination $(a_{0}/r_{x_{0}})\hat{x}_{0}+\cdots+(a_{k}/r_{x_{k}})\hat{x}_{k}=0$
	for some $j\leq k$, after relabeling so that $a_{0}>0$ is the largest
	among $a_{0},\ldots,a_{k}$ and all $a_{i}$'s are nonnegative. Then
	$-\hat{x}_{0}=\sum_{i=1}^{k}\frac{a_{i}/r_{x_{i}}}{a_{0}/r_{x_{0}}}\hat{x}_{i}$,
	and so 
	\[
	-r_{x_{0}}^{2}f(y_{0})^{2}=-\left\Vert \hat{x}_{0}\right\Vert ^{2}=\sum_{i=1}^{k}\frac{a_{i}/r_{x_{i}}}{a_{0}/r_{x_{0}}}\left\langle \hat{x}_{0},\hat{x}_{i}\right\rangle .
	\]
	Then at least one $i$ should satisfy $(a_{i}/a_{0})\left\langle \hat{x}_{0},\hat{x}_{i}\right\rangle \leq-\frac{r_{x_{0}}r_{x_{i}}f(y_{0})^{2}}{k},$which
	can be weakened to $\left\langle \hat{x}_{0},\hat{x}_{i}\right\rangle \leq-\frac{r_{x_{0}}r_{x_{i}}f(y_{0})^{2}}{d}$.
	Putting these together gives 
	\begin{align*}
	f(y_{0})^{2}(r_{x_{0}}^{2}+\frac{2}{d}r_{x_{0}}r_{x_{i}}+r_{x_{i}}^{2}) & \leq\left\Vert \hat{x}_{0}\right\Vert ^{2}-2\left\langle \hat{x}_{0},\hat{x}_{i}\right\rangle +\left\Vert \hat{x}_{i}\right\Vert ^{2}\\
	& =\left\Vert \hat{x}_{0}-\hat{x}_{i}\right\Vert ^{2}=(r_{x_{0}}+r_{x_{i}})^{2}.
	\end{align*}
	Then from AM-GM inequality, 
	\begin{align*}
	r_{x_{0}}^{2}+\frac{2}{d}r_{x_{0}}r_{x_{i}}+r_{x_{i}}^{2} & =\frac{d-1}{2d}(r_{x_{0}}^{2}+r_{x_{i}}^{2})+\left(\frac{d+1}{2d}(r_{x_{0}}^{2}+r_{x_{i}}^{2})+\frac{2}{d}r_{x_{0}}r_{x_{i}}\right)\\
	& \geq\frac{d-1}{d}r_{x_{0}}r_{x_{i}}+\left(\frac{d+1}{2d}(r_{x_{0}}^{2}+r_{x_{i}}^{2})+\frac{2}{d}r_{x_{0}}r_{x_{i}}\right)\\
	& =\frac{d+1}{2d}(r_{x_{0}}+r_{x_{i}})^{2}.
	\end{align*}
	Hence combining these gives 
	\[
	f(y_{0})^{2}\frac{d+1}{2d}(r_{x_{0}}+r_{x_{i}})^{2}\leq f(y_{0})^{2}(r_{x_{0}}^{2}+\frac{2}{d}r_{x_{0}}r_{x_{i}}+r_{x_{i}}^{2})\leq(r_{x_{0}}+r_{x_{i}})^{2},
	\]
	and hence 
	\[
	f(y_{0})\leq\sqrt{\frac{2d}{d+1}}.
	\]
	Therefore $y_{0}\in\bigcap_{i=0}^{k}\mathbb{B}_{\mathbb{R}^{d}}\left(x_{i},\sqrt{\frac{2d}{d+1}}r_{x_{i}}\right)$,
	i.e. $\bigcap_{i=0}^{k}\mathbb{B}_{\mathbb{R}^{d}}\left(x_{i},\sqrt{\frac{2d}{d+1}}r_{x_{i}}\right)$
	is nonempty.
	
	For the case $k>d$, the result follows by the Helly's theorem \cite{Eckhoff1993}.
	This asserts that a collection of $k\geq d+2$ convex sets in $\mathbb{R}^{d}$
	has a nonempty intersection if and only if it is true for each subcollection
	of size $d+1$. Now, for any $k+1$ points $\sigma=\{x_{0},\ldots,x_{k}\}$
	with $k\geq d+1$, any subset $\varsigma\subset\sigma$ with $|\varsigma|=d+1$
	satisfy that $\bigcap_{x\in\varsigma}\mathbb{B}_{\mathbb{R}^{d}}\left(x,\sqrt{\frac{2d}{d+1}}r_{x}\right)$
	is nonempty. Hence by Helly's theorem, $\bigcap_{i=0}^{k}\mathbb{B}_{\mathbb{R}^{d}}\left(x_{i},\sqrt{\frac{2d}{d+1}}r_{x_{i}}\right)$
	is nonempty as well.
\end{proof}

We also set up the interleaving relationship between the restricted
\v{C}ech complex $\textrm{\v{C}ech}_{\mathbb{X}}(\mathcal{X},r)$
and the ambient \v{C}ech complex $\textrm{\v{C}ech}_{\mathbb{R}^{d}}(\mathcal{X},r)$
in Lemma~\ref{lem:homotopy_interleaving_ambientcech_restrictedcech}.
The proof is a direct application of Lemma~\ref{lem:projreach_distance_projection_general}.
We restate Lemma~\ref{lem:homotopy_interleaving_ambientcech_restrictedcech}
and formally write its proof.

\textbf{Lemma~\ref{lem:homotopy_interleaving_ambientcech_restrictedcech}.} \textit{
	Let $\mathbb{X}\subset\mathbb{R}^{d}$ be a subset with reach $\tau>0$
	and let $\mathcal{X}\subset\mathbb{R}^{d}$ be a set of points. Let
	$\{r_{x}>0:x\in\mathcal{X}\}$ be a set of radii indexed by $x\in\mathcal{X}$.
	Then,
	\[
	\textrm{\v{C}ech}_{\mathbb{X}}(\mathcal{X},r)\subset\textrm{\v{C}ech}_{\mathbb{R}^{d}}(\mathcal{X},r)\subset\textrm{\v{C}ech}_{\mathbb{X}}(\mathcal{X},r'),
	\]
	where $r'=\{r'_{x}>0:x\in\mathcal{X}\}$ with
	\[
	r'_{x}=\sqrt{\frac{2\tau\left(r_{x}^{2}+d_{\mathbb{X}}(x)\left(2\tau-d_{\mathbb{X}}(x)\right)\right)}{\tau+\sqrt{\tau^{2}-\left(r_{x}^{2}+d_{\mathbb{X}}(x)\left(2\tau-d_{\mathbb{X}}(x)\right)\right)}}-d_{\mathbb{X}}(x)\left(2\tau-d_{\mathbb{X}}(x)\right)}.
	\]
	Equivalently, 
	\[
	\textrm{\v{C}ech}_{\mathbb{R}^{d}}(\mathcal{X},r'')\subset\textrm{\v{C}ech}_{\mathbb{X}}(\mathcal{X},r)\subset\textrm{\v{C}ech}_{\mathbb{R}^{d}}(\mathcal{X},r),
	\]
	where $r''=\{r''_{x}>0:x\in\mathcal{X}\}$ with 
	\[
	r''_{x}=\sqrt{\tau^{2}-d_{\mathbb{X}}(x)(2\tau-d_{\mathbb{X}}(x))-\frac{(2\tau^{2}-r_{x}^{2}-d_{\mathbb{X}}(x)(2\tau-d_{\mathbb{X}}(x)))^{2}}{4\tau^{2}}}.
	\]
}

\begin{proof}[Proof of Lemma~\ref{lem:homotopy_interleaving_ambientcech_restrictedcech}]
	
	For $\textrm{\v{C}ech}_{\mathbb{X}}(\mathcal{X},r)\subset\textrm{\v{C}ech}_{\mathbb{R}^{d}}(\mathcal{X},r)$,
	this is implied from $\mathbb{B}_{\mathbb{X}}(x,r_{x})\subset\mathbb{B}_{\mathbb{R}^{d}}(x,r_{x})$.
	
	For $\textrm{\v{C}ech}_{\mathbb{R}^{d}}(\mathcal{X},r)\subset\textrm{\v{C}ech}_{\mathbb{X}}(\mathcal{X},r')$,
	let $[x_{1},\ldots,x_{k}]\in\textrm{\v{C}ech}_{\mathbb{R}^{d}}(\mathcal{X},r)$,
	then there exists $\lambda_{1},\ldots,\lambda_{k}\in[0,1]$ with $\sum\lambda_{i}=1$
	such that 
	\[
	u:=\sum\lambda_{i}x_{i}\in\bigcap_{i=1}^{k}\mathbb{B}_{\mathbb{R}^{d}}(x_{i},r_{x_{i}}).
	\]
	Then from Lemma~\ref{lem:projreach_distance_projection_general}, 
	\begin{align*}
	& \left\Vert x_{i}-\pi_{\mathbb{X}}(u)\right\Vert \\
	& \leq\sqrt{\frac{2\tau\left(\left\Vert u-x_{i}\right\Vert ^{2}+d_{\mathbb{X}}(x_{i})\left(2\tau-d_{\mathbb{X}}(x_{i})\right)\right)}{\tau+\sqrt{\tau^{2}-\left(\left\Vert u-x_{i}\right\Vert ^{2}+d_{\mathbb{X}}(x_{i})\left(2\tau-d_{\mathbb{X}}(x_{i})\right)\right)}}-d_{\mathbb{X}}(x_{i})\left(2\tau-d_{\mathbb{X}}(x_{i})\right)}\\
	& <\sqrt{\frac{2\tau\left(r_{x_{i}}^{2}+d_{\mathbb{X}}(x_{i})\left(2\tau-d_{\mathbb{X}}(x_{i})\right)\right)}{\tau+\sqrt{\tau^{2}-\left(r_{x_{i}}^{2}+d_{\mathbb{X}}(x_{i})\left(2\tau-d_{\mathbb{X}}(x_{i})\right)\right)}}-d_{\mathbb{X}}(x_{i})\left(2\tau-d_{\mathbb{X}}(x_{i})\right)}:=r'_{x_{i}}.
	\end{align*}
	And hence 
	\[
	\pi_{\mathbb{X}}(u)\in\bigcap_{j=1}^{k}\mathbb{B}_{\mathbb{X}}\left(x_{j},r'_{x_{j}}\right).
	\]
	Therefore, $[x_{1},\ldots,x_{k}]\in\textrm{\v{C}ech}_{\mathbb{X}}(\mathcal{X},r')$
	as well, and hence $\textrm{\v{C}ech}_{\mathbb{R}^{d}}(\mathcal{X},r)\subset\textrm{\v{C}ech}_{\mathbb{X}}(\mathcal{X},r')$
	holds.
	
\end{proof}

Then, Corollary~\ref{cor:homotopy_interleaving_rips_restrictedcech}
is a combination of Lemma~\ref{lem:homotopy_interleaving_rips_ambientcech}
and \ref{lem:homotopy_interleaving_ambientcech_restrictedcech}. We
restate Corollary~\ref{cor:homotopy_interleaving_rips_restrictedcech}
and formally write its proof.

\textbf{Corollary~\ref{cor:homotopy_interleaving_rips_restrictedcech}.}
\textit{Let $\mathbb{X}\subset\mathbb{R}^{d}$ be a subset with reach
	$\tau>0$ and let $\mathcal{X}\subset\mathbb{R}^{d}$ be a set of
	points. Let $r=\{r_{x}>0:x\in\mathcal{X}\}$ be a set of radii indexed
	by $x\in\mathcal{X}$. Then, 
	\[
	\textrm{\v{C}ech}_{\mathbb{X}}(\mathcal{X},r)\subset\textrm{Rips}(\mathcal{X},r)\subset\textrm{\v{C}ech}_{\mathbb{X}}(\mathcal{X},r'''),
	\]
	where $r'''=\{r'''_{x}>0:x\in\mathcal{X}\}$ with 
	\[
	r'''_{x}=\sqrt{\frac{2\tau\left(\frac{2d}{d+1}r_{x}^{2}+d_{\mathbb{X}}(x)\left(2\tau-d_{\mathbb{X}}(x)\right)\right)}{\tau+\sqrt{\tau^{2}-\left(\frac{2d}{d+1}r_{x}^{2}+d_{\mathbb{X}}(x)\left(2\tau-d_{\mathbb{X}}(x)\right)\right)}}-d_{\mathbb{X}}(x)\left(2\tau-d_{\mathbb{X}}(x)\right)}.
	\]
}

\begin{proof}[Proof of Corollary~\ref{cor:homotopy_interleaving_rips_restrictedcech}]
	For $\textrm{\v{C}ech}_{\mathbb{X}}(\mathcal{X},r)\subset\textrm{Rips}(\mathcal{X},r)$,
	this is implied by Lemma~\ref{lem:homotopy_interleaving_rips_ambientcech}
	and \ref{lem:homotopy_interleaving_ambientcech_restrictedcech} as
	\[
	\textrm{\v{C}ech}_{\mathbb{X}}(\mathcal{X},r)\subset\textrm{\v{C}ech}_{\mathbb{R}^{d}}(\mathcal{X},r)\subset\textrm{Rips}(\mathcal{X},r).
	\]
	For $\textrm{Rips}(\mathcal{X},r)\subset\textrm{\v{C}ech}_{\mathbb{X}}(\mathcal{X},r''')$,
	this is again implied by Lemma~\ref{lem:homotopy_interleaving_rips_ambientcech}
	and \ref{lem:homotopy_interleaving_ambientcech_restrictedcech} as
	\[
	\textrm{Rips}(\mathcal{X},r)\subset\textrm{\v{C}ech}_{\mathbb{R}^{d}}\left(\mathcal{X},\sqrt{\frac{2d}{d+1}}r\right)\subset\textrm{\v{C}ech}_{\mathbb{X}}(\mathcal{X},r''').
	\]
\end{proof}

To show Theorem~\ref{thm:homotopy_cech} and \ref{thm:homotopy_rips},
we consider a generalized commutative diagram of \eqref{eq:homotopy_diagram}.
Let $\mathbb{X}$ be a paracompact space, let $\mathcal{U}=\{U_{i}\}_{i\in I}$,
$\mathcal{U}'=\{U'_{i}\}_{i\in I'}$, be good covers of $\mathbb{X}$,
and let $\mathcal{S}$ be a simplicial complex satisfying $\mathcal{N}\mathcal{U}\subset\mathcal{S}\subset\mathcal{N}\mathcal{U}'$,
so that the following diagram commutes: 
\begin{equation}
\xymatrix{ & \mathbb{X}\ar[dl]_{\phi}\\
	\mathcal{NU}\ar[dr]_{\imath_{\mathcal{N}\mathcal{U}\to\mathcal{S}}}\ar[rr]^{\imath_{\mathcal{N}\mathcal{U}\to\mathcal{N\mathcal{U}}'}} &  & \mathcal{NU}'\ar[ul]_{\psi'}\\
	& \mathcal{S}\ar[ur]_{\imath_{\mathcal{S}\to\mathcal{NU}'}}
}
.\label{eq:homotopy_simplexnerve_initial}
\end{equation}
In Theorem~\ref{thm:homotopy_cech}, $\mathcal{N}\mathcal{U}=\textrm{\v{C}ech}_{\mathbb{X}}(\mathcal{X},r)$,
$\mathcal{N}\mathcal{U}'=\textrm{\v{C}ech}_{\mathbb{X}}(\mathcal{X},r')$,
and $\mathcal{S}=\textrm{\v{C}ech}_{\mathbb{R}^{d}}(\mathcal{X},r)$,
and in Theorem~\ref{thm:homotopy_rips}, $\mathcal{N}\mathcal{U}=\textrm{\v{C}ech}_{\mathbb{X}}(\mathcal{X},r)$,
$\mathcal{N}\mathcal{U}'=\textrm{\v{C}ech}_{\mathbb{X}}(\mathcal{X},r''')$,
and $\mathcal{S}=\textrm{Rips}(\mathcal{X},r)$. Then our goal is
to construct a homotopy equivalence between $\mathbb{X}$ and $\mathcal{S}$.
The maps between $\mathbb{X}$ and $\mathcal{S}$ are naturally defined
as $\Phi:\mathbb{X}\to\mathcal{S}$ and $\Psi:\mathcal{S}\to\mathbb{X}$
by $\Phi=\imath_{\mathcal{N}\mathcal{U}\to\mathcal{S}}\circ\phi$
and $\Psi=\psi'\circ\imath_{\mathcal{S}\to\mathcal{NU}'}$. Then,
$\Psi\circ\Phi\simeq id_{\mathbb{X}}$ naturally follows by 
the commutative diagram in \eqref{eq:homotopy_simplexnerve_initial}.
However, $\Phi\circ\Psi\simeq id_{\mathcal{S}}$ needs an additional
condition. We suppose the existence of $\rho:\mathcal{S}\to\mathcal{N}\mathcal{U}$
such that $\imath_{\mathcal{N}\mathcal{U}\to\mathcal{S}}\circ\rho:S\to S$
is homotopy equivalent to $id_{\mathcal{S}}:\mathcal{S}\to\mathcal{S}$,
so that the following diagram commutes: 
\begin{equation}
\xymatrix{ & \mathbb{X}\ar[dl]_{\phi}\\
	\mathcal{NU}\ar@/_/[dr]_{\imath_{\mathcal{N}\mathcal{U}\to\mathcal{S}}}\ar[rr]^{\imath_{\mathcal{N}\mathcal{U}\to\mathcal{N\mathcal{U}}'}} &  & \mathcal{NU}'\ar[ul]_{\psi'}\\
	& \mathcal{S}\ar@/_/[ul]_{\rho}\ar[ur]_{\imath_{\mathcal{S}\to\mathcal{NU}'}}
}
.\label{eq:homotopy_simplexnerve_withcondition}
\end{equation}
Then $\Phi\circ\Psi\simeq id_{\mathcal{S}}$ can be deduced from this
commutative diagram. We summarize this result in Lemma~\ref{lem:homotopy_simplexnerve}.

\begin{lemma}
	
	\label{lem:homotopy_simplexnerve}
	
	Let $\mathbb{X}$ be a paracompact space, and let $\mathcal{U}=\{U_{i}\}_{i\in I}$,
	$\mathcal{U}'=\{U'_{i}\}_{i\in I'}$, be good covers of $\mathbb{X}$, with
	$I\subset I'$ and $U_{i}\subset U'_{i}$ for all $i\in I$. Let $\mathcal{S}$
	be a simplicial complex satisfying $\mathcal{N}\mathcal{U}\subset\mathcal{S}\subset\mathcal{N}\mathcal{U}'$.
	Suppose there exists $\rho:\mathcal{S}\to\mathcal{N}\mathcal{U}$
	such that $\imath_{\mathcal{N}\mathcal{U}\to\mathcal{S}}\circ\rho:S\to S$
	is homotopy equivalent to $id_{\mathcal{S}}:\mathcal{S}\to\mathcal{S}$.
	Then $\mathbb{X}$ and $\mathcal{S}$ are homotopy equivalent.
	
\end{lemma}

\begin{proof}[Proof of Lemma~\ref{lem:homotopy_simplexnerve}]
	
	Let $\phi:\mathbb{X}\to\mathcal{N}\mathcal{U}$ and $\psi:\mathcal{N}\mathcal{U}'\to \mathbb{X}$
	be the maps giving homotopy equivalence of the Nerve Theorem. Then from Lemma~\ref{lem:background_nerve_interleaving}, the following diagram commutes
	at homotopy level:
	\begin{equation}
	\xymatrix{ & \mathbb{X}\ar[dl]_{\phi}\\
		\mathcal{NU}\ar[rr]^{\imath_{\mathcal{N}\mathcal{U}\to\mathcal{N\mathcal{U}}'}} &  & \mathcal{NU}'\ar[ul]_{\psi'}
	}
	.\label{eq:homotopy_simplexnerve_nerve}
	\end{equation}
	
	To show homotopy equivalence, we define maps $\Phi:\mathbb{X}\to\mathcal{S}$
	and $\Psi:\mathcal{S}\to \mathbb{X}$ by $\Phi=\imath_{\mathcal{N}\mathcal{U}\to\mathcal{S}}\circ\phi$
	and $\Psi=\psi'\circ\imath_{\mathcal{S}\to\mathcal{NU}'}$. Then,
	we need to show that $\Psi\circ\Phi\simeq id_{\mathbb{X}}$ and $\Phi\circ\Psi\simeq id_{\mathcal{S}}$
	on homotopy level.
	
	For $\Psi\circ\Phi\simeq id_{\mathbb{X}}$, consider the diagram below, which is \eqref{eq:homotopy_simplexnerve_initial}:
	\[
	\xymatrix{ & \mathbb{X}\ar[dl]_{\phi}\\
		\mathcal{NU}\ar[dr]_{\imath_{\mathcal{N}\mathcal{U}\to\mathcal{S}}}\ar[rr]^{\imath_{\mathcal{N}\mathcal{U}\to\mathcal{N\mathcal{U}}'}} &  & \mathcal{NU}'\ar[ul]_{\psi'}\\
		& \mathcal{S}\ar[ur]_{\imath_{\mathcal{S}\to\mathcal{NU}'}}
	}
	.
	\]
	Then from \eqref{eq:homotopy_simplexnerve_nerve}, $\psi'\circ\imath_{\mathcal{N}\mathcal{U}\to\mathcal{NU}'}\circ\phi\simeq id_{\mathbb{X}}$
	holds, and hence 
	\begin{align*}
		\Psi\circ\Phi & =\psi'\circ\imath_{\mathcal{S}\to\mathcal{NU}'}\circ\imath_{\mathcal{N}\mathcal{U}\to\mathcal{S}}\circ\phi\\
		& =\psi'\circ\imath_{\mathcal{N}\mathcal{U}\to\mathcal{NU}'}\circ\phi\\
		& \simeq id_{\mathbb{X}}.
	\end{align*}
	And hence $\Psi\circ\Phi\simeq id_{\mathbb{X}}$ on homotopy level.
	
	For $\Phi\circ\Psi\simeq id_{\mathcal{S}}$, consider the diagram
	below, which is \eqref{eq:homotopy_simplexnerve_withcondition}:
	\[
	\xymatrix{ & \mathbb{X}\ar[dl]_{\phi}\\
		\mathcal{NU}\ar@/_/[dr]_{\imath_{\mathcal{N}\mathcal{U}\to\mathcal{S}}}\ar[rr]^{\imath_{\mathcal{N}\mathcal{U}\to\mathcal{N\mathcal{U}}'}} &  & \mathcal{NU}'\ar[ul]_{\psi'}\\
		& \mathcal{S}\ar@/_/[ul]_{\rho}\ar[ur]_{\imath_{\mathcal{S}\to\mathcal{NU}'}}
	}
	.
	\]
	Then $\imath_{\mathcal{N}\mathcal{U}\to\mathcal{S}}\circ\rho\simeq id_{\mathcal{S}}$
	holds from the condition, and from \eqref{eq:homotopy_simplexnerve_nerve},
	$\phi\circ\psi'\circ\imath_{\mathcal{N}\mathcal{U}\to\mathcal{NU}'}\simeq id_{\mathcal{NU}}$
	holds, and hence 
	\begin{align*}
		\Phi\circ\Psi & =\imath_{\mathcal{N}\mathcal{U}\to\mathcal{S}}\circ\phi\circ\psi'\circ\imath_{\mathcal{S}\to\mathcal{NU}'}\\
		& \simeq\imath_{\mathcal{N}\mathcal{U}\to\mathcal{S}}\circ\phi\circ\psi'\circ\imath_{\mathcal{S}\to\mathcal{NU}'}\circ\imath_{\mathcal{N}\mathcal{U}\to\mathcal{S}}\circ\rho\\
		& =\imath_{\mathcal{N}\mathcal{U}\to\mathcal{S}}\circ\phi\circ\psi'\circ\imath_{\mathcal{N}\mathcal{U}\to\mathcal{NU}'}\circ\rho\\
		& \simeq\imath_{\mathcal{N}\mathcal{U}\to\mathcal{S}}\circ\rho\\
		& \simeq id_{\mathcal{S}}.
	\end{align*}
	And hence $\Phi\circ\Psi\simeq id_{\mathcal{S}}$ on homotopy level.
	
\end{proof}

	

Hence for showing Theorem~\ref{thm:homotopy_cech} and \ref{thm:homotopy_rips},
the problem reduces to finding a subcomplex in $\textrm{\v{C}ech}_{\mathbb{X}}(\mathcal{X},r)$
that is homotopy equivalent to $\mathcal{S}=\textrm{\v{C}ech}_{\mathbb{R}^{d}}(\mathcal{X},r)$
or $\mathcal{S}=\textrm{Rips}(\mathcal{X},r)$. To do this, for each
$\sigma=[x_{1}\ldots x_{k}]\in\mathcal{S}$, we choose $b_{\sigma}$
to be a point in a convex hull of $\{x_{1},\ldots,x_{k}\}$, and we
find $y_{\sigma}\in\mathcal{X}$ to be a point close to $\pi_{\mathbb{X}}(b_{\sigma})$
such that 
\begin{equation}
[x_{1},\ldots,x_{k}]\simeq\sum_{i=1}^{k}[x_{1}\cdots\hat{x}_{i}\cdots x_{k}y_{\sigma}]\text{ in }\mathcal{S}.\label{eq:homotopy_cech_rips_equivalence_first}
\end{equation}
We use the notation $[x_{1},\ldots,x_{k}]$ to emphasize that each
simplex is considered with its topology structure, and $[x_{1}\cdots\hat{x}_{i}\cdots x_{k}y_{\sigma}]$
means that the vertex $x_{i}$ is removed. To show \eqref{eq:homotopy_cech_rips_equivalence_first},
we need several bounds for $\left\Vert x_{i}-y_{\sigma}\right\Vert $
and $\left\Vert y_{\varsigma}-y_{\sigma}\right\Vert $, which we collect
in Claim~\ref{claim:homotopy_cech_rips}.

\begin{claim}

\label{claim:homotopy_cech_rips}

Let $\tau>0$, $\mathbb{X}\subset\mathbb{R}^{d}$ be a subset with
reach $\tau_{\mathbb{X}}\geq\tau>0$, and $\mathcal{X}\subset\mathbb{R}^{d}$
be a set of points. Let $\{r_{x}>0:x\in\mathcal{X}\}$ be a set of
radii indexed by $x\in\mathcal{X}$. For each $\sigma\subset\mathcal{X}$,
let $\epsilon_{\sigma}\geq0$ be satisfying $d_{\mathbb{X}}(x)\leq\epsilon_{\sigma}$
for all $x\in\sigma$. Let $\delta>0$, and suppose $\mathbb{X}\subset\bigcup_{x\in\mathcal{X}}\mathbb{B}_{\mathbb{X}}(x,\delta)$.
For each $\sigma\subset\mathcal{X}$, let $b_{\sigma}$ be a point
in the convex hull of $\sigma$, and let $r_{\sigma}:=\max_{x\in\sigma}\left\Vert x-b_{\sigma}\right\Vert $.
Then for each $\sigma\subset\mathcal{X}$, there exists $y_{\sigma}\in\mathcal{X}$
that satisfy the followings: 
\begin{enumerate}
	\item[(i)] If $r_{\sigma}<\tau$, then
	\[
	d_{\mathbb{X}}(b_{\sigma})\leq\tau-\sqrt{(\tau-\epsilon_{\sigma})^{2}-r_{\sigma}^{2}}.
	\]

	\item[(ii)] If $r_{\sigma}\leq\tau-\epsilon_{\sigma}$, then 
	\[
	\left\Vert x_{i}-\pi_{\mathbb{X}}(b_{\sigma})\right\Vert \leq\sqrt{\tilde{r}_{b_{\sigma}}^{2}-\epsilon_{\sigma}(2\tau-\epsilon_{\sigma})},
	\]
	and 
	\[
	\left\Vert x_{i}-y_{\sigma}\right\Vert <\sqrt{\tilde{r}_{b_{\sigma}}^{2}-\epsilon_{\sigma}(2\tau-\epsilon_{\sigma})}+\delta,
	\]
	where 
	\[
	\tilde{r}_{b_{\sigma}}^{2}:=\frac{2\tau\left(r_{\sigma}^{2}+\epsilon_{\sigma}(2\tau-\epsilon_{\sigma})\right)}{\tau+\sqrt{\tau^{2}-\left(r_{\sigma}^{2}+\epsilon_{\sigma}(2\tau-\epsilon_{\sigma})\right)}}.
	\]

	\item[(iii)] If $r_{\sigma}<\tau-\epsilon_{\sigma}$ and suppose $\varsigma\subset\sigma$
	, then 
	\[
	\tilde{r}_{b_{\sigma}}^{2}-\epsilon_{\sigma}(2\tau-\epsilon_{\sigma})\leq\tilde{r}_{b_{\sigma}}^{2}-(2\tau^{2}-\tilde{r}_{b_{\sigma}}^{2})\left(\frac{\tau}{\sqrt{\tau^{2}-\tilde{r}_{\delta,b}^{2}}}-1\right)+\delta^{2},
	\]
	and
	\[
	\left\Vert y_{\varsigma}-y_{\sigma}\right\Vert <\sqrt{\tilde{r}_{b_{\sigma}}^{2}-(2\tau^{2}-\tilde{r}_{b_{\sigma}}^{2})\left(\frac{\tau}{\sqrt{\tau^{2}-\tilde{r}_{\delta,b}^{2}}}-1\right)}+2\delta,
	\]
	where $\tilde{r}_{b_{\sigma}}^{2}$ is from (ii) and 
	\[
		\tilde{r}_{\delta,b}^{2}:=\min\left\{ \delta^{2}+\epsilon_{\sigma}(2\tau-\epsilon_{\sigma}),\frac{1}{2}\tilde{r'}_{b_{\sigma}}^{2}\right\}.
	\]

\item[(iv)] Suppose that for all $\varsigma$ with $\{x_{1},\ldots,x_{j}\}\subset\varsigma\subset\sigma$,
$b_{\varsigma}\in\bigcap_{x\in\varsigma}\mathbb{B}_{\mathbb{R}^{d}}(x,r_{x})$. If $r_{\sigma}<\tau-\epsilon_{\sigma}$ and suppose 
\[
\delta+\sqrt{r_{\sigma}^{2}-\tilde{l}^{2}+\epsilon_{\sigma}(2\tau-\epsilon_{\sigma})-((\tau-\epsilon_{\sigma})^{2}-r_{\sigma}^{2}+\tilde{l}^{2}+(\tau-\epsilon_{\tilde{l}})^{2})\left(\frac{\tau}{\sqrt{\tau^{2}-\tilde{r}_{\delta,c}^{2}}}-1\right)}\leq r_{\min},
\]
where 
\begin{align*}
\tilde{l} & :=\frac{1}{2}\left(r_{\min}-\tau+\sqrt{(\tau-\epsilon_{\sigma})^{2}-r_{\sigma}^{2}}-\delta\right),\\
\epsilon_{\tilde{l}} & :=\tau-\sqrt{(\tau-\epsilon_{\sigma})^{2}-r_{\sigma}^{2}}+\tilde{l},\\
\tilde{r}_{\delta,c}^{2} & :=\min\left\{ \delta^{2}+\epsilon_{\sigma}(2\tau-\epsilon_{\sigma}),\frac{1}{2}(r_{\sigma}^{2}-\tilde{l}^{2}+\epsilon_{\sigma}(2\tau-\epsilon_{\sigma})+\epsilon_{\tilde{l}}(2\tau-\epsilon_{\tilde{l}}))\right\} .
\end{align*}
Then $\sigma\in\textrm{\v{C}ech}_{\mathbb{R}^{d}}(\mathcal{X},r)$
implies that 
\[
[x_{1}\cdots x_{j}y_{[x_{1}\cdots x_{j}]}y_{[x_{1}\cdots x_{j+1}]}\cdots y_{[x_{1}\cdots x_{k}]}]\in\textrm{\v{C}ech}_{\mathbb{R}^{d}}(\mathcal{X},r).
\]
\item[(v)] If $r_{\sigma}<\tau-\epsilon_{\sigma}$ and suppose 
\[
\sqrt{\tilde{r}_{b_{\sigma}}^{2}-(2\tau^{2}-\tilde{r}_{b_{\sigma}}^{2})\left(\frac{\tau}{\sqrt{\tau^{2}-\tilde{r}_{\delta,b}^{2}}}-1\right)}+2\delta\leq2r_{\min},
\]
where $\tilde{r}_{b_{\sigma}}$ is from (ii) and $\tilde{r}_{\delta,b}$
is from (iii), then $\sigma\in\textrm{Rips}(\mathcal{X},r)$ implies
that
\[
[x_{1}\cdots x_{j}y_{[x_{1}\cdots x_{j}]}y_{[x_{1}\cdots x_{j+1}]}\cdots y_{[x_{1}\cdots x_{k}]}]\in\textrm{Rips}(\mathcal{X},r).
\]
\end{enumerate}
\end{claim}

\begin{proof}[Proof of Claim~\ref{claim:homotopy_cech_rips}]

	For choosing $y_{\sigma}\in\mathcal{X}$, we divide the problem into two cases. If $d(\pi_{\mathbb{X}}(b_{\sigma}),\sigma)<\delta$,
	then choose $y_{\sigma}=\arg\min_{x\in\sigma}\left\Vert x-\pi_{\mathbb{X}}(b_{\sigma})\right\Vert $.
	Otherwise, by using covering condition, choose $y_{\sigma}\in\mathcal{X}$
	be satisfying $\left\Vert y_{\sigma}-\pi_{\mathbb{X}}(b_{\sigma})\right\Vert <\delta$.
	
	(i)
	
	From $d_{\mathbb{X}}(x_{i})\leq\epsilon_{\sigma}<\tau$ for $i=1,\ldots,k$, $d_{\mathbb{X}}(b_{\sigma}) \leq r_{\sigma}\leq \tau-\epsilon_{\sigma}$, and $d_{\mathbb{X}}(b_{\sigma}) < \tau$,
	applying Lemma~\ref{lem:projreach_distance_segment} gives 
	\begin{align*}
		\left\Vert b_{\sigma}-\pi_{\mathbb{X}}(b_{\sigma})\right\Vert  & \leq\tau-\sqrt{\left((\tau-\epsilon)^{2}-\sum_{i=1}^{k}\lambda_{i}\left\Vert x_{i}-b_{\sigma}\right\Vert ^{2}\right)_{+}}\\
		& \leq\tau-\sqrt{(\tau-\epsilon_{\sigma})^{2}-r_{\sigma}^{2}}.
	\end{align*}
	
	(ii)
	
	Note that $\left\Vert x_{i}-b_{\sigma}\right\Vert \leq r_{\sigma}\leq \tau-\epsilon_{\sigma}$,
	so applying Lemma~\ref{lem:projreach_distance_projection_general}
	gives 
	\begin{align*}
		& \left\Vert x_{i}-\pi_{\mathbb{X}}(b_{\sigma})\right\Vert \\
		& \leq\sqrt{\frac{2\tau\left(\left\Vert x_{i}-b_{\sigma}\right\Vert ^{2}+\epsilon_{\sigma}(2\tau-\epsilon_{\sigma})\right)}{\tau+\sqrt{\tau^{2}-\left(\left\Vert x_{i}-b_{\sigma}\right\Vert ^{2}+\epsilon_{\sigma}(2\tau-\epsilon_{\sigma})\right)}}-\epsilon_{\sigma}(2\tau-\epsilon_{\sigma})}.\\
		& \leq\sqrt{\frac{2\tau\left(r_{\sigma}^{2}+\epsilon_{\sigma}(2\tau-\epsilon_{\sigma})\right)}{\tau+\sqrt{\tau^{2}-\left(r_{\sigma}^{2}+\epsilon_{\sigma}(2\tau-\epsilon_{\sigma})\right)}}-\epsilon_{\sigma}(2\tau-\epsilon_{\sigma})}\\
		& \leq\sqrt{\tilde{r}_{b_{\sigma}}^{2}-\epsilon_{\sigma}(2\tau-\epsilon_{\sigma})},		
	\end{align*}
	where 
	\[
	\tilde{r}_{b_{\sigma}}^{2}:=\frac{2\tau\left(r_{\sigma}^{2}+\epsilon_{\sigma}(2\tau-\epsilon_{\sigma})\right)}{\tau+\sqrt{\tau^{2}-\left(r_{\sigma}^{2}+\epsilon_{\sigma}(2\tau-\epsilon_{\sigma})\right)}}.
	\]
	And hence 
	\begin{align*}
		\left\Vert x_{i}-y_{\sigma}\right\Vert  & \leq\left\Vert x_{i}-\pi_{\mathbb{X}}(b_{\sigma})\right\Vert +\left\Vert \pi_{\mathbb{X}}(b_{\sigma})-y_{\sigma}\right\Vert \\
		& <\sqrt{\tilde{r}_{b_{\sigma}}^{2}-\epsilon_{\sigma}(2\tau-\epsilon_{\sigma})}+\delta.
	\end{align*}
	
	(iii)
	
	Without loss of generality assume that $\varsigma=[x_{1}\cdots x_{j}]$
	and $\sigma=[x_{1}\cdots x_{k}]$ with $j\leq k$. Then (ii) implies that for each $i=1,\ldots,j$ , 
	\[
	\left\Vert x_{i}-\pi_{\mathbb{X}}(b_{\sigma})\right\Vert \leq\sqrt{\frac{2\tau\left(r_{\sigma}^{2}+\epsilon_{\sigma}(2\tau-\epsilon_{\sigma})\right)}{\tau+\sqrt{\tau^{2}-\left(r_{\sigma}^{2}+\epsilon_{\sigma}(2\tau-\epsilon_{\sigma})\right)}}-\epsilon_{\sigma}(2\tau-\epsilon_{\sigma})}.
	\]
	We divide the problem into two cases. First, consider the case $d(b_{\varsigma},\varsigma)<\delta$.
	Then $y_{\varsigma}=x_{l}$ for some $x_{l}\in\varsigma$ and 
	\begin{align*}
		\left\Vert y_{\varsigma}-\pi_{\mathbb{X}}(b_{\sigma})\right\Vert  & =\left\Vert x_{l}-\pi_{\mathbb{X}}(b_{\sigma})\right\Vert \leq\sqrt{\frac{2\tau\left(r_{\sigma}^{2}+\epsilon_{\sigma}(2\tau-\epsilon_{\sigma})\right)}{\tau+\sqrt{\tau^{2}-\left(r_{\sigma}^{2}+\epsilon_{\sigma}(2\tau-\epsilon_{\sigma})\right)}}-\epsilon_{\sigma}(2\tau-\epsilon_{\sigma})}\\
		& =\sqrt{\tilde{r}_{b_{\sigma}}^{2}-\epsilon_{\sigma}(2\tau-\epsilon_{\sigma})},
	\end{align*}
	where $\tilde{r}_{b_{\sigma}}^{2}=\frac{2\tau\left(r_{\sigma}^{2}+\epsilon_{\sigma}(2\tau-\epsilon_{\sigma})\right)}{\tau+\sqrt{\tau^{2}-\left(r_{\sigma}^{2}+\epsilon_{\sigma}(2\tau-\epsilon_{\sigma})\right)}}$ is from (ii).
	And hence for $d(b_{\varsigma},\varsigma)<\delta$ case,
	\begin{align}
		\left\Vert y_{\varsigma}-y_{\sigma}\right\Vert  & \leq\left\Vert y_{\varsigma}-\pi_{\mathbb{X}}(b_{\sigma})\right\Vert +\left\Vert \pi_{\mathbb{X}}(b_{\sigma})-y_{\sigma}\right\Vert \nonumber \\
		& <\sqrt{\tilde{r}_{b_{\sigma}}^{2}-\epsilon_{\sigma}(2\tau-\epsilon_{\sigma})}+\delta.\label{eq:homotopy_cech_rips_barycenterdist_close}
	\end{align}
	Otherwise, $d(b_{\varsigma},\varsigma)\geq\delta$, i.e. $\left\Vert x_{i}-b_{\varsigma}\right\Vert \geq\delta$
	for all $1\leq i\leq j$. Then 
	\begin{align*}
		\left\Vert \pi_{\mathbb{X}}(b_{\sigma})-x_{i}\right\Vert  & <\sqrt{\frac{2\tau\left((\tau-\epsilon_{\sigma})^{2}+\epsilon_{\sigma}(2\tau-\epsilon_{\sigma})\right)}{\tau+\sqrt{\tau^{2}-\left((\tau-\epsilon_{\sigma})^{2}+\epsilon_{\sigma}(2\tau-\epsilon_{\sigma})\right)}}-\epsilon_{\sigma}(2\tau-\epsilon_{\sigma})}\\
		& =\sqrt{\tau^{2}+(\tau-\epsilon_{\sigma})^{2}}.
	\end{align*}
	So Lemma~\ref{lem:projreach_distance_projection_center} is applicable
	and gives 
	\begin{align*}
		\left\Vert \pi_{\mathbb{X}}(b_{\sigma})-\pi_{\mathbb{X}}(b_{\varsigma})\right\Vert  & \leq\sqrt{\tilde{r}_{b_{\sigma}}^{2}-(2\tau^{2}-\tilde{r}_{b_{\sigma}}^{2})\left(\frac{\tau}{\sqrt{\tau^{2}-\tilde{r}_{b_{\varsigma},b_{\sigma}}^{2}}}-1\right)},
	\end{align*}
	where 
	\begin{align*}
		\tilde{r'}_{b_{\sigma}}^{2} & :=\sum_{i=1}^{k}\lambda_{i}\left(\left\Vert x_{i}-\pi_{\mathbb{X}}(b_{\sigma})\right\Vert ^{2}+\epsilon_{\sigma}(2\tau-\epsilon_{\sigma})\right),\\
		\tilde{r'}_{b_{\varsigma},b_{\sigma}}^{2} & :=\min\left\{ \sum_{i=1}^{k}\lambda_{i}\left(\left\Vert x_{i}-b_{\varsigma}\right\Vert ^{2}+\epsilon_{\sigma}(2\tau-\epsilon_{\sigma})\right),\frac{1}{2}\tilde{r'}_{b_{\sigma}}^{2}\right\} .
	\end{align*}
	Then, RHS is an increasing function of $\left\Vert x_{i}-\pi_{\mathbb{X}}(b_{\sigma})\right\Vert ^{2}$
	and a decreasing function of $\left\Vert x_{i}-b_{\varsigma}\right\Vert ^{2}$,
	and from $\left\Vert x_{i}-b_{\varsigma}\right\Vert \geq\delta$ for
	all $1\leq i\leq j$ , we have 
	\[
	\left\Vert \pi_{\mathbb{X}}(b_{\varsigma})-\pi_{\mathbb{X}}(b_{\sigma})\right\Vert \leq\sqrt{\tilde{r}_{b_{\sigma}}^{2}-(2\tau^{2}-\tilde{r}_{b_{\sigma}}^{2})\left(\frac{\tau}{\sqrt{\tau^{2}-\tilde{r}_{\delta,b}^{2}}}-1\right)},
	\]
	where 
	\begin{align*}
		\tilde{r}_{b_{\sigma}}^{2} & =\frac{2\tau\left(r_{\sigma}^{2}+\epsilon_{\sigma}(2\tau-\epsilon_{\sigma})\right)}{\tau+\sqrt{\tau^{2}-\left(r_{\sigma}^{2}+\epsilon_{\sigma}(2\tau-\epsilon_{\sigma})\right)}},\\
		\tilde{r}_{\delta,b}^{2} & :=\min\left\{ \delta^{2}+\epsilon_{\sigma}(2\tau-\epsilon_{\sigma}),\frac{1}{2}\tilde{r}_{b_{\sigma}}^{2}\right\} .
	\end{align*}
	And hence for $d(b_{\varsigma},\varsigma)\geq\delta$ case,
	\begin{align}
		\left\Vert y_{\varsigma}-y_{\sigma}\right\Vert  & \leq\left\Vert y_{\varsigma}-\pi_{\mathbb{X}}(b_{\varsigma})\right\Vert +\left\Vert \pi_{\mathbb{X}}(b_{\varsigma})-\pi_{\mathbb{X}}(b_{\sigma})\right\Vert +\left\Vert \pi_{\mathbb{X}}(b_{\sigma})-y_{\sigma}\right\Vert \nonumber \\
		& <\sqrt{\tilde{r}_{b_{\sigma}}^{2}-(2\tau^{2}-\tilde{r}_{b_{\sigma}}^{2})\left(\frac{\tau}{\sqrt{\tau^{2}-\tilde{r}_{\delta,b}^{2}}}-1\right)}+2\delta.\label{eq:homotopy_cech_rips_barycenterdist_far}
	\end{align}
	Hence \eqref{eq:homotopy_cech_rips_barycenterdist_close} and \eqref{eq:homotopy_cech_rips_barycenterdist_far}
	gives that for any case, 
	\begin{align}
		& \left\Vert y_{\varsigma}-y_{\sigma}\right\Vert \nonumber \\
		& <\delta+\max\left\{ \sqrt{\tilde{r}_{b_{\sigma}}^{2}-\epsilon_{\sigma}(2\tau-\epsilon_{\sigma})},\sqrt{\tilde{r}_{b_{\sigma}}^{2}-(2\tau^{2}-\tilde{r}_{b_{\sigma}}^{2})\left(\frac{\tau}{\sqrt{\tau^{2}-\tilde{r}_{\delta,b}^{2}}}-1\right)}+\delta\right\} .\label{eq:homotopy_cech_rips_barycenterdist_max}
	\end{align}
	
	Now, we show $\tilde{r}_{b_{\sigma}}^{2}-\epsilon_{\sigma}(2\tau-\epsilon_{\sigma})\leq\tilde{r}_{b_{\sigma}}^{2}-(2\tau^{2}-\tilde{r}_{b_{\sigma}}^{2})\left(\frac{\tau}{\sqrt{\tau^{2}-\tilde{r'}_{\delta}^{2}}}-1\right)+\delta^{2}$.
	We again divide it into two cases. First, when $\frac{1}{2}\tilde{r}_{b_{\sigma}}^{2}\leq\delta^{2}+\epsilon_{\sigma}(2\tau-\epsilon_{\sigma})$,
	then $\tilde{r}_{\delta,b}^{2}=\frac{1}{2}\tilde{r}_{b_{\sigma}}^{2}$
	and therefore, 
	\begin{align*}
		\tilde{r}_{b_{\sigma}}^{2}-(2\tau^{2}-\tilde{r}_{b_{\sigma}}^{2})\left(\frac{\tau}{\sqrt{\tau^{2}-\tilde{r}_{\delta,b}^{2}}}-1\right) & =\tilde{r}_{b_{\sigma}}^{2}-(2\tau^{2}-\tilde{r}_{b_{\sigma}}^{2})\left(\frac{\tau-\sqrt{\tau^{2}-\frac{1}{2}\tilde{r}_{b_{\sigma}}^{2}}}{\sqrt{\tau^{2}-\frac{1}{2}\tilde{r}_{b_{\sigma}}^{2}}}\right)\\
		& =2\tau^{2}-2\tau\sqrt{\tau^{2}-\frac{1}{2}\tilde{r}_{b_{\sigma}}^{2}}\\
		& =2\tau^{2}-\sqrt{4\tau^{4}-2\tau^{2}\tilde{r}_{b_{\sigma}}^{2}}\\
		& \geq2\tau^{2}-\sqrt{4\tau^{4}-2\tau^{2}\tilde{r}_{b_{\sigma}}^{2}+\frac{1}{4}\tilde{r}_{b_{\sigma}}^{4}}\\
		& =\frac{1}{2}\tilde{r}_{b_{\sigma}}^{2}\\
		& \geq\tilde{r}_{b_{\sigma}}^{2}-(\delta^{2}+\epsilon_{\sigma}(2\tau-\epsilon_{\sigma}))\\
		& =\left(\tilde{r}_{b_{\sigma}}^{2}-\epsilon_{\sigma}(2\tau-\epsilon_{\sigma})\right)-\delta^{2}.
	\end{align*}
	And hence for $\frac{1}{2}\tilde{r}_{b_{\sigma}}^{2}\leq\delta^{2}+\epsilon_{\sigma}(2\tau-\epsilon_{\sigma})$
	case, 
	\begin{equation}
	\tilde{r}_{b_{\sigma}}^{2}-\epsilon_{\sigma}(2\tau-\epsilon_{\sigma})\leq\tilde{r}_{b_{\sigma}}^{2}-(2\tau^{2}-\tilde{r}_{b_{\sigma}}^{2})\left(\frac{\tau}{\sqrt{\tau^{2}-\tilde{r'}_{\delta}^{2}}}-1\right)+\delta^{2}.\label{eq:homotopy_cech_rips_barycenterdist_compare_first}
	\end{equation}
	Second, when $\delta^{2}+\epsilon_{\sigma}(2\tau-\epsilon_{\sigma})\leq\frac{1}{2}\tilde{r}_{b_{\sigma}}^{2}$,
	then $\tilde{r}_{\delta,b}^{2}=\delta^{2}+\epsilon_{\sigma}(2\tau-\epsilon_{\sigma})$,
	and therefore, 
	\begin{align*}
		\tilde{r}_{b_{\sigma}}^{2}-(2\tau^{2}-\tilde{r}_{b_{\sigma}}^{2})\left(\frac{\tau}{\sqrt{\tau^{2}-\tilde{r}_{\delta,b}^{2}}}-1\right) & =\tilde{r}_{b_{\sigma}}^{2}-(2\tau^{2}-\tilde{r}_{b_{\sigma}}^{2})\left(\frac{\tau-\sqrt{\tau^{2}-\tilde{r}_{\delta,b}^{2}}}{\sqrt{\tau^{2}-\tilde{r}_{\delta,b}^{2}}}\right)\\
		& =\tilde{r}_{b_{\sigma}}^{2}-\frac{(2\tau^{2}-\tilde{r}_{b_{\sigma}}^{2})\tilde{r}_{\delta,b}^{2}}{\sqrt{\tau^{2}-\tilde{r}_{\delta,b}^{2}}(\tau+\sqrt{\tau^{2}-\tilde{r}_{\delta,b}^{2}})}.
	\end{align*}
	Then from $\tilde{r}_{\delta,b}^{2}=\delta^{2}+\epsilon_{\sigma}(2\tau-\epsilon_{\sigma})\leq\frac{1}{2}\tilde{r}_{b_{\sigma}}^{2}$,
	$\frac{2\tau^{2}-\tilde{r}_{b_{\sigma}}^{2}}{\sqrt{\tau^{2}-\tilde{r}_{\delta,b}^{2}}(\tau+\sqrt{\tau^{2}-\tilde{r}_{\delta,b}^{2}})}$
	is lower bounded as 
	\begin{align*}
		\frac{2\tau^{2}-\tilde{r}_{b_{\sigma}}^{2}}{\sqrt{\tau^{2}-\tilde{r}_{\delta,b}^{2}}(\tau+\sqrt{\tau^{2}-\tilde{r}_{\delta,b}^{2}})} & =\frac{2\tau^{2}-\tilde{r}_{b_{\sigma}}^{2}}{\tau^{2}-\tilde{r}_{\delta,b}^{2}+\tau\sqrt{\tau^{2}-\tilde{r}_{\delta,b}^{2}}}\\
		& \leq\frac{2\tau^{2}-\tilde{r}_{b_{\sigma}}^{2}}{\tau^{2}-\tilde{r}_{\delta,b}^{2}+(\tau^{2}-\tilde{r}_{\delta,b}^{2})}\\
		& \leq\frac{2\tau^{2}-\tilde{r}_{b_{\sigma}}^{2}}{2\tau^{2}-\tilde{r}_{b_{\sigma}}^{2}}=1,
	\end{align*}
	and hence 
	\begin{align*}
		\tilde{r}_{b_{\sigma}}^{2}-(2\tau^{2}-\tilde{r}_{b_{\sigma}}^{2})\left(\frac{\tau}{\sqrt{\tau^{2}-\tilde{r}_{\delta,b}^{2}}}-1\right) & =\tilde{r}_{b_{\sigma}}^{2}-\frac{(2\tau^{2}-\tilde{r}_{b_{\sigma}}^{2})\tilde{r}_{\delta,b}^{2}}{\sqrt{\tau^{2}-\tilde{r}_{\delta,b}^{2}}(\tau+\sqrt{\tau^{2}-\tilde{r}_{\delta,b}^{2}})}\\
		& \geq\tilde{r}_{b_{\sigma}}^{2}-\tilde{r}_{\delta,b}^{2}\\
		& =\left(\tilde{r}_{b_{\sigma}}^{2}-\epsilon_{\sigma}(2\tau-\epsilon_{\sigma})\right)-\delta^{2}.
	\end{align*}
	And hence for $\delta^{2}+\epsilon_{\sigma}(2\tau-\epsilon_{\sigma})\leq\frac{1}{2}\tilde{r}_{b_{\sigma}}^{2}$
	case, 
	\begin{equation}
	\tilde{r}_{b_{\sigma}}^{2}-\epsilon_{\sigma}(2\tau-\epsilon_{\sigma})\leq\tilde{r}_{b_{\sigma}}^{2}-(2\tau^{2}-\tilde{r}_{b_{\sigma}}^{2})\left(\frac{\tau}{\sqrt{\tau^{2}-\tilde{r'}_{\delta}^{2}}}-1\right)+\delta^{2}.\label{eq:homotopy_cech_rips_barycenterdist_compare_second}
	\end{equation}
	Hence from \eqref{eq:homotopy_cech_rips_barycenterdist_compare_first}
	and \eqref{eq:homotopy_cech_rips_barycenterdist_compare_second},
	for any case, 
	\[
	\tilde{r}_{b_{\sigma}}^{2}-\epsilon_{\sigma}(2\tau-\epsilon_{\sigma})\leq\tilde{r}_{b_{\sigma}}^{2}-(2\tau^{2}-\tilde{r}_{b_{\sigma}}^{2})\left(\frac{\tau}{\sqrt{\tau^{2}-\tilde{r'}_{\delta}^{2}}}-1\right)+\delta^{2}.
	\]
	and applying this to \eqref{eq:homotopy_cech_rips_barycenterdist_max}
	gives that 
	\[
	\left\Vert y_{\varsigma}-y_{\sigma}\right\Vert <\sqrt{\tilde{r}_{b_{\sigma}}^{2}-(2\tau^{2}-\tilde{r}_{b_{\sigma}}^{2})\left(\frac{\tau}{\sqrt{\tau^{2}-\tilde{r}_{\delta,b}^{2}}}-1\right)}+2\delta.
	\]

	(iv)
	
	We need to find $\tilde{b}\in\mathbb{R}^{d}$ and check that $\left\Vert x_{i}-\tilde{b}\right\Vert <r_{x_{i}}$
	for $i=1,\ldots,k$ and $\left\Vert y_{\varsigma}-\tilde{b}\right\Vert <r_{y_{\varsigma}}$
	for all $\varsigma$ with $\{x_{1},\ldots,x_{j}\}\subset\varsigma\subset\sigma=\{x_{1},\ldots,x_{k}\}$.
	We let 
	\[
		\tilde{l}:=\frac{1}{2}\left(r_{\min}-\tau+\sqrt{(\tau-\epsilon_{\sigma})^{2}-r_{\sigma}^{2}}-\delta\right),
	\]
	and divide the problem into two cases.
	
	First, if all $\varsigma$ with $\{x_{1},\ldots,x_{j}\}\subset\varsigma\subset\sigma$
	satisfy that $\left\Vert b_{\varsigma}-b_{\sigma}\right\Vert \leq\tilde{l}$,
	then we can set $\tilde{b}:=b_{\sigma}$. For this case, for $i=1,\ldots,j,$
	\[
	\left\Vert x_{i}-\tilde{b}\right\Vert =\left\Vert x_{i}-b_{\sigma}\right\Vert <r_{x_{i}},
	\]
	and for any $\varsigma$ with $\{x_{1},\ldots,x_{j}\}\subset\varsigma\subset\sigma,$
	(i) and the covering condition imply
	\begin{align*}
		& \left\Vert y_{\varsigma}-\tilde{b}\right\Vert \\
		& \leq\left\Vert y_{\varsigma}-\pi_{\mathbb{X}}(b_{\varsigma})\right\Vert +\left\Vert \pi_{\mathbb{X}}(b_{\varsigma})-b_{\varsigma}\right\Vert +\left\Vert b_{\varsigma}-b_{\sigma}\right\Vert \\
		& <\delta+\left(\tau-\sqrt{(\tau-\epsilon_{\sigma})^{2}-r_{\sigma}^{2}}\right)+\tilde{l}\\
		& =r_{\min}-\tilde{l}\leq r_{\min}\leq r_{y_{\varsigma}}.
	\end{align*}
	
	For the other case, let $\hat{k}:=\max\left\{ j:\left\Vert b_{[x_{1}\cdots x_{j}]}-b_{\sigma}\right\Vert >\tilde{l}\right\} $
	be the largest integer that $\left\Vert b_{[x_{1}\cdots x_{j}]}-b_{\sigma}\right\Vert >\tilde{l}$
	, and let $\hat{\sigma}:=[x_{1}\cdots x_{\hat{k}}]$. Let $\tilde{b}:=\frac{\tilde{l}}{\left\Vert b_{\sigma}-b_{\hat{\sigma}}\right\Vert }b_{\hat{\sigma}}+\left(1-\frac{\tilde{l}}{\left\Vert b_{\sigma}-b_{\hat{\sigma}}\right\Vert }\right)b_{\sigma}$,
	so that $\left\Vert \tilde{b}-b_{\sigma}\right\Vert =\tilde{l}$.
	We first show that $\left\Vert \tilde{b}-x_{i}\right\Vert <r_{x_{i}}$
	for $i=1,\ldots,\hat{k}$. Since $\left\Vert x_{i}-b_{\sigma}\right\Vert <r_{x_{i}}$
	and $\left\Vert x_{i}-b_{\hat{\sigma}}\right\Vert <r_{x_{i}}$, and
	hence from Claim~\ref{claim:projreach_distance_midpoint}, 
	\begin{equation}
	\left\Vert x_{i}-\tilde{b}\right\Vert \leq\sqrt{\frac{\tilde{l}}{\left\Vert b_{\sigma}-b_{\hat{\sigma}}\right\Vert }\left\Vert x_{i}-b_{\hat{\sigma}}\right\Vert ^{2}+\left(1-\frac{\tilde{l}}{\left\Vert b_{\sigma}-b_{\hat{\sigma}}\right\Vert }\right)\left\Vert x_{i}-b_{\sigma}\right\Vert ^{2}}<r_{x_{i}}.\label{eq:homotopy_cech_rips_cech_dist_barypoint_radius}
	\end{equation}
	Next, we show that for all $\varsigma$ with $\hat{\sigma}\subsetneq\varsigma\subset\sigma$,
	$\left\Vert y_{\varsigma}-\tilde{b}\right\Vert <r_{y_{\varsigma}}$.
	Note that $\left\Vert \tilde{b}-b_{\varsigma}\right\Vert \leq\left\Vert \tilde{b}-b_{\sigma}\right\Vert +\left\Vert b_{\sigma}-b_{\varsigma}\right\Vert \leq2\tilde{l}$,
	and then (i) and the covering condition imply 
	\begin{align*}
		& \left\Vert y_{\varsigma}-\tilde{b}\right\Vert \\
		& \leq\left\Vert y_{\varsigma}-\pi_{\mathbb{X}}(b_{\varsigma})\right\Vert +\left\Vert \pi_{\mathbb{X}}(b_{\varsigma})-b_{\varsigma}\right\Vert +\left\Vert b_{\varsigma}-b_{\sigma}\right\Vert \\
		& <\delta+\left(\tau-\sqrt{(\tau-\epsilon_{\sigma})^{2}-r_{\sigma}^{2}}\right)+2\tilde{l}\\
		& =r_{\min}\leq r_{y_{\varsigma}}.
	\end{align*}
	Finally, we show that for all $\varsigma\subset\hat{\sigma}$, $\left\Vert y_{\varsigma}-\tilde{b}\right\Vert <r_{y_{\varsigma}}$.
	Note that $r_{\hat{\sigma}}\leq\sqrt{r_{\sigma}^{2}-\left\Vert b_{\sigma}-b_{\hat{\sigma}}\right\Vert ^{2}}$,
	so $x_{i}\in\mathbb{B}_{\mathbb{R}^{d}}(b_{\sigma},r_{\sigma})\cap\mathbb{B}_{\mathbb{R}^{d}}(b_{\hat{\sigma}},r_{\hat{\sigma}})\subset\mathbb{B}_{\mathbb{R}^{d}}(\tilde{b},\sqrt{r_{\sigma}^{2}-\tilde{l}^{2}})$
	holds. Hence for all $i=1,\ldots,\hat{k}$,
	\begin{equation}
	\left\Vert \tilde{b}-x_{i}\right\Vert <\sqrt{r_{\sigma}^{2}-\tilde{l}^{2}}\leq r_{\sigma}.\label{eq:homotopy_cech_rips_cech_dist_barypoint_nestedsmp}
	\end{equation}
	We again divide it into two cases. First, consider the case $d(b_{\varsigma},\varsigma)<\delta$.
	Then $y_{\varsigma}=x_{l}$ for some $x_{l}\in\varsigma$, hence from
	\eqref{eq:homotopy_cech_rips_cech_dist_barypoint_radius},
	\[
	\left\Vert y_{\varsigma}-\tilde{b}\right\Vert =\left\Vert x_{l}-\tilde{b}\right\Vert <r_{x_{l}}=r_{y_{\varsigma}}.
	\]
	Otherwise, $\left\Vert x_{i}-b_{\varsigma}\right\Vert \geq\delta$
	for all $x_{i}\in\varsigma$. Then (i) implies 
	\begin{align*}
	d_{\mathbb{X}}(\tilde{b}) & \leq d_{\mathbb{X}}(b_{\sigma})+\left\Vert \tilde{b}-b_{\sigma}\right\Vert \\
	& \leq\tau-\sqrt{(\tau-\epsilon_{\sigma})^{2}-r_{\sigma}^{2}}+\tilde{l}.
	\end{align*}
	And from \eqref{eq:homotopy_cech_rips_cech_dist_barypoint_nestedsmp},
	Lemma~\ref{lem:projreach_distance_projection_center} is applicable
	and gives an upper bound for $\left\Vert \pi_{\mathbb{X}}(b_{\varsigma})-\tilde{b}\right\Vert $
	as 
	\begin{align*}
	\left\Vert \pi_{\mathbb{X}}(b_{\varsigma})-\tilde{b}\right\Vert  & \leq\sqrt{\tilde{r}_{\tilde{b}}^{2}-(\tau^{2}-\tilde{r}_{\tilde{b}}^{2}+(\tau-\epsilon_{\tilde{l}})^{2})\left(\frac{\tau}{\sqrt{\tau^{2}-\tilde{r}_{b_{\varsigma},\tilde{b}}^{2}}}-1\right)},
	\end{align*}
	where 
	\begin{align*}
	\epsilon_{\tilde{l}} & :=\tau-\sqrt{(\tau-\epsilon_{\sigma})^{2}-r_{\sigma}^{2}}+\tilde{l},\\
	\tilde{r}_{\tilde{b}}^{2} & :=\sum_{i=1}^{k}\lambda_{i}\left(\left\Vert x_{i}-\tilde{b}\right\Vert ^{2}+\epsilon_{\sigma}(2\tau-\epsilon_{\sigma})\right),\\
	\tilde{r}_{b_{\varsigma},\tilde{b}}^{2} & :=\min\left\{ \sum_{i=1}^{k}\lambda_{i}\left(\left\Vert x_{i}-b_{\varsigma}\right\Vert ^{2}+\epsilon_{\sigma}(2\tau-\epsilon_{\sigma})\right),\frac{1}{2}(\tilde{r}_{\tilde{b}}^{2}+\epsilon_{\tilde{l}}(2\tau-\epsilon_{\tilde{l}}))\right\} .
	\end{align*}
	Then, RHS is an increasing function of $\left\Vert x_{i}-\tilde{b}\right\Vert ^{2}$
	and decreasing function of $\left\Vert x_{i}-b_{\varsigma}\right\Vert ^{2}$,
	so applying \eqref{eq:homotopy_cech_rips_cech_dist_barypoint_nestedsmp}
	and from $\left\Vert x_{i}-b_{\varsigma}\right\Vert ^{2}\geq\delta$
	for all $x_{i}\in\varsigma$, we have 
	\begin{align*}
	&\left\Vert \pi_{\mathbb{X}}(b_{\varsigma})-\tilde{b}\right\Vert\\ &\leq\sqrt{r_{\sigma}^{2}-\tilde{l}^{2}+\epsilon_{\sigma}(2\tau-\epsilon_{\sigma})-((\tau-\epsilon_{\sigma})^{2}-r_{\sigma}^{2}+\tilde{l}^{2}+(\tau-\epsilon_{\tilde{l}})^{2})\left(\frac{\tau}{\sqrt{\tau^{2}-\tilde{r}_{\delta,c}^{2}}}-1\right)},
	\end{align*}
	where 
	\[
	\tilde{r}_{\delta,c}^{2}:=\min\left\{ \delta^{2}+\epsilon_{\sigma}(2\tau-\epsilon_{\sigma}),\frac{1}{2}(r_{\sigma}^{2}-\tilde{l}^{2}+\epsilon_{\sigma}(2\tau-\epsilon_{\sigma})+\epsilon_{\tilde{l}}(2\tau-\epsilon_{\tilde{l}}))\right\} .
	\]
	Then, 
	\begin{align*}
	& \left\Vert y_{\varsigma}-\tilde{b}\right\Vert \\
	& \leq\left\Vert y_{\varsigma}-\pi_{\mathbb{X}}(b_{\varsigma})\right\Vert +\left\Vert \pi_{\mathbb{X}}(b_{\varsigma})-\tilde{b}\right\Vert \\
	& <\delta+\sqrt{r_{\sigma}^{2}-\tilde{l}^{2}+\epsilon_{\sigma}(2\tau-\epsilon_{\sigma})-((\tau-\epsilon_{\sigma})^{2}-r_{\sigma}^{2}+\tilde{l}^{2}+(\tau-\epsilon_{\tilde{l}})^{2})\left(\frac{\tau}{\sqrt{\tau^{2}-\tilde{r}_{\delta,c}^{2}}}-1\right)}\\
	& \leq r_{\min}<r_{y_{\varsigma}}.
	\end{align*}

	(v)
	
	We need to check that $\left\Vert x_{i}-y_{\sigma}\right\Vert <r_{x_{i}}+r_{y_{\sigma}}$
	and $\left\Vert y_{\varsigma}-y_{\sigma}\right\Vert <r_{y_{\varsigma}}+r_{y_{\sigma}}$.
	
	For $\left\Vert x_{i}-y_{\sigma}\right\Vert <r_{x_{i}}+r_{y_{\sigma}}$,
	note that (ii) (iii) and the condition on $\delta$ gives 
	\begin{align*}
		\left\Vert x_{1}-y_{\sigma}\right\Vert  & <\sqrt{\tilde{r}_{b_{\sigma}}^{2}-\epsilon_{\sigma}(2\tau-\epsilon_{\sigma})}+\delta\\
		& \leq\sqrt{\tilde{r}_{b_{\sigma}}^{2}-(2\tau^{2}-\tilde{r}_{b_{\sigma}}^{2})\left(\frac{\tau}{\sqrt{\tau^{2}-\tilde{r}_{\delta,b}^{2}}}-1\right)}+2\delta\\
		& \leq2r_{\min}\leq r_{x_{i}}+r_{y_{0}},
	\end{align*}
	where $\tilde{r}_{b_{\sigma}}$ is from (ii) and $\tilde{r}_{\delta,b}$
	is from (iii).
	
	For $\left\Vert y_{\varsigma}-y_{\sigma}\right\Vert <r_{y_{\varsigma}}+r_{y_{\sigma}}$,
	note that (iii) and the condition on $\delta$ gives 
	\begin{align*}
		\left\Vert y_{\varsigma}-y_{\sigma}\right\Vert  & <\sqrt{\tilde{r}_{b_{\sigma}}^{2}-(2\tau^{2}-\tilde{r}_{b_{\sigma}}^{2})\left(\frac{\tau}{\sqrt{\tau^{2}-\tilde{r}_{\delta,b}^{2}}}-1\right)}+2\delta\\
		& \leq2r_{\min}\leq r_{y_{\varsigma}}+r_{y_{0}}.
	\end{align*}

\end{proof}

After repeating \eqref{eq:homotopy_cech_rips_equivalence_first}
several times, we get the homotopy equivalence as 
\begin{equation}
[x_{1}\cdots x_{k}]\simeq\sum_{\omega\in S_{k}}[x_{1}y_{[x_{\omega(1)}x_{\omega(2)}]}\cdots y_{[x_{1}\cdots x_{k}]}]\text{ in }\mathcal{S},\label{eq:homotopy_cech_rips_equivalence_subdivision}
\end{equation}
where $S_{k}$ is the permutation group of size $k$. Since we are
trying to map $\mathcal{S}=\textrm{\v{C}ech}_{\mathbb{R}^{d}}(\mathcal{X},r)$
or $\textrm{Rips}(\mathcal{X},r)$ to into $\textrm{\v{C}ech}_{\mathbb{X}}(\mathcal{X},r)$
which is consisting of smaller simplices, we want to each simplex
$[x_{1}y_{[x_{\omega(1)}x_{\omega(2)}]}\cdots y_{[x_{1}\cdots x_{k}]}]$
be of smaller size. We measure the size by ${\rm Rad}_{r}(\sigma)$
or ${\rm Rad}(\sigma)$ as introduced in \eqref{eq:homotopy_barycenter_radius_r}
and \eqref{eq:homotopy_barycenter_radius}. We collect several bounds
for ${\rm Rad}_{r}(\sigma)$ or ${\rm Rad}(\sigma)$ in Claim~\ref{claim:homotopy_cech_rips_radius}.

\begin{claim}
	
	\label{claim:homotopy_cech_rips_radius}
	
	Let $\tau>0$, $\mathbb{X}\subset\mathbb{R}^{d}$ be a subset with
	reach $\tau_{\mathbb{X}}\geq\tau>0$, and $\mathcal{X}\subset\mathbb{R}^{d}$
	be a set of points. Let $\{r_{x}>0:x\in\mathcal{X}\}$ be a set of
	radii indexed by $x\in\mathcal{X}$. Let $\epsilon\geq0$ be satisfying
	$d_{\mathbb{X}}(x)\leq\epsilon$ for all $x\in\mathcal{X}$. Let $\delta>0$,
	and suppose $\mathbb{X}\subset\bigcup_{x\in\mathcal{X}}\mathbb{B}_{\mathbb{X}}(x,\delta)$.
	\begin{enumerate}
		\item[(i)]  
		\[
		{\rm Rad}_{r}(\sigma)\leq\frac{r_{\max}}{r_{\min}}{\rm Rad}(\sigma),
		\]
		where $r_{\min}=\min\{r_{x}:x\in\mathcal{X}\}$ and $r_{\max}=\max\{r_{x}:x\in\mathcal{X}\}$.
		\item[(ii)]  ${\rm bc}_{r}$ and ${\rm Rad}_{r}$ satisfy the following: for
		each $\sigma=\{x_{1},\ldots,x_{k}\}\subset\mathcal{X}$, $y_{\sigma}\in\mathcal{X}$
		can be chosen so that if $r_{\sigma}:={\rm Rad}_{r}(\sigma)<\tau-\epsilon_{\sigma}$,
		then 
		\begin{align*}
		& {\rm Rad}_{r}(\{x_{1}y_{\{x_{1}x_{2}\}}\ldots y_{\{x_{1}\cdots x_{k}\}}\})\\
		& \leq\sqrt{\frac{d}{2(d+1)}}\frac{r_{\max}}{r_{\min}}\left(\sqrt{\tilde{r}_{b}^{2}-(2\tau^{2}-\tilde{r}_{b}^{2})\left(\frac{\tau}{\sqrt{\tau^{2}-\tilde{r}_{\delta,b}^{2}}}-1\right)}+2\delta\right),
		\end{align*}
		where 
		\begin{align*}
			\tilde{r}_{b}^{2} & :=\frac{2\tau\left(r_{\sigma}^{2}+\epsilon(2\tau-\epsilon)\right)}{\tau+\sqrt{\tau^{2}-\left(r_{\sigma}^{2}+\epsilon(2\tau-\epsilon)\right)}},\\
			\tilde{r}_{\delta,b}^{2} & :=\min\left\{ \delta^{2}+\epsilon(2\tau-\epsilon),\frac{1}{2}\tilde{r}_{b}^{2}\right\} .
		\end{align*}
		Further, if 
		\[
		\delta+\sqrt{r_{\sigma}^{2}-\tilde{l}^{2}+\epsilon(2\tau-\epsilon)-((\tau-\epsilon)^{2}-r_{\sigma}^{2}+\tilde{l}^{2}+(\tau-\epsilon_{\tilde{l}})^{2})\left(\frac{\tau}{\sqrt{\tau^{2}-\tilde{r}_{\delta,c}^{2}}}-1\right)}\leq r_{\min},
		\]
		where 
		\begin{align*}
		\tilde{l} & :=\frac{1}{2}\left(r_{\min}-\tau+\sqrt{(\tau-\epsilon)^{2}-r_{\sigma}^{2}}-\delta\right),\\
		\epsilon_{\tilde{l}} & :=\tau-\sqrt{(\tau-\epsilon)^{2}-r_{\sigma}^{2}}+\tilde{l},\\
		\tilde{r}_{\delta,c}^{2} & :=\min\left\{ \delta^{2}+\epsilon(2\tau-\epsilon),\frac{1}{2}(r_{\sigma}^{2}-\tilde{l}^{2}+\epsilon(2\tau-\epsilon)+\epsilon_{\tilde{l}}(2\tau-\epsilon_{\tilde{l}}))\right\} .
		\end{align*}
		then $\sigma\in\textrm{\v{C}ech}_{\mathbb{R}^{d}}(\mathcal{X},r)$
		implies that 
		\[
		\{x_{1}\cdots x_{j}y_{\{x_{1}\cdots x_{j}\}}y_{\{x_{1}\cdots x_{j+1}\}}\cdots y_{\{x_{1}\cdots x_{k}\}}\}\in\textrm{\v{C}ech}_{\mathbb{R}^{d}}(\mathcal{X},r).
		\]
		\item[(iii)]  ${\rm bc}$ and ${\rm Rad}$ satisfy the following: for each $\sigma=\{x_{1},\ldots,x_{k}\}\subset\mathcal{X}$,
		$y_{\sigma}\in\mathcal{X}$ can be chosen so that if $r_{\sigma}:={\rm Rad}(\sigma)<\tau-\epsilon_{\sigma}$,
		then 
		\[
		{\rm Rad}(\{x_{1}y_{\{x_{1}x_{2}\}}\ldots y_{\{x_{1}\cdots x_{k}\}}\})\leq\sqrt{\frac{d}{2(d+1)}}\left(\sqrt{\tilde{r}_{b}^{2}-(2\tau^{2}-\tilde{r}_{b}^{2})\left(\frac{\tau}{\sqrt{\tau^{2}-\tilde{r}_{\delta,b}^{2}}}-1\right)}+2\delta\right),
		\]
		where 
		\begin{align*}
			\tilde{r}_{b}^{2} & :=\frac{2\tau\left(r_{\sigma}^{2}+\epsilon(2\tau-\epsilon)\right)}{\tau+\sqrt{\tau^{2}-\left(r_{\sigma}^{2}+\epsilon(2\tau-\epsilon)\right)}},\\
			\tilde{r}_{\delta,b}^{2} & :=\min\left\{ \delta^{2}+\epsilon(2\tau-\epsilon),\frac{1}{2}\tilde{r'}_{b}^{2}\right\} .
		\end{align*}
		Further, if 
		\[
		\sqrt{\tilde{r}_{b}^{2}-(2\tau^{2}-\tilde{r}_{b}^{2})\left(\frac{\tau}{\sqrt{\tau^{2}-\tilde{r}_{\delta,b}^{2}}}-1\right)}+2\delta\leq2r_{\min},
		\]
		then $\sigma\in\textrm{Rips}(\mathcal{X},r)$ implies that 
		\[
		[x_{1}\cdots x_{j}y_{\{x_{1}\cdots x_{j}\}}y_{\{x_{1}\cdots x_{j+1}\}}\cdots y_{\{x_{1}\cdots x_{k}\}}]\in\textrm{Rips}(\mathcal{X},r).
		\]
	\end{enumerate}
\end{claim}

\begin{proof}[Proof of Claim~\ref{claim:homotopy_cech_rips_radius}]
	
	(i)
	
	From ${\rm bc}_{r}(\sigma)=\arg\min_{y\in\mathbb{R}^{d}}\max_{x\in\sigma}\frac{\left\Vert x-y\right\Vert }{r_{x}}$,
	
	\begin{align*}
		{\rm Rad}_{r}(\sigma) & =\max_{x\in\sigma}\left\Vert x-{\rm bc}_{r}(\sigma)\right\Vert \leq r_{\max}\max_{x\in\sigma}\frac{\left\Vert x-{\rm bc}_{r}(\sigma)\right\Vert }{r_{x}}\\
		& \leq r_{\max}\max_{x\in\sigma}\frac{\left\Vert x-{\rm bc}(\sigma)\right\Vert }{r_{x}}\leq\frac{r_{\max}}{r_{\min}}\max_{x\in\sigma}\left\Vert x-{\rm bc}(\sigma)\right\Vert \\
		& =\frac{r_{\max}}{r_{\min}}{\rm Rad}(\sigma).
	\end{align*}
	Conversely, from ${\rm bc}(\sigma)=\arg\min_{y\in\mathbb{R}^{d}}\max_{x\in\sigma}\left\Vert x-y\right\Vert $,
	\begin{align*}
		{\rm Rad}(\sigma) & =\max_{x\in\sigma}\left\Vert x-{\rm bc}(\sigma)\right\Vert \leq\max_{x\in\sigma}\left\Vert x-{\rm bc}_{r}(\sigma)\right\Vert ={\rm Rad}_{r}(\sigma).
	\end{align*}
	
	(ii)
	
	For each $\sigma\subset\mathcal{X}$, we set $b_{\sigma}:={\rm bc}_{r}(\sigma)$
	and choose the corresponding $y_{\sigma}\in\mathcal{X}$ according
	to Claim~\ref{claim:homotopy_cech_rips}. Note that ${\rm Rad}_{r}(\sigma)=\max_{x\in\sigma}\left\Vert x-b_{\sigma}\right\Vert $.
	Then Claim~\ref{claim:homotopy_cech_rips} (ii) implies that 
	\[
	\left\Vert x_{i}-y_{\sigma}\right\Vert <\sqrt{\tilde{r}_{b}^{2}-\epsilon_{\sigma}(2\tau-\epsilon_{\sigma})}+\delta,
	\]
	and Claim~\ref{claim:homotopy_cech_rips} (iii) implies that for any
	$\varsigma\subset\sigma$, 
	\[
	\left\Vert y_{\varsigma}-y_{\sigma}\right\Vert <\sqrt{\tilde{r}_{b}^{2}-(2\tau^{2}-\tilde{r}_{b}^{2})\left(\frac{\tau}{\sqrt{\tau^{2}-\tilde{r}_{\delta,b}^{2}}}-1\right)}+2\delta.
	\]
	Then from Claim~\ref{claim:homotopy_cech_rips} (iii), $\sqrt{\tilde{r}_{b}^{2}-\epsilon_{\sigma}(2\tau-\epsilon_{\sigma})}+\delta\leq\sqrt{\tilde{r}_{b}^{2}-(2\tau^{2}-\tilde{r}_{b}^{2})\left(\frac{\tau}{\sqrt{\tau^{2}-\tilde{r}_{\delta,b}^{2}}}-1\right)}+2\delta$,
	and hence 
	\[
	\{x_{1},y_{\{x_{1}x_{2}\}},\ldots,y_{\{x_{1}\cdots x_{k}\}}\}\in{\rm Rips}\left(\mathcal{X},\frac{1}{2}\left(\sqrt{\tilde{r}_{b}^{2}-(2\tau^{2}-\tilde{r}_{b}^{2})\left(\frac{\tau}{\sqrt{\tau^{2}-\tilde{r}_{\delta,b}^{2}}}-1\right)}+2\delta\right)\right).
	\]
	And Lemma~\ref{lem:homotopy_interleaving_rips_ambientcech} implies
	that 
	\begin{align*}
	& \{x_{1},y_{\{x_{1}x_{2}\}},\ldots,y_{\{x_{1}\cdots x_{k}\}}\}\\
	& \in\textrm{\v{C}ech}_{\mathbb{R}^{d}}\left(\mathcal{X},\sqrt{\frac{d}{2(d+1)}}\left(\sqrt{\tilde{r}_{b}^{2}-(2\tau^{2}-\tilde{r}_{b}^{2})\left(\frac{\tau}{\sqrt{\tau^{2}-\tilde{r}_{\delta,b}^{2}}}-1\right)}+2\delta\right)\right).
	\end{align*}
	And hence 
	\begin{align*}
	& {\rm Rad}(\{x_{1},y_{\{x_{1}x_{2}\}},\ldots,y_{\{x_{1}\cdots x_{k}\}}\})\\
	& \leq\sqrt{\frac{d}{2(d+1)}}\left(\sqrt{\tilde{r}_{b}^{2}-(2\tau^{2}-\tilde{r}_{b}^{2})\left(\frac{\tau}{\sqrt{\tau^{2}-\tilde{r}_{\delta,b}^{2}}}-1\right)}+2\delta\right).
	\end{align*}
	Then from (i), 
	\begin{align*}
	& {\rm Rad}_{r}(\{x_{1},y_{\{x_{1}x_{2}\}},\ldots,y_{\{x_{1}\cdots x_{k}\}}\})\\
	& \leq\sqrt{\frac{d}{2(d+1)}}\frac{r_{\max}}{r_{\min}}\left(\sqrt{\tilde{r}_{b}^{2}-(2\tau^{2}-\tilde{r}_{b}^{2})\left(\frac{\tau}{\sqrt{\tau^{2}-\tilde{r}_{\delta,b}^{2}}}-1\right)}+2\delta\right).
	\end{align*}
	Further, Claim~\ref{claim:homotopy_cech_rips} (iv) implies that if
	\[
	\delta+\sqrt{r_{\sigma}^{2}-\tilde{l}^{2}+\epsilon(2\tau-\epsilon)-((\tau-\epsilon)^{2}-r_{\sigma}^{2}+\tilde{l}^{2}+(\tau-\epsilon_{\tilde{l}})^{2})\left(\frac{\tau}{\sqrt{\tau^{2}-\tilde{r}_{\delta,c}^{2}}}-1\right)}\leq r_{\min},
	\]
	then $\sigma\in\textrm{\v{C}ech}_{\mathbb{R}^{d}}(\mathcal{X},r)$
	implies that $b_{\varsigma}={\rm bc}_{r}(\varsigma)\in\bigcap_{x\in\varsigma}\mathbb{B}_{\mathbb{R}^{d}}(x,r_{x})$.
	And hence 
	\[
	\{x_{1},\ldots,x_{j},y_{\{x_{1}\cdots x_{j}\}},y_{\{x_{1}\cdots x_{j+1}\}},\ldots,y_{\{x_{1}\cdots x_{k}\}}\}\in\textrm{\v{C}ech}_{\mathbb{R}^{d}}(\mathcal{X},r).
	\]
	
	(iii)
	
	For each $\sigma\subset\mathcal{X}$, we set $b_{\sigma}:={\rm bc}(\sigma)$
	and choose the corresponding $y_{\sigma}\in\mathcal{X}$ according
	to Claim~\ref{claim:homotopy_cech_rips}. Note that ${\rm Rad}(\sigma)=\max_{x\in\sigma}\left\Vert x-b_{\sigma}\right\Vert $.
	Then Claim~\ref{claim:homotopy_cech_rips} (ii) implies that 
	\[
	\left\Vert x_{i}-y_{\sigma}\right\Vert <\sqrt{\tilde{r}_{b}^{2}-\epsilon_{\sigma}(2\tau-\epsilon_{\sigma})}+\delta,
	\]
	and Claim~\ref{claim:homotopy_cech_rips} (iii) implies that for any
	$\varsigma\subset\sigma$, 
	\[
	\left\Vert y_{\varsigma}-y_{\sigma}\right\Vert <\sqrt{\tilde{r}_{b}^{2}-(2\tau^{2}-\tilde{r}_{b}^{2})\left(\frac{\tau}{\sqrt{\tau^{2}-\tilde{r}_{\delta,b}^{2}}}-1\right)}+2\delta.
	\]
	Then from Claim~\ref{claim:homotopy_cech_rips} (iii), $\sqrt{\tilde{r}_{b}^{2}-\epsilon_{\sigma}(2\tau-\epsilon_{\sigma})}+\delta\leq\sqrt{\tilde{r}_{b}^{2}-(2\tau^{2}-\tilde{r}_{b}^{2})\left(\frac{\tau}{\sqrt{\tau^{2}-\tilde{r}_{\delta,b}^{2}}}-1\right)}+2\delta$,
	and hence 
	\[
	\{x_{1},y_{\{x_{1}x_{2}\}},\ldots,y_{\{x_{1}\cdots x_{k}\}}\}\in{\rm Rips}\left(\mathcal{X},\frac{1}{2}\left(\sqrt{\tilde{r}_{b}^{2}-(2\tau^{2}-\tilde{r}_{b}^{2})\left(\frac{\tau}{\sqrt{\tau^{2}-\tilde{r}_{\delta,b}^{2}}}-1\right)}+2\delta\right)\right).
	\]
	And Lemma~\ref{lem:homotopy_interleaving_rips_ambientcech} implies
	that 
	\begin{align*}
	& \{x_{1},y_{\{x_{1}x_{2}\}},\ldots,y_{\{x_{1}\cdots x_{k}\}}\}\\
	& \in\textrm{\v{C}ech}_{\mathbb{R}^{d}}\left(\mathcal{X},\sqrt{\frac{d}{2(d+1)}}\left(\sqrt{\tilde{r}_{b}^{2}-(2\tau^{2}-\tilde{r}_{b}^{2})\left(\frac{\tau}{\sqrt{\tau^{2}-\tilde{r}_{\delta,b}^{2}}}-1\right)}+2\delta\right)\right).
	\end{align*}
	And hence
	\begin{align*}
	& {\rm Rad}(\{x_{1},y_{\{x_{1}x_{2}\}},\ldots,y_{\{x_{1}\cdots x_{k}\}}\})\\
	& \leq\sqrt{\frac{d}{2(d+1)}}\left(\sqrt{\tilde{r}_{b}^{2}-(2\tau^{2}-\tilde{r}_{b}^{2})\left(\frac{\tau}{\sqrt{\tau^{2}-\tilde{r}_{\delta,b}^{2}}}-1\right)}+2\delta\right).
	\end{align*}
	Further, Claim~\ref{claim:homotopy_cech_rips} (v) implies that if
	\[
	\sqrt{\tilde{r}_{b}^{2}-(2\tau^{2}-\tilde{r}_{b}^{2})\left(\frac{\tau}{\sqrt{\tau^{2}-\tilde{r}_{\delta,b}^{2}}}-1\right)}+2\delta\leq2r_{\min},
	\]
	then $\sigma\in\textrm{Rips}(\mathcal{X},r)$ implies that 
	\[
	\{x_{1},\ldots,x_{j},y_{\{x_{1}\cdots x_{j}\}},y_{\{x_{1}\cdots x_{j+1}\}},\ldots,y_{\{x_{1}\cdots x_{k}\}}\}\in\textrm{Rips}(\mathcal{X},r).
	\]
	
\end{proof}

For Theorem~\ref{thm:homotopy_cech}, Claim~\ref{claim:homotopy_cech_rips}
and \ref{claim:homotopy_cech_rips_radius} imply that there exists
a map $\rho:\textrm{\v{C}ech}_{\mathbb{R}^{d}}(\mathcal{X},r)\to\textrm{\v{C}ech}_{\mathbb{X}}(\mathcal{X},r)$
that $\imath_{\textrm{\v{C}ech}_{\mathbb{X}}(\mathcal{X},r)\to\textrm{\v{C}ech}_{\mathbb{R}^{d}}(\mathcal{X},r)}\circ\rho:\textrm{\v{C}ech}_{\mathbb{R}^{d}}(\mathcal{X},r)\to\textrm{\v{C}ech}_{\mathbb{R}^{d}}(\mathcal{X},r)$
is homotopic to the identity $id_{\textrm{\v{C}ech}_{\mathbb{R}^{d}}(\mathcal{X},r)}$.
Then from Lemma~\ref{lem:homotopy_interleaving_ambientcech_restrictedcech},
the following diagram commutes:
\[
\xymatrix{ & \mathbb{X}\ar[dl]_{\phi}\\
	\textrm{\v{C}ech}_{\mathbb{X}}(\mathcal{X},r)\ar@/_/[dr]_{\imath_{\textrm{\v{C}ech}_{\mathbb{X}}(\mathcal{X},r)\to\textrm{\v{C}ech}_{\mathbb{R}^{d}}(\mathcal{X},r)}}\ar[rr]^{\imath_{\textrm{\v{C}ech}_{\mathbb{X}}(\mathcal{X},r)\to\textrm{\v{C}ech}_{\mathbb{X}}(\mathcal{X},r')}} &  & \textrm{\v{C}ech}_{\mathbb{X}}(\mathcal{X},r')\ar[ul]_{\psi'}\\
	& \mathcal{S}\ar@/_/[ul]_{\rho}\ar[ur]_{\imath_{\textrm{\v{C}ech}_{\mathbb{R}^{d}}(\mathcal{X},r)\to\textrm{\v{C}ech}_{\mathbb{X}}(\mathcal{X},r')}}
}
.
\]
Then Lemma~\ref{lem:homotopy_simplexnerve} implies that $\mathbb{X}$
is homotopy equivalent to the ambient \v{C}ech complex $\textrm{\v{C}ech}_{\mathbb{R}^{d}}(\mathcal{X},r)$.
We restate Theorem~\ref{thm:homotopy_cech} and formally write its
proof below.

\textbf{Theorem~\ref{thm:homotopy_cech}.} \textit{Let $\mathbb{X}\subset\mathbb{R}^{d}$
	be a subset with reach $\tau>0$ and let $\mathcal{X}\subset\mathbb{R}^{d}$
	be a closed discrete set of points. Let $\{r_{x}>0:x\in\mathcal{X}\}$ be a set of
	radii indexed by $x\in\mathcal{X}$ with $r_{\min}:=\inf_{x\in\mathcal{X}}\{r_{x}\}>0$
	and $r_{\max}:=\sup_{x\in\mathcal{X}}\{r_{x}\}<\infty$, and let $\epsilon:=\sup\{d_{\mathbb{X}}(x):x\in\mathcal{X}\}$.
	Suppose $\mathbb{X}$ is covered by the union of balls centered at
	$x\in\mathcal{X}$ and radius $\delta$ as 
	\[
	\mathbb{X}\subset\bigcup_{x\in\mathcal{X}}\mathbb{B}_{\mathbb{R}}(x,\delta).
	\]
	Suppose that the maximum radius $r_{\max}$ is bounded as 
	\[
	r_{\max}\leq\tau-\epsilon.
	\]
	Also, suppose $\delta$ satisfies the following condition: 
	\begin{align*}
	& \delta+\sqrt{r_{\max}^{2}-\tilde{l}^{2}+\epsilon(2\tau-\epsilon)-((\tau-\epsilon)^{2}-r_{\max}^{2}+\tilde{l}^{2}+(\tau-\epsilon_{\tilde{l}})^{2})\left(\frac{\tau}{\sqrt{\tau^{2}-\tilde{r}_{\delta,c}}}-1\right)} \\
	& \leq r_{\min}, \\
	& \sqrt{\frac{d}{2(d+1)}}\frac{r_{\max}}{r_{\min}}\left(\sqrt{\tilde{r}_{b}^{2}-(2\tau^{2}-\tilde{r}_{b}^{2})\left(\frac{\tau}{\sqrt{\tau^{2}-\tilde{r}_{\delta,b}^{2}}}-1\right)}+2\delta\right)\leq r_{\min}'',
	\end{align*}
	\begin{align*}
	& \tilde{l}:=\frac{1}{2}\left(r_{\min}-\tau+\sqrt{(\tau-\epsilon)^{2}-r_{\max}^{2}}-\delta\right),\qquad\epsilon_{\tilde{l}}:=\tau-\sqrt{(\tau-\epsilon)^{2}-r_{\max}^{2}}+\tilde{l},\\
	& \tilde{r}_{\delta,c}^{2}:=\min\left\{ \delta^{2}+\epsilon(2\tau-\epsilon),\frac{1}{2}(r_{\max}^{2}-\tilde{l}^{2}+\epsilon(2\tau-\epsilon)+\epsilon_{\tilde{l}}(2\tau-\epsilon_{\tilde{l}}))\right\} ,\\
	& r_{\min}'':=\sqrt{\tau^{2}-\epsilon(2\tau-\epsilon)-\frac{(2\tau^{2}-r_{\min}^{2}-\epsilon(2\tau-\epsilon))^{2}}{4\tau^{2}}},\\
	& \tilde{r}_{b}^{2}:=\frac{2\tau\left((r_{\min}'')^{2}+\epsilon(2\tau-\epsilon)\right)}{\tau+\sqrt{\tau^{2}-\left((r_{\min}'')^{2}+\epsilon(2\tau-\epsilon)\right)}},\qquad\tilde{r}_{\delta,b}^{2}:=\min\left\{ \delta^{2}+\epsilon(2\beta-\epsilon),\frac{1}{2}\tilde{r}_{b}^{2}\right\} .
	\end{align*}
	Then $\mathbb{X}$ is homotopy equivalent to the ambient \v{C}ech
	complex $\textrm{\v{C}ech}_{\mathbb{R}^{d}}(\mathcal{X},r)$.}

\begin{proof}[Proof of Theorem~\ref{thm:homotopy_cech}]

	For a simplex $\sigma=[x_{1}\cdots x_{k}]\in\textrm{\v{C}ech}_{\mathbb{R}^{d}}(\mathcal{X},r)$,
	let $r_{\sigma}:=Rad_{r}(\sigma)$, and choose $y_{\sigma}\in\mathcal{X}$
	according to Claim~\ref{claim:homotopy_cech_rips_radius}. Then as
	long as $r_{\sigma}<\tau-\epsilon_{\sigma}$, Claim~\ref{claim:homotopy_cech_rips_radius}
	(ii) asserts that $[x_{1}\cdots x_{k}y_{[x_{1}\cdots x_{k}]}]\in\textrm{\v{C}ech}_{\mathbb{R}^{d}}(\mathcal{X},r)$
	holds. Now, define the homotopy map $F_{1}:[x_{1}\cdots x_{k}]\times[0,1]\to[x_{1}\cdots x_{k}y_{[x_{1}\cdots x_{k}]}]$
	as 
	\[
	F_{1}\left(\sum_{i=1}^{k}\lambda_{i}x_{i},t\right)=\sum_{i=1}^{k}(\lambda_{i}-\min\lambda_{i})x_{i}+k\min\lambda_{i}\left(ty_{\sigma}+(1-t)\frac{1}{k}\sum_{i=1}^{k}x_{i}\right).
	\]
	Then $F_{1}$ gives homotopy between $\imath_{\sigma\to\sigma*y_{\sigma}}$
	and $f_{1}:\sigma\to\sigma*y_{\sigma}$ defined as 
	\[
	f_{1}\left(\sum_{i=1}^{k}\lambda_{i}x_{i}\right)=\sum_{i=1}^{k}(\lambda_{i}-\min\lambda_{i})x_{i}+(k\min\lambda_{i})y_{\sigma},
	\]
	giving homotopy equivalence as 
	\[
	[x_{1}\cdots x_{k}]\simeq\sum_{i=1}^{k}[x_{1}\cdots\hat{x}_{i}\cdots x_{k}y_{[x_{1}\cdots x_{k}]}]\text{ in }\textrm{\v{C}ech}_{\mathbb{R}^{d}}(\mathcal{X},r).
	\]
	And again, Claim~\ref{claim:homotopy_cech_rips_radius} (ii) asserts
	that $[x_{1}\cdots x_{k-1}y_{[x_{1}\cdots x_{k-1}]}y_{[x_{1}\cdots x_{k}]}]\in\textrm{\v{C}ech}_{\mathbb{R}^{d}}(\mathcal{X},r)$
	holds. Now, the homotopy map $F_{2}:[x_{1}\cdots x_{k-1}y_{[x_{1}\cdots x_{k}]}]\times[0,1]\to[x_{1}\cdots x_{k-1}y_{[x_{1}\cdots x_{k-1}]}y_{[x_{1}\cdots x_{k}]}]$
	is defined as 
	\begin{align*}
	F_{2}\left(\sum_{i=1}^{k-1}\lambda_{i}x_{i}+\lambda_{k}y_{[x_{1}\cdots x_{k}]},t\right) & =\lambda_{k}y_{[x_{1}\cdots x_{k}]}+\sum_{i=1}^{k-1}(\lambda_{i}-\min\lambda_{i})x_{i}\\
	& \quad+(k-1)\min\lambda_{i}\left(ty_{[x_{1}\cdots x_{k-1}]}+(1-t)\frac{1}{k-1}\sum_{i=1}^{k-1}x_{i}\right).
	\end{align*}
	Then $F_{2}$ gives homotopy between $\imath_{[x_{1}\cdots x_{k-1}y_{[x_{1}\cdots x_{k}]}]\to[x_{1}\cdots x_{k-1}y_{[x_{1}\cdots x_{k-1}]}y_{[x_{1}\cdots x_{k}]}]}$
	and \\
	$f_{2}:[x_{1}\cdots x_{k-1}y_{[x_{1}\cdots x_{k}]}]\to[x_{1}\cdots x_{k-1}y_{[x_{1}\cdots x_{k-1}]}y_{[x_{1}\cdots x_{k}]}]$
	defined as 
	\[
	f_{2}\left(\sum_{i=1}^{k-1}\lambda_{i}x_{i}+\lambda_{k}y_{[x_{1}\cdots x_{k}]}\right)=\lambda_{k}y_{[x_{1}\cdots x_{k}]}+\sum_{i=1}^{k-1}(\lambda_{i}-\min\lambda_{i})x_{i}+(k-1)\min\lambda_{i}y_{[x_{1}\cdots x_{k-1}]},
	\]
	giving the homotopy equivalence as 
	\[
	[x_{1}\cdots x_{k-1}y_{[x_{1}\cdots x_{k}]}]\simeq\sum_{i=1}^{k-1}[x_{1}\cdots\hat{x}_{i}\cdots x_{k-1}y_{[x_{1}\cdots x_{k-1}]}y_{[x_{1}\cdots x_{k}]}]\text{ in }\textrm{\v{C}ech}_{\mathbb{R}^{d}}(\mathcal{X},r).
	\]
	By repeating this and concatenating the homotopy maps, we have the
	homotopy map $F_{\sigma}:\sigma\times[0,1]\to\textrm{\v{C}ech}_{\mathbb{R}^{d}}(\mathcal{X},r)$
	giving homotopy between $\imath_{\sigma\to\textrm{\v{C}ech}_{\mathbb{R}^{d}}(\mathcal{X},r)}$
	and $f_{\sigma}:\sigma\to\textrm{\v{C}ech}_{\mathbb{R}^{d}}(\mathcal{X},r)$
	with $f_{\sigma}(\sigma)=\sum_{\omega\in S_{k}}[x_{\omega(1)}y_{[x_{\omega(1)}x_{\omega(2)}]}\ldots y_{[x_{1}\cdots x_{k}]}]$,
	i.e. $F_{\sigma}$ is giving homotopy equivalence between $[x_{1},\ldots,x_{k}]$
	and $\sum_{\omega\in S_{k}}[x_{\omega(1)}y_{[x_{\omega(1)}x_{\omega(2)}]}\ldots y_{[x_{1}\cdots x_{k}]}]$
	in $\textrm{\v{C}ech}_{\mathbb{R}^{d}}(\mathcal{X},r)$.
	
	Now, we extend $f_{\sigma}$ and $F_{\sigma}$ to the entire ambient \v{C}ech complex $\textrm{\v{C}ech}_{\mathbb{R}^{d}}(\mathcal{X},r)$. Define $f:\textrm{\v{C}ech}_{\mathbb{R}^{d}}(\mathcal{X},r)\to\textrm{\v{C}ech}_{\mathbb{R}^{d}}(\mathcal{X},r)$
	and $F:\textrm{\v{C}ech}_{\mathbb{R}^{d}}(\mathcal{X},r)\times[0,1]\to\textrm{\v{C}ech}_{\mathbb{R}^{d}}(\mathcal{X},r)$
	as $f|_{\sigma}=f_{\sigma}$ and $F|_{\sigma\times[0,1]}=F_{\sigma}$
	for each $\sigma\in\textrm{\v{C}ech}_{\mathbb{R}^{d}}(\mathcal{X},r)$.
	Then for $\sigma,\tau\in\textrm{\v{C}ech}_{\mathbb{R}^{d}}(\mathcal{X},r)$,
	$F_{\sigma}$ and $F_{\tau}$ coincides on $\sigma\cap\tau\times[0,1]$,
	so $f$ and $F$ are well defined. Also, from $\mathcal{X}$ being
	closed and discrete, $\textrm{\v{C}ech}_{\mathbb{R}^{d}}(\mathcal{X},r)$
	is locally finite. Hence for any compact set $C\subset\textrm{\v{C}ech}_{\mathbb{R}^{d}}(\mathcal{X},r)$,
	$C$ intersects with only finite number of simplices $\sigma_{1},\ldots,\sigma_{k}\in\textrm{\v{C}ech}_{\mathbb{R}^{d}}(\mathcal{X},r)$,
	so $F^{-1}(C)=\bigcup_{i=1}^{k}F_{\sigma_{i}}^{-1}(C)$ is compact,
	and hence $F$ is continuous. So $F$ gives the homotopy between $id_{\textrm{\v{C}ech}_{\mathbb{R}^{d}}(\mathcal{X},r)}$
	and $f$.
	
	Now, applying Claim~\ref{claim:homotopy_cech_rips_radius} (ii) gives
	that 
	\begin{align*}
	& {\rm Rad}_{r}([x_{\omega(1)}y_{[x_{\omega(1)}x_{\omega(2)}]}\cdots y_{[x_{1}\cdots x_{k}]}])\\
	& \leq\sqrt{\frac{d}{2(d+1)}}\frac{r_{\max}}{r_{\min}}\left(\sqrt{\tilde{r}_{b}^{2}-(2\tau^{2}-\tilde{r}_{b}^{2})\left(\frac{\tau}{\sqrt{\tau^{2}-\tilde{r}_{\delta,b}^{2}}}-1\right)}+2\delta\right),
	\end{align*}
	where 
	\begin{align*}
	\tilde{r}_{b}^{2} & :=\frac{2\tau\left(r_{\sigma}^{2}+\epsilon(2\tau-\epsilon)\right)}{\tau+\sqrt{\tau^{2}-\left(r_{\sigma}^{2}+\epsilon(2\tau-\epsilon)\right)}},\\
	\tilde{r}_{\delta,b}^{2} & :=\min\left\{ \delta^{2}+\epsilon(2\tau-\epsilon),\frac{1}{2}\tilde{r}_{b}^{2}\right\} .
	\end{align*}
	Hence by repeating this sufficiently many times (say $N$), we can
	guarantee $f^{(N)}(\textrm{\v{C}ech}_{\mathbb{R}^{d}}(\mathcal{X},r))$
	to be consisting of simplices with their radii (i.e. $Rad(\sigma)$)
	at most $\tilde{r}_{\sigma}$, where $\tilde{r}_{\sigma}$ is the
	solution of 
	\[
	f(t)=\sqrt{\frac{d}{2(d+1)}}\frac{r_{\max}}{r_{\min}}\left(\sqrt{\tilde{r}_{b}^{2}(t)-(2\tau^{2}-\tilde{r}_{b}^{2}(t))\left(\frac{\tau}{\sqrt{\tau^{2}-\tilde{r}_{\delta,b}^{2}(t)}}-1\right)}+2\delta\right)-t=0,
	\]
	with $\tilde{r}_{b}^{2}(t)=\frac{2\tau\left(t^{2}+\epsilon(2\tau-\epsilon)\right)}{\tau+\sqrt{\tau^{2}-\left(t^{2}+\epsilon(2\tau-\epsilon)\right)}}$
	and $\tilde{r}_{\delta,b}^{2}(t)=\min\left\{ \delta^{2}+\epsilon(2\tau-\epsilon),\frac{1}{2}\tilde{r}_{b}^{2}(t)\right\} $.
	Let
	\[
		r_{\min}'':=\sqrt{\tau^{2}-\epsilon(2\tau-\epsilon)-\frac{(2\tau^{2}-r_{\min}^{2}-\epsilon(2\tau-\epsilon))^{2}}{4\tau^{2}}},
	\]
	then $f\left(r_{\min}''\right)\leq0$ implies that $\tilde{r}_{\sigma}\leq r_{\min}''$.
	Hence satisfying $f(r_{\min}'')\leq0$ implies that $f^{(N)}(\textrm{\v{C}ech}_{\mathbb{R}^{d}}(\mathcal{X},r))\subset\textrm{\v{C}ech}_{\mathbb{R}^{d}}\left(\mathcal{X},r_{\min}''\right)$.
	And Lemma~\ref{lem:homotopy_interleaving_ambientcech_restrictedcech}
	implies that 
	\[
	f^{(N)}(\textrm{\v{C}ech}_{\mathbb{R}^{d}}(\mathcal{X},r))\subset\textrm{\v{C}ech}_{\mathbb{R}^{d}}\left(\mathcal{X},r_{\min}''\right)\subset\textrm{\v{C}ech}_{\mathbb{X}}(\mathcal{X},r_{\min})\subset\textrm{\v{C}ech}_{\mathbb{X}}(\mathcal{X},r),
	\]
	i.e. $f^{(N)}:\textrm{\v{C}ech}_{\mathbb{R}^{d}}(\mathcal{X},r)\to\textrm{\v{C}ech}_{\mathbb{X}}(\mathcal{X},r)$.
	Then by construction, $\imath_{\textrm{\v{C}ech}_{\mathbb{X}}(\mathcal{X},r)\to\textrm{\v{C}ech}_{\mathbb{R}^{d}}(\mathcal{X},r)}\circ f^{(N)}$
	is homotopy equivalent to $id_{\textrm{\v{C}ech}_{\mathbb{R}^{d}}(\mathcal{X},r)}$.
	Hence by applying Lemma~\ref{lem:homotopy_simplexnerve}, $\mathbb{X}$
	and $\textrm{\v{C}ech}_{\mathbb{R}^{d}}(\mathcal{X},r)$ are homotopy
	equivalent.
\end{proof}

For Theorem~\ref{thm:homotopy_rips}, Claim~\ref{claim:homotopy_cech_rips}
and \ref{claim:homotopy_cech_rips_radius} imply that there exists
a map $\rho:\textrm{Rips}(\mathcal{X},r)\to\textrm{\v{C}ech}_{\mathbb{X}}(\mathcal{X},r)$
that $\imath_{\textrm{\v{C}ech}_{\mathbb{X}}(\mathcal{X},r)\to\textrm{Rips}(\mathcal{X},r)}\circ\rho:\textrm{Rips}(\mathcal{X},r)\to\textrm{Rips}(\mathcal{X},r)$
is homotopic to the identity $id_{\textrm{Rips}(\mathcal{X},r)}$. Then from Corollary
\ref{cor:homotopy_interleaving_rips_restrictedcech}, the following
diagram commutes:
\[
\xymatrix{ & \mathbb{X}\ar[dl]_{\phi}\\
	\textrm{\v{C}ech}_{\mathbb{X}}(\mathcal{X},r)\ar@/_/[dr]_{\imath_{\textrm{\v{C}ech}_{\mathbb{X}}(\mathcal{X},r)\to\textrm{Rips}(\mathcal{X},r)}}\ar[rr]^{\imath_{\textrm{\v{C}ech}_{\mathbb{X}}(\mathcal{X},r)\to\textrm{\v{C}ech}_{\mathbb{X}}(\mathcal{X},r''')}} &  & \textrm{\v{C}ech}_{\mathbb{X}}(\mathcal{X},r''')\ar[ul]_{\psi'}\\
	& \textrm{Rips}(\mathcal{X},r)\ar@/_/[ul]_{\rho}\ar[ur]_{\imath_{\textrm{Rips}(\mathcal{X},r)\to\textrm{\v{C}ech}_{\mathbb{X}}(\mathcal{X},r''')}}
}
.
\]
Then Lemma~\ref{lem:homotopy_simplexnerve} implies that $\mathbb{X}$
is homotopy equivalent to the Vietoris-Rips complex $\textrm{Rips}(\mathcal{X},r)$.
We restate Theorem~\ref{thm:homotopy_rips} and formally write its
proof below.

\textbf{Theorem~\ref{thm:homotopy_rips}.} \textit{Let $\mathbb{X}\subset\mathbb{R}^{d}$
	be a subset with reach $\tau>0$ and let $\mathcal{X}\subset\mathbb{R}^{d}$
	be a closed discrete set of points. Let $\{r_{x}>0:x\in\mathcal{X}\}$ be a set of
	radii indexed by $x\in\mathcal{X}$ with $r_{\min}:=\inf_{x\in\mathcal{X}}\{r_{x}\}>0$
	and $r_{\max}:=\sup_{x\in\mathcal{X}}\{r_{x}\}<\infty$, and let $\epsilon:=\sup\{d_{\mathbb{X}}(x):x\in\mathcal{X}\}$.
	Suppose $\mathbb{X}$ is covered by the union of balls centered at
	$x\in\mathcal{X}$ and radius $\delta$ as 
	\[
	\mathbb{X}\subset\bigcup_{x\in\mathcal{X}}\mathbb{B}_{\mathbb{R}}(x,\delta).
	\]
	Suppose that the maximum radius $r_{\max}$ is bounded as 
	\[
	r_{\max}\leq\sqrt{\frac{d+1}{2d}}\left(\tau-\epsilon\right).
	\]
	Also, suppose $\delta$ satisfies the following condition: 
	\begin{align*}
	& \sqrt{\tilde{r}_{b}^{2}(r_{\max})-(2\tau^{2}-\tilde{r}_{b}^{2}(r_{\max}))\left(\frac{\tau}{\sqrt{\tau^{2}-\tilde{r}_{\delta,b}^{2}(r_{\max})}}-1\right)}+2\delta\leq2r_{\min}, \\
	& \sqrt{\frac{d}{2(d+1)}}\left(\sqrt{\tilde{r}_{b}^{2}(r_{\min}'')-(2\tau^{2}-\tilde{r}_{b}^{2}(r_{\min}''))\left(\frac{\tau}{\sqrt{\tau^{2}-\tilde{r}_{\delta,b}^{2}(r_{\min}'')}}-1\right)}+2\delta\right)\leq r_{\min}'',
	\end{align*}
	where 
	\begin{align*}
	& r_{\min}'':=\sqrt{\tau^{2}-\epsilon(2\tau-\epsilon)-\frac{(2\tau^{2}-r_{\min}^{2}-\epsilon(2\tau-\epsilon))^{2}}{4\tau^{2}}},\\
	& \tilde{r}_{b}^{2}(t):=\frac{2\tau\left(t^{2}+\epsilon(2\tau-\epsilon)\right)}{\tau+\sqrt{\tau^{2}-\left(t^{2}+\epsilon(2\tau-\epsilon)\right)}},\qquad\tilde{r}_{\delta,b}^{2}(t):=\min\left\{ \delta^{2}+\epsilon(2\tau-\epsilon),\frac{1}{2}\tilde{r}_{b}^{2}(t)\right\} .
	\end{align*}
	Then $\mathbb{X}$ is homotopy equivalent to the Vietoris-Rips complex
	$\textrm{Rips}(\mathcal{X},r)$.}

\begin{proof}[Proof of Theorem~\ref{thm:homotopy_rips}]

	For a simplex $\sigma=[x_{1}\cdots x_{k}]\in\textrm{Rips}(\mathcal{X},r)$,
	let $r_{\sigma}:=Rad(\sigma)$, and choose $y_{\sigma}\in\mathcal{X}$
	according to Claim~\ref{claim:homotopy_cech_rips_radius}. Then as
	long as $r_{\sigma}<\tau-\epsilon_{\sigma}$, Claim~\ref{claim:homotopy_cech_rips_radius}
	(iii) asserts that $[x_{1}\cdots x_{k}y_{[x_{1}\cdots x_{k}]}]\in\textrm{Rips}(\mathcal{X},r)$
	holds. Now, define the homotopy map $F_{1}:[x_{1}\cdots x_{k}]\times[0,1]\to[x_{1}\cdots x_{k}y_{[x_{1}\cdots x_{k}]}]$
	as 
	\[
	F_{1}\left(\sum_{i=1}^{k}\lambda_{i}x_{i},t\right)=\sum_{i=1}^{k}(\lambda_{i}-\min\lambda_{i})x_{i}+k\min\lambda_{i}\left(ty_{\sigma}+(1-t)\frac{1}{k}\sum_{i=1}^{k}x_{i}\right).
	\]
	Then $F_{1}$ gives homotopy between $\imath_{\sigma\to\sigma*y_{\sigma}}$
	and $f_{1}:\sigma\to\sigma*y_{\sigma}$ defined as 
	\[
	f_{1}\left(\sum_{i=1}^{k}\lambda_{i}x_{i}\right)=\sum_{i=1}^{k}(\lambda_{i}-\min\lambda_{i})x_{i}+(k\min\lambda_{i})y_{\sigma},
	\]
	giving homotopy equivalence as 
	\[
	[x_{1}\cdots x_{k}]\simeq\sum_{i=1}^{k}[x_{1}\cdots\hat{x}_{i}\cdots x_{k}y_{[x_{1}\cdots x_{k}]}]\text{ in }\textrm{Rips}(\mathcal{X},r)
	\]
	And again, Claim~\ref{claim:homotopy_cech_rips_radius} (iii) asserts
	that $[x_{1}\cdots x_{k-1}y_{[x_{1}\cdots x_{k-1}]}y_{[x_{1}\cdots x_{k}]}]\in\textrm{Rips}(\mathcal{X},r)$
	holds. Now, the homotopy map $F_{2}:[x_{1}\cdots x_{k-1}y_{[x_{1}\cdots x_{k}]}]\times[0,1]\to[x_{1}\cdots x_{k-1}y_{[x_{1}\cdots x_{k-1}]}y_{[x_{1}\cdots x_{k}]}]$
	is defined as 
	\begin{align*}
	F_{2}\left(\sum_{i=1}^{k-1}\lambda_{i}x_{i}+\lambda_{k}y_{[x_{1}\cdots x_{k}]},t\right) & =\lambda_{k}y_{[x_{1}\cdots x_{k}]}+\sum_{i=1}^{k-1}(\lambda_{i}-\min\lambda_{i})x_{i}\\
	& \quad+(k-1)\min\lambda_{i}\left(ty_{[x_{1}\cdots x_{k-1}]}+(1-t)\frac{1}{k-1}\sum_{i=1}^{k-1}x_{i}\right).
	\end{align*}
	Then $F_{2}$ gives homotopy between $\imath_{[x_{1}\cdots x_{k-1}y_{[x_{1}\cdots x_{k}]}]\to[x_{1}\cdots x_{k-1}y_{[x_{1}\cdots x_{k-1}]}y_{[x_{1}\cdots x_{k}]}]}$
	and \\
	$f_{2}:[x_{1}\cdots x_{k-1}y_{[x_{1}\cdots x_{k}]}]\to[x_{1}\cdots x_{k-1}y_{[x_{1}\cdots x_{k-1}]}y_{[x_{1}\cdots x_{k}]}]$
	defined as 
	\[
	f_{2}\left(\sum_{i=1}^{k-1}\lambda_{i}x_{i}+\lambda_{k}y_{[x_{1}\cdots x_{k}]}\right)=\lambda_{k}y_{[x_{1}\cdots x_{k}]}+\sum_{i=1}^{k-1}(\lambda_{i}-\min\lambda_{i})x_{i}+(k-1)\min\lambda_{i}y_{[x_{1}\cdots x_{k-1}]},
	\]
	giving the homotopy equivalence as 
	\[
	[x_{1}\cdots x_{k-1}y_{[x_{1}\cdots x_{k}]}]\simeq\sum_{i=1}^{k-1}[x_{1}\cdots\hat{x}_{i}\cdots x_{k-1}y_{[x_{1}\cdots x_{k-1}]}y_{[x_{1}\cdots x_{k}]}]\text{ in }\textrm{Rips}(\mathcal{X},r)
	\]
	By repeating this and concatenating the homotopy maps, we have the
	homotopy map $F_{\sigma}:\sigma\times[0,1]\to\textrm{Rips}(\mathcal{X},r)$
	giving homotopy between $\imath_{\sigma\to\textrm{Rips}(\mathcal{X},r)}$
	and $f_{\sigma}:\sigma\to\textrm{Rips}(\mathcal{X},r)$ with $f_{\sigma}(\sigma)=\sum_{\omega\in S_{k}}[x_{\omega(1)}y_{[x_{\omega(1)}x_{\omega(2)}]}\ldots y_{[x_{1}\cdots x_{k}]}]$,
	i.e. $F_{\sigma}$ is giving homotopy equivalence between $[x_{1},\ldots,x_{k}]$
	and $\sum_{\omega\in S_{k}}[x_{\omega(1)}y_{[x_{\omega(1)}x_{\omega(2)}]}\ldots y_{[x_{1}\cdots x_{k}]}]$
	in $\textrm{Rips}(\mathcal{X},r)$.
	
	Now, we extend $f_{\sigma}$ and $F_{\sigma}$ to the entire Vietoris-Rips complex $\textrm{Rips}(\mathcal{X},r)$. Define $f:\textrm{Rips}(\mathcal{X},r)\to\textrm{Rips}(\mathcal{X},r)$
	and $F:\textrm{Rips}(\mathcal{X},r)\times[0,1]\to\textrm{Rips}(\mathcal{X},r)$
	as $f|_{\sigma}=f_{\sigma}$ and $F|_{\sigma\times[0,1]}=F_{\sigma}$
	for each $\sigma\in\textrm{Rips}(\mathcal{X},r)$. Then for $\sigma,\tau\in\textrm{Rips}(\mathcal{X},r)$,
	$F_{\sigma}$ and $F_{\tau}$ coincides on $\sigma\cap\tau\times[0,1]$,
	so $f$ and $F$ are well defined. Also, from $\mathcal{X}$ being
	closed and discrete, $\textrm{Rips}(\mathcal{X},r)$ is locally finite.
	Hence for any compact set $C\subset\textrm{Rips}(\mathcal{X},r)$,
	$C$ intersects with only finite number of simplices $\sigma_{1},\ldots,\sigma_{k}\in\textrm{Rips}(\mathcal{X},r)$,
	so $F^{-1}(C)=\bigcup_{i=1}^{k}F_{\sigma_{i}}^{-1}(C)$ is compact,
	and hence $F$ is continuous. So $F$ gives the homotopy between $id_{\textrm{Rips}(\mathcal{X},r)}$
	and $f$.
	
	Now, applying Claim~\ref{claim:homotopy_cech_rips_radius} (iii) gives
	that 
	\begin{align*}
	& {\rm Rad}([x_{\omega(1)}y_{[x_{\omega(1)}x_{\omega(2)}]}\cdots y_{[x_{1}\cdots x_{k}]}])\\
	& \leq\sqrt{\frac{d}{2(d+1)}}\frac{r_{\max}}{r_{\min}}\left(\sqrt{\tilde{r}_{b}^{2}-(2\tau^{2}-\tilde{r}_{b}^{2})\left(\frac{\tau}{\sqrt{\tau^{2}-\tilde{r}_{\delta,b}^{2}}}-1\right)}+2\delta\right),
	\end{align*}
	where 
	\begin{align*}
	\tilde{r}_{b}^{2} & :=\frac{2\tau\left(r_{\sigma}^{2}+\epsilon(2\tau-\epsilon)\right)}{\tau+\sqrt{\tau^{2}-\left(r_{\sigma}^{2}+\epsilon(2\tau-\epsilon)\right)}},\\
	\tilde{r}_{\delta,b}^{2} & :=\min\left\{ \delta^{2}+\epsilon(2\tau-\epsilon),\frac{1}{2}\tilde{r}_{b}^{2}\right\} .
	\end{align*}
	Hence by repeating this sufficiently many times (say $N$), we can
	guarantee $f^{(N)}(\textrm{Rips}(\mathcal{X},r))$ to be consisting
	of simplices with their radii (i.e. $Rad(\sigma)$) at most $\tilde{r}_{\sigma}$,
	where $\tilde{r}_{\sigma}$ is the solution of 
	\[
	f(t)=\sqrt{\frac{d}{2(d+1)}}\left(\sqrt{\tilde{r}_{b}^{2}(t)-(2\tau^{2}-\tilde{r}_{b}^{2}(t))\left(\frac{\tau}{\sqrt{\tau^{2}-\tilde{r}_{\delta,b}^{2}(t)}}-1\right)}+2\delta\right)-t=0,
	\]
	with $\tilde{r}_{b}^{2}(t)=\frac{2\tau\left(t^{2}+\epsilon(2\tau-\epsilon)\right)}{\tau+\sqrt{\tau^{2}-\left(t^{2}+\epsilon(2\tau-\epsilon)\right)}}$
	and $\tilde{r}_{\delta,b}^{2}(t)=\min\left\{ \delta^{2}+\epsilon(2\tau-\epsilon),\frac{1}{2}\tilde{r}_{b}^{2}(t)\right\} .$
	Let \[
		r_{\min}'':=\sqrt{\tau^{2}-\epsilon(2\tau-\epsilon)-\frac{(2\tau^{2}-r_{\min}^{2}-\epsilon(2\tau-\epsilon))^{2}}{4\tau^{2}}},
	\]
	then $f\left(r_{\min}''\right)\leq0$ implies that $\tilde{r}_{\sigma}\leq r_{\min}''$.
	Hence satisfying $f\left(r_{\min}''\right)\leq0$ implies that $f^{(N)}(\textrm{Rips}(\mathcal{X},r))\subset\textrm{\v{C}ech}_{\mathbb{R}^{d}}\left(\mathcal{X},r_{\min}''\right)$.
	And Lemma~\ref{lem:homotopy_interleaving_ambientcech_restrictedcech}
	implies that 
	\[
	f^{(N)}(\textrm{Rips}(\mathcal{X},r))\subset\textrm{\v{C}ech}_{\mathbb{R}^{d}}\left(\mathcal{X},r_{\min}''\right)\subset\textrm{\v{C}ech}_{\mathbb{X}}(\mathcal{X},r_{\min})\subset\textrm{\v{C}ech}_{\mathbb{X}}(\mathcal{X},r),
	\]
	i.e. $f^{(N)}:\textrm{Rips}(\mathcal{X},r)\to\textrm{\v{C}ech}_{\mathbb{X}}(\mathcal{X},r)$.
	Then by construction, $\imath_{\textrm{\v{C}ech}_{\mathbb{X}}(\mathcal{X},r)\to\textrm{Rips}(\mathcal{X},r)}\circ f^{(N)}$
	is homotopy equivalent to $id_{\textrm{Rips}(\mathcal{X},r)}$. Hence
	by applying Lemma~\ref{lem:homotopy_simplexnerve}, $\mathbb{X}$
	and $\textrm{Rips}(\mathcal{X},r)$ are homotopy equivalent.

\end{proof}

Now, Corollary~\ref{cor:homotopy_cech_rips_mureach} is from the combination
of Corollary~\ref{cor:defretract_doubleoffset} and Theorem~\ref{thm:homotopy_cech}
and \ref{thm:homotopy_rips}. We restate Corollary~\ref{cor:homotopy_cech_rips_mureach}
and formally write its proof below.

\textbf{Corollary~\ref{cor:homotopy_cech_rips_mureach}.} \textit{
	Let $\mathbb{X}\subset\mathbb{R}^{d}$ be a subset with positive $\mu$-reach
	$\tau^{\mu}>0$ and let $\mathcal{X}\subset\mathbb{R}^{d}$ be a set
	of points. Let $\{r_{x}>0:x\in\mathcal{X}\}$ be a set of radii indexed
	by $x\in\mathcal{X}$ with $r_{\min}:=\min_{x\in\mathcal{X}}\{r_{x}\}$
	and $r_{\max}:=\max_{x\in\mathcal{X}}\{r_{x}\}$. Let $s,t,\epsilon\geq0$
	with $\frac{t}{\mu} < s < \tau^{\mu}$, and let $\mathbb{Y}:=(((\mathbb{X}^{s})^{\complement})^{t})^{\complement}$
	be the double offset, with $d_{\mathbb{Y}}(x)\leq\epsilon$ for all
	$x\in\mathcal{X}$. Suppose $\mathbb{Y}$ is covered by the union
	of balls centered at $x\in\mathcal{X}$ and radius $\delta$ as 
	\[
	\mathbb{Y}\subset\bigcup_{x\in\mathcal{X}}\mathbb{B}_{\mathbb{R}}(x,\delta).
	\]
	\begin{enumerate}
		\item[(i)] Suppose $r_{\max}\leq t-\epsilon$, and $\delta$ satisfies the following condition: 
		\begin{align*}
		& \delta+\sqrt{r_{\max}^{2}-\tilde{l}^{2}+\epsilon(2t-\epsilon)-((t-\epsilon)^{2}-r_{\max}^{2}+\tilde{l}^{2}+(t-\epsilon_{\tilde{l}})^{2})\left(\frac{t}{\sqrt{t^{2}-\tilde{r}_{\delta,c}}}-1\right)}\\
		& \leq r_{\min},\\
		& \sqrt{\frac{d}{2(d+1)}}\frac{r_{\max}}{r_{\min}}\left(\sqrt{\tilde{r}_{b}^{2}-(2t^{2}-\tilde{r}_{b}^{2})\left(\frac{t}{\sqrt{t^{2}-\tilde{r}_{\delta,b}^{2}}}-1\right)}+2\delta\right)\leq r_{\min}'',
		\end{align*}
		where 
		\begin{align*}
		& \tilde{l}:=\frac{1}{2}\left(r_{\min}-t+\sqrt{(t-\epsilon)^{2}-r_{\max}^{2}}-\delta\right),\qquad\epsilon_{\tilde{l}}:=t-\sqrt{(t-\epsilon)^{2}-r_{\max}^{2}}+\tilde{l},\\
		& \tilde{r}_{\delta,c}^{2}:=\min\left\{ \delta^{2}+\epsilon(2t-\epsilon),\frac{1}{2}(r_{\max}^{2}-\tilde{l}^{2}+\epsilon(2t-\epsilon)+\epsilon_{\tilde{l}}(2t-\epsilon_{\tilde{l}}))\right\} ,\\
		& r_{\min}'':=\sqrt{t^{2}-\epsilon(2t-\epsilon)-\frac{(2t^{2}-r_{\min}^{2}-\epsilon(2t-\epsilon))^{2}}{4t^{2}}},\\
		& \tilde{r}_{b}^{2}:=\frac{2t\left((r_{\min}'')^{2}+\epsilon(2t-\epsilon)\right)}{t+\sqrt{t^{2}-\left((r_{\min}'')^{2}+\epsilon(2t-\epsilon)\right)}},\qquad\tilde{r}_{\delta,b}^{2}:=\min\left\{ \delta^{2}+\epsilon(2t-\epsilon),\frac{1}{2}\tilde{r}_{b}^{2}\right\} .
		\end{align*}
		Then $\mathbb{X}$ is homotopy equivalent to the ambient \v{C}ech
		complex $\textrm{\v{C}ech}_{\mathbb{R}^{d}}(\mathcal{X},r)$. 
		\item[(ii)] Suppose $r_{\max}\leq\sqrt{\frac{d+1}{2d}}\left(t-\epsilon\right)$, and $\delta$ satisfies the following condition: 
		\begin{align*}
		& \sqrt{\tilde{r}_{b}^{2}(r_{\max})-(2t^{2}-\tilde{r}_{b}^{2}(r_{\max}))\left(\frac{t}{\sqrt{t^{2}-\tilde{r}_{\delta,b}^{2}(r_{\max})}}-1\right)}+2\delta\leq2r_{\min},\\
		& \sqrt{\frac{d}{2(d+1)}}\left(\sqrt{\tilde{r}_{b}^{2}(r_{\min}'')-(2t^{2}-\tilde{r}_{b}^{2}(r_{\min}''))\left(\frac{t}{\sqrt{t^{2}-\tilde{r}_{\delta,b}^{2}(r_{\min}'')}}-1\right)}+2\delta\right)\leq r_{\min}'',
		\end{align*}
		where 
		\begin{align*}
		& r_{\min}'':=\sqrt{t^{2}-\epsilon(2t-\epsilon)-\frac{(2t^{2}-r_{\min}^{2}-\epsilon(2t-\epsilon))^{2}}{4t^{2}}},\\
		& \tilde{r}_{b}^{2}(t):=\frac{2t\left(t^{2}+\epsilon(2t-\epsilon)\right)}{t+\sqrt{t^{2}-\left(t^{2}+\epsilon(2t-\epsilon)\right)}},\qquad\tilde{r}_{\delta,b}^{2}(t):=\min\left\{ \delta^{2}+\epsilon(2t-\epsilon),\frac{1}{2}\tilde{r}_{b}^{2}(t)\right\} .
		\end{align*}
		Then $\mathbb{X}$ is homotopy equivalent to the Vietoris-Rips complex
		$\textrm{Rips}(\mathcal{X},r)$. 
	\end{enumerate}
}

\begin{proof}[Proof of Corollary~\ref{cor:homotopy_cech_rips_mureach}]

Consider the double offset $\mathbb{Y}:=(((\mathbb{X}^{s})^{\complement})^{t})^{\complement}$.
Since $\mathbb{X}$ has a positive $\mu$-reach $\tau^{\mu}$, $s\leq\tau^{\mu}$
and $t\leq\mu s$, Corollary~\ref{cor:defretract_doubleoffset}
implies that the reach of $\mathbb{Y}$ is bounded by $\tau_{\mathbb{Y}}\geq t$.

(i)

For the ambient \v{C}ech complex case, $r_{x}\leq t-\epsilon\leq\tau_{\mathbb{Y}}-\epsilon$ holds,
so Theorem~\ref{thm:homotopy_cech} applies under the appropriate
condition of $\delta$.

(ii)

For the Vietoris-Rips complex case, $r_{x}\leq\sqrt{\frac{d+1}{2d}}(t-\epsilon)\leq\sqrt{\frac{d+1}{2d}}(\tau_{\mathbb{Y}}-\epsilon)$
as well, so Theorem~\ref{thm:homotopy_rips} applies under the appropriate
condition of $\delta$.

\end{proof}

To prove the covering lemma~\ref{lem:density_covering_probability}, we first show the following upper bound on the covering number holds.
\begin{claim} \label{claim:density_covering_support} Suppose the distribution
	$P$ satisfies the $(a, b)$-condition in \eqref{eq::ab-condition}. Then for all $\epsilon < \epsilon_{0}$, the covering number $\mathcal{N}(\mathbb{X},\left\Vert \cdot\right\Vert ,2\epsilon)$
	is bounded as 
	\[
	\mathcal{N}(\mathbb{X},\left\Vert \cdot\right\Vert ,2\epsilon)\leq\frac{1}{a}\epsilon^{-b}.
	\]
\end{claim}

\begin{proof}[Proof of Claim~\ref{claim:density_covering_support}]
	
	Let $x_{1},\ldots,x_{M}$ be a maximal $2\epsilon$-packing of $\mathbb{X}$, with $M=\mathcal{M}(\mathbb{X},\left\Vert \cdot\right\Vert ,2\epsilon)$. Then $\mathbb{B}_{\mathbb{R}^{d}}(x_{i},\epsilon)$ and $\mathbb{B}_{\mathbb{R}^{d}}(x_{j},\epsilon)$ do not intersect for any $i,j$, and hence 
	\begin{equation}
	\sum_{i=1}^{M}P(\mathbb{B}_{\mathbb{R}^{d}}(x_{i},\epsilon))\leq P(\mathbb{R}^{d})=1.\label{eq:density_probability_covering}
	\end{equation}
	Then for all $\epsilon<\epsilon_{0}$, the $(a,b)$-condition
	implies 
	\[
	P(\mathbb{B}_{\mathbb{R}^{d}}(x_{i},\epsilon))\geq a\epsilon^{b},
	\]
	hence applying this to \eqref{eq:density_probability_covering} gives that
	\[
	\mathcal{M}(\mathbb{X},\left\Vert \cdot\right\Vert ,2\epsilon)\leq\frac{1}{a}\epsilon^{-b}.
	\]
	Then from the relationship between the covering number and the packing number,
	\[
	\mathcal{N}(\mathbb{X},\left\Vert \cdot\right\Vert ,2\epsilon)\leq\mathcal{M}(\mathbb{X},\left\Vert \cdot\right\Vert ,2\epsilon)\leq\frac{1}{a}\epsilon^{-b}.
	\]
	
\end{proof}

Now, we restate Lemma~\ref{lem:density_covering_probability} and
formally write its proof below.

\textbf{Lemma~\ref{lem:density_covering_probability}.} \textit{Let
	$\{X_{1},\dots,X_{n}\}$ be an i.i.d. sample from the distribution $P$ and let $\{r_{n}=(r_{n,1},\dots,r_{n,n})\}_{n\in\mathbb{N}}$
	be a triangular array of positive numbers such that, for each $n$,
	\[
	2\left(\frac{\log n}{an}\right)^{1/b}\leq\min_{i}r_{n,i}\leq2\epsilon_{0}.
	\]
	Then, the probability that the sample is an $r_{n}$-covering of $\mathbb{X}$
	is bounded as 
	\[
	P\left(\mathbb{X}\subset\bigcup_{i=1}^{n}\mathbb{B}_{\mathbb{R}^{d}}(X_{i},r_{n,i})\right)\geq1-\frac{1}{2^{b}\log n}.
	\]
}

\begin{proof}[Proof of Lemma~\ref{lem:density_covering_probability}]	
	
	Now, to prove Lemma~\ref{lem:density_covering_probability}, set $\epsilon:=\frac{1}{4}\min_{i}r_{n,i}$. Under the $(a,b)$-condition, the previous claim~\ref{claim:density_covering_support} implies that there exists $x_{1},\ldots,x_{N}$
	with $N\leq a^{-1}\epsilon^{-b}$ satisfying 
	\[
	\mathbb{X}\subset\bigcup_{j=1}^{N}\mathbb{B}_{\mathbb{R}^{d}}(x_{j},2\epsilon).
	\]
	Let $E$ be the event that all $\mathbb{B}_{\mathbb{R}^{d}}(x_{j},2\epsilon)$
	have intersections with $\{X_{1},\ldots,X_{n}\}$, that is, for each $1\leq j\leq N$,
	there exists $1\leq i\leq n$ with $X_{i}\in \mathbb{B}_{\mathbb{R}^{d}}(x_{j},2\epsilon)$.
	Then note that $4\epsilon=\min_{i}r_{n,i}\leq r_{n,i}$, and hence we have the following relations between balls:
	\[
	\mathbb{B}_{\mathbb{R}^{d}}(x_{j},2\epsilon)\subset \mathbb{B}_{\mathbb{R}^{d}}(X_{i},4\epsilon)\subset \mathbb{B}_{\mathbb{R}^{d}}(X_{i},r_{n,i}).
	\]
	Therefore, under the event $E$, we have
	\[
	\mathbb{X}\subset\bigcup_{j=1}^{N}\mathbb{B}_{\mathbb{R}^{d}}(x_{j},2\epsilon)\subset\bigcup_{i=1}^{n}\mathbb{B}_{\mathbb{R}^{d}}(X_{i},r_{n,i}),
	\]
which implies 
	\begin{equation}
		\mathbb{P}\left(\mathbb{X}\subset\bigcup_{i=1}^{n}\mathbb{B}_{\mathbb{R}^{d}}(X_{i},r_{n,i})\right)\geq\mathbb{P}(E).\label{eq:density_covering_events}
	\end{equation}
	Now, $P(E)$ can be lower bounded as 
	\begin{align*}
		\mathbb{P}(E) & =\mathbb{P}\left(\bigcap_{j=1}^{N}\bigcup_{j=1}^{n}\{X_{i}\in \mathbb{B}_{\mathbb{R}^{d}}(x_{j},2\epsilon)\}\right)\\
		& =1-\mathbb{P}\left(\bigcup_{j=1}^{N}\bigcap_{i=1}^{n}\{X_{i}\notin \mathbb{B}_{\mathbb{R}^{d}}(x_{j},2\epsilon)\}\right)\\
		& \geq1-\sum_{j=1}^{N}\mathbb{P}\left(\bigcap_{i=1}^{n}\{X_{i}\notin \mathbb{B}_{\mathbb{R}^{d}}(x_{j},2\epsilon)\}\right)\\
		& =1-\sum_{j=1}^{N}\prod_{i=1}^{n}(1-P(\mathbb{B}_{\mathbb{R}^{d}}(x_{j},2\epsilon))\\
		& \geq1-\sum_{j=1}^{N}\exp\left(-\sum_{i=1}^{n}P(\mathbb{B}_{\mathbb{R}^{d}}(x_{j},2\epsilon))\right),
	\end{align*}
	where the last line is from that $1-t\leq\exp(-t)$ for all $t\in\mathbb{R}$.
	Now, from the covering number bound $N\leq a^{-1}\epsilon^{-b}$ with the condition, $2\epsilon = \frac{1}{2}\min_i r_{n,i} \geq \left(\frac{\log n}{an}\right)^{1/b}$, we can further lower bound $P(E)$ as following:
	\begin{align*}
		P(E) & \geq1-N\exp\left(-an(2\epsilon)^{b}\right) \\
		& \geq1-a^{-1}\epsilon^{-b} \exp\left(-an(2\epsilon)^{b}\right) \\
	    & 1- \frac{n}{2^b\log n}\exp\left(-\log n\right)\\
		& =1-\frac{1}{2^b\log n},
	\end{align*}
    which implies 
	\begin{align*}
	\mathbb{P}\left(\mathbb{X}\subset\bigcup_{i=1}^{n}\mathbb{B}_{\mathbb{R}^{d}}(X_{i},r_{n,i})\right) \geq 1-\frac{1}{2^b\log n},
	\end{align*}
	as desired.
\end{proof}

\end{document}